\documentclass[11pt,a4paper]{article}
\usepackage[utf8]{inputenc}

\usepackage[top=0.99in, bottom=0.99in, left=0.98in, right=0.98in]{geometry}
\usepackage{amsmath,amsthm,amssymb,amscd,dsfont,bm,epic,accents}
\usepackage{graphicx}
\usepackage[all]{xy}
\usepackage{tikz}
\usetikzlibrary{intersections}

\usepackage{multicol}
\usepackage{mathrsfs}
\usepackage{enumerate}
\usepackage{caption}
\usepackage{subcaption}
\usepackage{graphicx}
\usepackage{color}
\usepackage{tikz}
\usetikzlibrary{shapes.multipart, positioning}

\usepackage{tkz-euclide}
\usetikzlibrary{quotes}
\tikzset{every edge quotes/.style =
          { fill = white,
            sloped,
            execute at begin node = $,
            execute at end node   = $  }}

\usepackage[driverfallback=hypertex]{hyperref}
\usepackage{nameref,zref-xr}                   

\usepackage{cancel}
\usepackage[normalem]{ulem}
\usepackage{enumitem}

\newtheorem{dummy}{}[section]
\newtheorem{thm}[dummy]{Theorem}
\newtheorem{prop}[dummy]{Proposition}
\newtheorem{pr}[dummy]{Proposition}
\newtheorem{lem}[dummy]{Lemma}

\newtheorem{cor}[dummy]{Corollary}

\theoremstyle{definition}
\newtheorem{definition}[dummy]{Definition}
\newtheorem{dfn}[dummy]{Definition}

\theoremstyle{remark}
\newtheorem{rmk}[dummy]{Remark}
\newtheorem{ex}[dummy]{Example}
\newtheorem{obs}[dummy]{Observation}

\newcommand{\alt}{{\text{alt}}}

\newcommand{\tw}{\text{tw}}
\newcommand{\twE}{\text{tw}}

\newcommand{\oQMb}{{\overline{\mathcal{QM}}}}

\newcommand{\oJMb}{{\overline{\mathcal{JM}}}}

\numberwithin{equation}{section}

\usepackage{booktabs,multirow,tabularx}

\DeclareMathOperator{\Id}{Id}

\def\M{{{\mathcal M}}}
\def\Mbar{{\overline{\mathcal M}}}
\def\Mstar{{{\mathcal M^*}}}
\def\Mbarstar{{\overline{\mathcal M^*}}}

\def\GPI{{\text{GPI}}}
\def\PI{{\text{PI}}}
\def\h{{\mathfrak h}}

\usepackage{scalerel}

\DeclareMathOperator*{\bboxplus}{\scalerel*{\boxplus}{\sum}}

\title{The point insertion technique and open $r$-spin theories I: moduli and orientation}
\date{}
\author{Ran J. Tessler \and Yizhen Zhao}

\AtEndDocument{\bigskip{\footnotesize%
    \textsc{R. J. Tessler:}\par
  \textsc{Department of Mathematics, Weizmann Institute of Science, POB 26, Rehovot 7610001, Israel} \par  
  \textit{E-mail address}: \texttt{ran.tessler@weizmann.ac.il} \par
  \addvspace{\medskipamount}
  \textsc{Y. Zhao:}\par
  \textsc{Institut de Math\'ematiques de Jussieu – Paris Rive Gauche, 75005 Paris, France} \par  
  \textit{E-mail address}: \texttt{yizhen.zhao@imj-prg.fr}
}}

\begin{document}
\maketitle

\begin{abstract}
The papers \cite{BCT1,BCT2,BCT3,GKT} constructed an intersection theory on the moduli space of $r$-spin disks, and proved it satisfies mirror symmetry and relations with integrable hierarchies. That theory considered only disks with a single type of boundary state.
In this work, we initiate the study of more general $r$-spin surfaces:
we define the notion of graded $r$-spin surfaces with multiple internal and boundary states, and their moduli spaces. 

In $g=0,$ the disk case, we also define the associated open Witten bundle, and prove that the Witten bundle is canonically oriented relative to the moduli space.
Moreover, we describe a method for gluing several moduli spaces along certain boundaries, show that gluing lifts to the Witten bundle and relative cotangent line bundles, and that the result is again canonically relatively oriented.

{We then turn to $g=1,$ the cylinder case. In this case there are foundational problems in constructing the theory, whose origin is the fact that the Witten ``bundle'' ceases to be an orbifold vector bundle. We overcome this by removing the strata in which fibres are not of expected dimension, thus obtaining an orbibundle over the complement. We then extend the gluing method to $g=1$, and prove that also in $g=1$ the Witten bundle has canonical relative orientation.}

In the sequel \cite{TZ2}, we construct, based on the work of this paper, a family of $\lfloor r/2\rfloor$ intersection theories indexed by $\h\in\{0,\ldots,\lfloor r/2\rfloor-1\},$ where the $\h$-th one has $\h+1$ possible boundary states, and calculate their intersection numbers. The $\h=0$ theory is equivalent to the one constructed in \cite{BCT1,BCT2}.

{In the sequel \cite{TZ3} we rely on this construction, restricted to the $\h=0$ case, to construct an intersection theory on the moduli space of $r$-spin cylinders, and prove that its potential, after a change of variables, yields the $g=1$ part of the $r$th Gelfand--Dikii wave function, thus prove the $g=1$ part of the main conjecture of \cite{BCT3}.}

\end{abstract}

\section{Introduction}
A \emph{(smooth) $r$-spin curve} is a smooth marked curve $(C;z_1, \ldots, z_n)$ endowed with an \emph{$r$-spin structure}, which is a line bundle $S\to C,$ together with an identification
\[S^{\otimes r} \cong \omega_{C}\left(-\sum_{i=1}^n a_i[z_i]\right),\]
where $a_i \in \{0,1,\ldots, r-1\}$ are called the twists.  The moduli space $\mathcal{M}_{g,\{a_1, \ldots, a_n\}}^{1/r}$ of smooth $r$-spin curves has a natural compactification, the \emph{moduli space of stable $r$-spin curve} $\Mbar_{g,\{a_1, \ldots, a_n\}}^{1/r}.$ 
This space admits a virtual fundamental class $c_W$, the \textit{Witten's class}. 
Let $\pi: \mathcal{C} \to \Mbar_{0,\{a_1, \ldots, a_n\}}^{1/r}$ be the universal curve and write $\mathcal{S}$ for the universal $r$-spin structure. In genus zero
$(R^1\pi_*\mathcal{S})^{\vee}$ is an (orbifold) vector bundle, and the Witten class is just its Euler class,
\[c_W= e((R^1\pi_*\mathcal{S})^{\vee}).\] The definition of $c_W$ in higher genus is more involved, see \cite{PV,ChiodoWitten,Moc06,FJR,CLL} for various constructions.

Witten \cite{Witten93} defined the \textit{(closed) $r$-spin intersection numbers} by
\[\left<\tau^{a_1}_{d_1}\cdots\tau^{a_n}_{d_n}\right>^{\frac{1}{r},c}_g:=r^{1-g}\int_{\Mbar^{1/r}_{g,\{a_1, \ldots, a_n\}}} \hspace{-1cm} c_W \cap \psi_1^{d_1} \cdots \psi_n^{d_n},
\]
where $\psi_i=c_1(\mathbb{L}_i), \in H^2(\Mbar^{1/r}_{g,\{a_1, \ldots, a_n\}}),~i=1,\ldots, n,$ are the first Chern classes of the relative cotangent lines at the markings. This theory is an important example of a \emph{cohomological field theory} \cite{KontsManin}, motivated the definition of Fan--Jarvis--Ruan--Witten (FJRW) \cite{FJR} of quantum singularities, and led to a proof of Pixton's conjecture on the tautological rings of ~$\Mbar_{g,n}$~\cite{PPZ}. It was also conjectured by Witten to be governed by the Gelfand-Dikii hierarchy \cite{Witten93}, a conjecture proven by Faber, Shadrin, and Zvonkine in \cite{FSZ10}.

The study of similar intersection theories on moduli spaces of surfaces with boundaries was initiated in \cite{PST14}. In \cite{BCT1,BCT2,BCT3} Buryak, Clader and the first named author constructed a disk analogue of the $r$-spin intersection theory, proved that it is governed by the Gelfand-Dikii wave function, and made an all-genus conjecture. 
This theory allowed the twists at internal markings range in $\{0,\ldots,r-1\}$ but the boundary markings were only allowed an $r-2$ twist. One might hope that more general intersection theories exist, which allow defining intersection numbers with more types of boundary twists. In this paper and its sequel \cite{TZ2} we construct such theories.
\subsection{The content of this paper}

Our first objects of study are \emph{graded $r$-spin surfaces}, which are, roughly speaking, the following objects. Let $C$ be an orbifold curve equipped with a conjugation $\phi: C \to C$ under which the coarse underlying curve $|C|$ is realized as the union of a surface with boundary $\Sigma$ and its conjugate $\overline{\Sigma},$ glued along their common boundary:
\[|C| = \Sigma \cup_{{\partial}\Sigma} \overline{\Sigma}.\]
Let $x_1, \ldots, x_k \in {\partial}\Sigma$ be a collection of (different) boundary marked points, let $z_1, \ldots, z_l \in \Sigma \setminus {\partial}\Sigma$ be a collection of (different) internal marked points, and let $\overline{z}_i:= \phi(z_i) \in \overline{\Sigma}$ be their conjugates.  A \emph{graded $r$-spin structure} with \emph{boundary twists} $b_1,\ldots,b_k$ and \emph{internal twists} $a_1, \ldots, a_l$ is an orbifold line bundle $S$ on $C$ together with an isomorphism
\[|S|^{\otimes r} \cong \omega_{|C|} \otimes {\mathcal{O}}\left(-\sum_{i=1}^l a_i[z_i] - \sum_{i=1}^l a_i[\overline{z}_i] - \sum_{j=1}^k b_j[x_j]\right)\]
on $|C|$, an involution $\widetilde{\phi}: S \to S$ lifting $\phi$, and a \emph{grading}, a certain orientation of $\left(S|_{{\partial}\Sigma \setminus \{x_j\}}\right)^{\widetilde\phi}$. 

We prove below that there exists a moduli space $\Mbar_{g,\{b_1,\ldots,b_k\},\{a_1, \ldots, a_l\}}^{1/r}$ of graded $r$-spin surfaces with boundary twists $b_1,\ldots,b_k$ and internal twists $a_1, \ldots, a_l$. We show that this moduli space is a smooth, effective, compact, orientable orbifold with corners.

This moduli space is also associated with several natural orbifold vector bundles.
First, there are the relative cotangent line bundles $\mathbb{L}_1 \ldots, \mathbb{L}_l$ at the internal marked points. We then restrict to $g=0,1$ (disk and cylinder case), and describe the \emph{open Witten bundle}, the real vector bundle, defined as
\[
\mathcal{W}: = (R^0\pi_*(\mathcal{S}^{\vee} \otimes \omega_{\pi}))_+,
\]
where '$+$'  denote the spaces of $\widetilde{\phi}$-invariant sections. {In the cylinder case this object is a vector bundle only after restricting to the complement of a subspace, which we fully characterize.}
We study the behaviour of these objects under restrictions to boundary strata, and, in particular, under a certain identification of boundary strata induced from a key operation we consider, the \emph{point insertion operation} defined in Section \ref{sec:point_ins}.
The point insertion provides a natural identification between certain boundary strata of one moduli space, with boundary strata of other moduli spaces, in which some special boundary points are replaced by internal points. 

For every $\h\in\{0,1,\ldots,\lfloor r/2\rfloor-1\}$, we define a new moduli space $\widetilde{\mathcal M}^{\frac{1}{r},\h}_{g,\{b_1,\ldots,b_k\},\{a_1,\ldots,a_l\}},$ where $b_1,\ldots,b_k\in\{r-2,r-4,\ldots,r-2-2\h\}$ and $a_1,\ldots,a_l\in\{0,\ldots,r-1\},$ which is obtained from $\overline{\mathcal M}^{\frac{1}{r},\h}_{g,\{b_1,\ldots,b_k\},\{a_1,\ldots,a_l\}}$ by gluing different moduli spaces along certain codimension-$1$ boundaries, in a way dictated by the point insertion operation.   We prove that ${\mathcal{W}}$ and each $\mathbb{L}_i,~i=1,\ldots, l$ also glue, giving rise to bundles defined over $\widetilde{\mathcal M}^{\frac{1}{r},\h}_{g,\{b_1,\ldots,b_k\},\{a_1,\ldots,a_l\}}$ when $g=0,$ or over an explicit subspace of it  when $g=1$.  Our main result is that all bundles
$$(\widetilde{{\mathcal{W}}})^{2d+1}\oplus\bigoplus_{i=1}^l\mathbb{L}_i^{\oplus d_i}
$$
are \emph{canonically oriented relative to }$\widetilde{\mathcal M}^{\frac{1}{r},\h}_{0,\{b_1,\ldots,b_k\},\{a_1,\ldots,a_l\}}.$ See Theorem \ref{thm:main_or} for a precise statement.

\subsection{Sequels}
In the sequel \cite{TZ2}, for each $\h\in\{0,1,\ldots,\lfloor r/2\rfloor-1\}$, we construct an intersection theory on the moduli space $\widetilde{\mathcal M}^{\frac{1}{r},\h}_{0,\{b_1,\ldots,b_k\},\{a_1,\ldots,a_l\}}.$ More precisely, we find canonical families of boundary conditions for the open Witten bundle and relative cotangent line bundles, and show that they give rise to intersection numbers, which are independent of all choices. 
We calculate all intersection numbers using a topological recursion relation, prove that the $\h=0$ theory is equivalent to the one constructed in \cite{BCT1,BCT2} and describe generalizations to other open intersection theories, which include the Fermat quintic open FJRW theory.

{In the sequel \cite{TZ3} we study the $g=1,\h=0$ case, and fully construct an intersection theory for it. We prove that its potential, after a change of variables, satisfies the equations of the wave function of the $r$th Gelfand--Dikii hierarchy.}

\subsection{Plan of the paper and comparison to existing literature}
Section \ref{sec review} reviews the notion of graded $r$-spin surfaces, their properties and moduli space, and in $g=0,1$ also the open Witten bundle. In Section \ref{sec:orientation} we study the properties of the orientations of the moduli spaces in Witten bundles. The idea of point insertion is introduced in Section \ref{sec:point_ins}, as well as the glued moduli space, which is based on it.

In \cite{BCT1} Buryak, Clader and the first named author considered graded $r$-spin disks with only $r-2$ boundary twists. They showed that the moduli space is a smooth, compact orientable orbifold with corners, and that the Witten bundle is canonically oriented relative to it. Our results generalize their results to more general objects and moduli spaces. Our method of constructing the orientation is new and more direct.
The point insertion idea did not appear in the literature before, to the best of our knowledge.  
A simpler version of it, in the spinless case, is described in an unpublished text by Jake Solomon and the first named author \cite{ST_unpublished}.
The key new technical result is the study of the behaviour of the orientations with respect to the identification of boundaries dictated by the point insertions. The outcome of this study is that the bundles and space obtained by gluing moduli spaces along the identified boundaries are relatively canonically oriented. This stronger orientability result is much stronger than the one obtained in \cite{BCT1}, even in the case of only $r-2$ boundary twists.

\subsection{Acknowledgements}
The authors would like to thank A. Buryak, M.~Gross, T.~Kelly, K. Hori, J.~Solomon and E.~Witten for interesting discussions related to this work.  
R.T. (incumbent of the Lillian and George Lyttle Career Development Chair) and Y. Zhao were supported by the ISF grant No. 335/19 and 1729/23.

\section{Review of graded $r$-spin surfaces, their moduli and bundles}
\label{sec review}

In this section, following \cite{BCT1}, we review the definition of graded $r$-spin surfaces, their moduli space, and the relevant bundles.  More details and proofs can be found in \cite{BCT1}.

\subsection{Graded $r$-spin surfaces}

The basic objects in this work are marked Riemann surfaces with boundary; we view them as arising from closed curves with an involution.  More preciously, a \textit{nodal marked surface} is defined as a tuple
$$(C, \phi, \Sigma, \{z_i\}_{i \in I}, \{x_j\}_{j \in B}, m^I, m^B),$$
in which
\begin{itemize}
\item $C$ is a nodal orbifold Riemann surface, which may be composed of disconnected components, and it has isotropy only at special points;
\item $\phi: C \to C$ is an anti-holomorphic involution which, from a topological perspective, realizes the coarse underlying Riemann surface $|C|$  as the union of two Riemann surfaces, $\Sigma$ and $\overline{\Sigma}=\phi(\Sigma),$ glued along the common subset $\text{Fix}(|\phi|)$;
\item $\{z_i\}_{i \in I}\subset  C$ consists of distinct \textit{internal marked points} whose images in $|C|$ lie in $\Sigma\setminus\text{Fix}(|\phi|)$, and they have \textit{conjugate marked points} $\overline{z_i}:= \phi(z_i)$;
\item $\{x_j\}_{j \in B} \subset \text{Fix}(\phi)$ consists of distinct \textit{boundary marked points} whose images in $|C|$ lie in ${\partial} \Sigma$;
\item $m^I$ (respectively $m^B$) is a marking of $I$ and (respectively $B$), \textit{i.e.} an one-to-one equivalence between  $I$ (respectively $m^B$) and $\{1,2,\dots,\lvert I \rvert\}$ (respectively $\{1,2,\dots,\lvert B \rvert\}$).
\end{itemize}

A node of a nodal marked surface can be internal, boundary or contracted boundary, both internal and boundary nodes can be separating or non-separating, so there are five types of nodes, as illustrated in Figure~\ref{fig node type} by shading $\Sigma \subseteq |C|$ in each case.  Note that ${\partial}\Sigma\subset\text{Fix}(|\phi|)$ is a union of circles, and $\text{Fix}(|\phi|)\setminus{\partial}\Sigma$ is the union of the contracted boundaries.

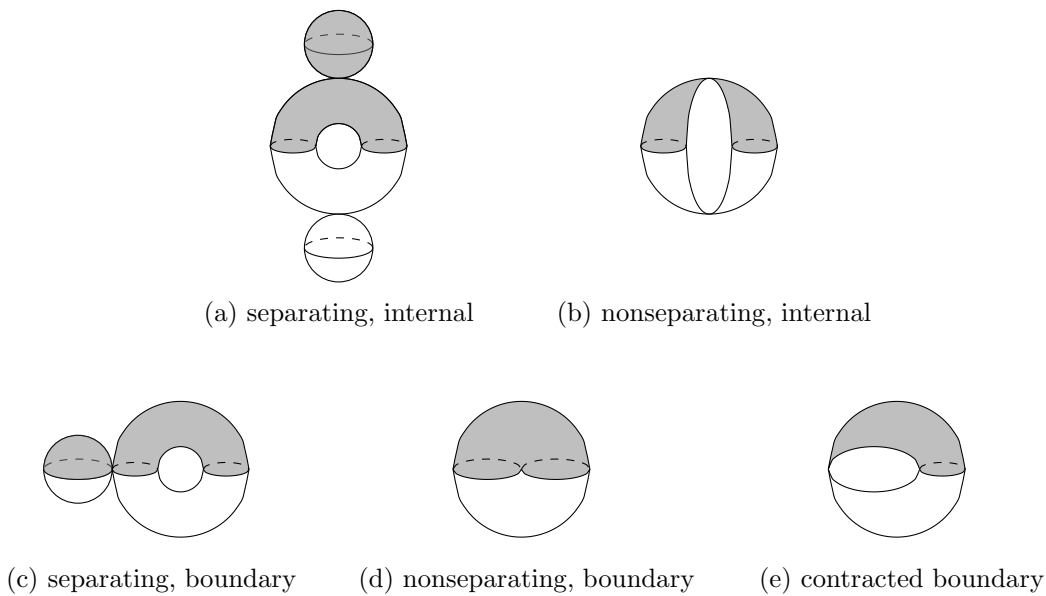
\begin{figure}[h]
\centering

  \begin{subfigure}{.3\textwidth}
  \centering

\begin{tikzpicture}[scale=0.3]
   \draw[ domain=-1:1,  smooth, variable=\x] plot ({\x}, {sqrt(1-\x*\x)});
   \draw[ domain=-3:3,   smooth, variable=\x] plot ({\x}, {sqrt(9-\x*\x)});

   \draw[ domain=-1:1, smooth, variable=\x] plot ({\x}, {-sqrt(1-\x*\x)});
   \draw[ domain=-3:3,  smooth, variable=\x] plot ({\x}, {-sqrt(9-\x*\x)});

    \draw [fill=gray,fill opacity=0.5] plot [domain=-1:1,  smooth]({\x}, {sqrt(1-\x*\x)}) --plot [domain=1:3,  smooth]({\x}, {-sqrt(1-(\x-2)*(\x-2))*0.3}) -- plot [domain=3:-3,  smooth ] ({\x}, {sqrt(9-\x*\x)})-- plot [domain=-3:-1,  smooth]({\x}, {-sqrt(1-(\x+2)*(\x+2))*0.3});

    \draw  [dashed, domain=1:3,  smooth] plot ({\x}, {sqrt(1-(\x-2)*(\x-2))*0.3});

    \draw  [dashed, domain=-3:-1,  smooth] plot ({\x}, {sqrt(1-(\x+2)*(\x+2))*0.3});

    \draw (0,4.5) circle (1.5);
    \draw (-1.5,4.5) arc (180:360:1.5 and 0.45);
    \draw[dashed] (1.5,4.5) arc (0:180:1.5 and 0.45);

     \draw (0,-4.5) circle (1.5);
    \draw (-1.5,-4.5) arc (180:360:1.5 and 0.45);
    \draw[dashed] (1.5,-4.5) arc (0:180:1.5 and 0.45);

    \draw [fill=gray,fill opacity=0.5] (0,4.5) circle (1.5);
   
\end{tikzpicture}

  \caption{separating, internal }
\end{subfigure}
  \begin{subfigure}{.3\textwidth}
  \centering

\begin{tikzpicture}[scale=0.3]

   \draw[ domain=-3:3,   smooth, variable=\x] plot ({\x}, {-sqrt(9-\x*\x)});
  \draw [fill=gray,fill opacity=0.5] plot [domain=-1:1,  smooth]({\x}, {sqrt(1-\x*\x)*3}) --plot [domain=1:3,  smooth]({\x}, {-sqrt(1-(\x-2)*(\x-2))*0.3}) -- plot [domain=3:-3,  smooth ] ({\x}, {sqrt(9-\x*\x)})-- plot [domain=-3:-1,  smooth]({\x}, {-sqrt(1-(\x+2)*(\x+2))*0.3});

  \draw  [dashed, domain=1:3,  smooth] plot ({\x}, {sqrt(1-(\x-2)*(\x-2))*0.3});

    \draw  [dashed, domain=-3:-1,  smooth] plot ({\x}, {sqrt(1-(\x+2)*(\x+2))*0.3});

    \draw[ domain=-1:1,   smooth, variable=\x] plot ({\x}, {-sqrt(1-\x*\x)*3});
\end{tikzpicture}
\vspace{0.9cm}
  \caption{non-separating, internal }
\end{subfigure}

\begin{subfigure}{.3\textwidth}
  \centering
\vspace{0.9cm}
\begin{tikzpicture}[scale=0.3]
  \draw[ domain=-1:1, smooth, variable=\x] plot ({\x}, {-sqrt(1-\x*\x)});
   \draw[ domain=-3:3,  smooth, variable=\x] plot ({\x}, {-sqrt(9-\x*\x)});

    \draw [fill=gray,fill opacity=0.5] plot [domain=-1:1,  smooth]({\x}, {sqrt(1-\x*\x)}) --plot [domain=1:3,  smooth]({\x}, {-sqrt(1-(\x-2)*(\x-2))*0.3}) -- plot [domain=3:-3,  smooth ] ({\x}, {sqrt(9-\x*\x)})-- plot [domain=-3:-1,  smooth]({\x}, {-sqrt(1-(\x+2)*(\x+2))*0.3});

    \draw  [dashed, domain=1:3,  smooth] plot ({\x}, {sqrt(1-(\x-2)*(\x-2))*0.3});

    \draw  [dashed, domain=-3:-1,  smooth] plot ({\x}, {sqrt(1-(\x+2)*(\x+2))*0.3});

    \draw (-4.5,0) circle (1.5);
    \draw (-6,0) arc (180:360:1.5 and 0.45);
    \draw[dashed] (-3,0) arc (0:180:1.5 and 0.45);
    \draw[fill = gray, opacity = 0.5] (-6,0) arc (180:360:1.5 and 0.45) arc (0:180:1.5);
\end{tikzpicture}
\vspace{0.15cm}

  \caption{separating, boundary }
\end{subfigure}
\begin{subfigure}{.3\textwidth}
  \centering

\begin{tikzpicture}[scale=0.3]
\vspace{0.15cm}

   \draw[ domain=-3:3,  smooth, variable=\x] plot ({\x}, {-sqrt(9-\x*\x)});

    \draw [fill=gray,fill opacity=0.5] plot [domain=0:3,  smooth]({\x}, {-sqrt(2.25-(\x-1.5)*(\x-1.5))*0.3}) -- plot [domain=3:-3,  smooth ] ({\x}, {sqrt(9-\x*\x)})-- plot [domain=-3:0,  smooth]({\x}, {-sqrt(2.25-(\x+1.5)*(\x+1.5))*0.3});

    \draw  [dashed, domain=0:3,  smooth] plot ({\x}, {sqrt(2.25-(\x-1.5)*(\x-1.5))*0.3});

    \draw  [dashed, domain=-3:0,  smooth] plot ({\x}, {sqrt(2.25-(\x+1.5)*(\x+1.5))*0.3});
  
\end{tikzpicture}
\vspace{0.15cm}

  \caption{non-separating, boundary}
\end{subfigure}
\begin{subfigure}{.3\textwidth}
  \centering

\begin{tikzpicture}[scale=0.3]
\vspace{0.15cm}
  \draw[ domain=-3:3,  smooth, variable=\x] plot ({\x}, {-sqrt(9-\x*\x)});

    \draw [fill=gray,fill opacity=0.5] plot [domain=1:3,  smooth]({\x}, {-sqrt(1-(\x-2)*(\x-2))*0.3}) -- plot [domain=3:-3,  smooth ] ({\x}, {sqrt(9-\x*\x)})-- plot [domain=-3:1,  smooth]({\x}, {sqrt(4-(\x+1)*(\x+1))*0.5});

    \draw  [dashed, domain=1:3,  smooth] plot ({\x}, {sqrt(1-(\x-2)*(\x-2))*0.3});

    \draw  [ domain=-3:1,  smooth] plot ({\x}, {-sqrt(4-(\x+1)*(\x+1))*0.5});
\end{tikzpicture}
\vspace{0.15cm}

  \caption{contracted boundary }
\end{subfigure}

\caption{The five types of nodes on a nodal marked surface.}
\label{fig node type}
\end{figure}

An \textit{anchored nodal marked surface} is a nodal marked surface together with a $\phi$-invariant choice of a distinguished internal marked point (called the \textit{anchor}) on each connected component $C'$ of $C$ that is disjoint from the set $\text{Fix}(\phi)$. We denote by $Anc\subseteq I$ the set of indexes of anchors lying on $\Sigma$. 
\begin{rmk}
We focus mainly on surfaces with boundaries and their degenerations, \textit{i.e.} connected nodal marked surfaces with nonempty $\text{Fix}(\phi)$. For a connected component $C'$ of a nodal marked surface $C$, if $C'$ does not intersect with the set $\text{Fix}(\phi)$, we view $C'$ as obtained by normalizing a separating internal or contracted boundary node. 
An anchor $z_i$ in a connected component $C'$ should be considered as the half-node corresponding to the separating node that we normalize to obtain $C'$.
\end{rmk}

Let $C$ be an anchored nodal marked surface with
order-$r$ cyclic isotopy groups at markings and nodes, a \textit{$r$-spin structure} on $C$ is 
\begin{itemize}
    \item 
 an orbifold complex line bundle $L$ on $C$,
\item an isomorphism 
$$\kappa:L^{\otimes r} \cong \omega_{C,log}:=\omega_{C} \otimes {\mathcal{O}}\left(-\sum_{i \in I}  [z_i] - \sum_{i \in I}  [\overline{z_i}] - \sum_{j \in B} [x_j]\right),$$

\item an involution $\widetilde{\phi}: L \to L$ lifting $\phi$.
\end{itemize}
The local isotopy of $L$ at a point $p$ is characterized by an integer $\operatorname{mult}_p(L)\in \{0,1,\dots,r-1\}$ in the following way: the local structure of the total space of $L$ near $p$ is $[\mathbb C^2/(\mathbb Z/r\mathbb Z)]$, where the canonical generator $\xi \in \mathbb Z/r\mathbb Z$ acts by $\xi\cdot (x,y)=(\xi x, \xi^{\operatorname{mult}_p(L)}y)$.
We denote by $RI\subseteq I$ and $RB\subseteq B$ the subsets of internal and boundary marked points $p$ satisfying $\operatorname{mult}_p(L)=0$. An \textit{associated twisted r-spin structure} $S$ is defined by
$$
S:=L\otimes \mathcal O\left( -\sum_{i \in \widetilde{RI}} r [z_i] - \sum_{i \in \widetilde{RI}} r [\overline{z_i}] - \sum_{j \in RB} r [x_j]\right),
$$
where $\widetilde{RI}\subseteq RI$ is a subset satisfying $ RI\setminus \widetilde{RI}\subseteq Anc$.

For an internal marked point $z_i$, we define the  \textit{internal twist} at $z_i$ to be $a_i:=\operatorname{mult}_{z_i}(L)-1$ if $i\in I\setminus \widetilde{RI}$ and $a_i:=r-1$ if $i\in \widetilde{RI}$. For a boundary marked point $x_i$, we define the  \textit{boundary twist} at $x_j$ as $b_j:=\operatorname{mult}_{x_j}(L)-1$ if $i\in B\setminus {RB}$ and as $b_j:=r-1$ if $j\in {RB}$. Note that all the marked points with twist $-1$ are indexed by $RI\setminus \widetilde{RI}\subseteq Anc$. When the surface $C$ is smooth, the coarse underlying bundle $|S|$ over the coarse underlying curve $|C|$ satisfies
$$|S|^{\otimes r} \cong \omega_{|C|} \otimes {\mathcal{O}}\left(-\sum_{i \in I} a_i [z_i] - \sum_{i \in I} a_i [\overline{z_i}] - \sum_{j \in B} b_j[x_j]\right).$$

\begin{obs}
\label{obs rank open}
A connected genus-$g$ nodal marked surface admits a twisted $r$-spin structure with twists $a_i$ and $b_j$ if and only if
\begin{equation}\label{eq rank open}
\frac{2\sum_{i \in I} a_i + \sum_{j\in B} b_j+ (g-1)(r-2)}{r}\in \mathbb Z.
\end{equation}
This formula is the specialization to our setting of the more well-known fact \cite{Witten93}: a (closed) connected nodal marked genus-$g$ curve admits a twisted $r$-spin structure twists $a_i$ if and only if 
\begin{equation}\label{eq rank close}
\frac{\sum_{i\in I} a_i +(g-1)(r-2)}{r}\in {\mathbb{Z}}.
\end{equation}
\end{obs}
    
We can extend the definition of twists to \textit{half-nodes}. Let $n: \widehat{C} \to C$ be the normalization morphism. For a half-node $q\in \widehat{C}$, we denote by $\sigma_0(q)$ the other half-node corresponding to the same node $n(q)$ as $q$. The isotopies of $n^*L$ at $q$ and $\sigma_0(q)$ satisfy $$\operatorname{mult}_{q}(n^*L)+\operatorname{mult}_{\sigma_0(q)}(n^*L)\equiv0 \mod r.$$ It is important to note that $n^*S$ may not be a twisted $r$-spin structure (associated with $n^*L$), because its connected components could potentially contain too many marked points with twist $-1$ (note that marked points with twist $-1$ are anchors).  Nevertheless, there is a canonical way to choose a minimal subset $\mathcal{R}$ of the half-nodes making
\begin{equation}
\label{eq normalize S}
\widehat{S}:= n^* S \otimes {\mathcal{O}}\left(-\sum_{q \in \mathcal{R}} r [q]\right)
\end{equation}
a twisted $r$-spin structure: denoting by $\mathcal T$ the set of half-nodes $q$ of $C$ where $\operatorname{mult}_{q}(n^*L)=0$, we define
\begin{equation*}
\mathcal A:= \left\{q \in \mathcal T\colon 
\begin{array}{ccc}
    \text{$n(q)$ is a separating internal node; after normalizing}\\
    \text{the node $n(q)$, the half-node $\sigma_0(q)$ belongs to a connected }\\
    \text{component meeting $\text{Fix}(\phi)$ or containing an anchor.}
\end{array} \right\}
\end{equation*}
and set $\mathcal R := \mathcal T \setminus \mathcal A$. See \cite[Section 2.3]{BCT1} for more details. 
We define $c_t$, \textit{the twist of $S$ at a half-node} $h_t$, as $c_t:=\operatorname{mult}_{h_t}(n^*L)-1$ if $h_t \notin \mathcal R$ and as $c_t:=r-1$ if $h_t\in \mathcal R$.
For each irreducible component $C_l$ of $\widehat{C}$ with half-nodes $\{h_t\}_{t \in N_l}$, we have
$$\bigg(|\widehat{S}|\big|_{|C_l|}\bigg)^{\otimes r} \cong \omega_{|C_l|} \otimes {\mathcal{O}}\left(-\sum_{\substack{i \in I\\z_i\in C_l}} a_i[z_i] - \sum_{\substack{i \in I\\ \overline{z_i}\in C_l}} a_i [\overline{z_i}] - \sum_{\substack{j \in B\\ x_j\in C_l}} b_j[x_j] - \sum_{t \in N_l} c_t [h_t]\right);$$
note that in the case where $C_l$ intersects with $\partial \Sigma$, the set $\{h_t\}_{t \in N_l}$ is invariant under $\phi$, and the half-nodes conjugated by $\phi$ have the same twist.

Note that if $h_{t_1}=\sigma_0(h_{t_2})$, then
\begin{equation}\label{eq sum of twist at node}c_{t_1} + c_{t_2} \equiv -2 \mod r.\end{equation}
We say a node is \textit{Ramond} if one (hence both) of its half-nodes $h_t$ satisfy $c_t \equiv -1 \mod r$,  and it is said to be \textit{Neveu--Schwarz (NS)} otherwise. Note that if a node is Ramond, then both of its half-nodes lie in the set $\mathcal T$; moreover, a half-node has twist $-1$ if and only if it lies in $\mathcal A$. The set $\mathcal{R}$ in equation~\eqref{eq normalize S} is chosen in a way that each separating internal Ramond node has precisely one half-edge in $\mathcal{R}.$

Associated to each twisted $r$-spin structure $S$, we define a Serre-dual bundle $J:=S^{\vee} \otimes \omega_{C}$. Note that the involutions on $C$ and $L$ induce involutions on $S$ and $J$; by an abuse of notation, we denote the involutions on $S$ and $J$ also by $\widetilde{\phi}$.

For a nodal marked surface, the boundary ${\partial}\Sigma$ of $\Sigma$ is endowed with a well-defined orientation, determined by the complex orientation on the preferred half $\Sigma \subseteq |C|$. This orientation induces the notion of positivity for $\phi$-invariant sections of $\omega_{|C|}$ over ${\partial} \Sigma$: let $p$ be a point of $\partial \Sigma$ which is not a node, we say a section $s$ is \textit{positive} at a $p$ if, for any tangent vector $v \in T_p({\partial} \Sigma)$ in the direction of orientation, we have $\langle s(p), v \rangle > 0$, where $\langle -, - \rangle$ is the natural pairing between cotangent and tangent vectors.

Let $C$ be an anchored nodal marked surface, and let $A$ be the complement of the special points in $\partial \Sigma$.  We say a twisted $r$-spin structure on such $C$ is \textit{compatible on the boundary components} if there exists a $\widetilde{\phi}$-invariant section $v \in C^0\left(A, |S|^{\widetilde{\phi}}\right)$ 
(called a \textit{lifting} of $S$  on boundary components) such that the image of $v^{\otimes r}$ under the map on sections induced by the inclusion $|S|^{\otimes r} \to  \omega_{|C|}$ is positive.  
We say $w \in C^0\left(A, |J|^{\widetilde{\phi}}\right)$ is a \textit{Serre-dual lifting of $J$ on the boundary components with respect to $v$} if $\langle w, v \rangle \in C^0(A, \omega_{|C|})$ is positive, where $\langle -, - \rangle$ is the natural pairing between $|S|^{\vee}$ and $|S|$.  This $w$ is uniquely determined by $v$ up to positively scaling.

 The equivalence classes of liftings of $J$ (or equivalently, $S$) on the boundary components up to positively scaling can be considered as continuous sections of the $S^0$-bundle $\left(|J|^{\widetilde{\phi}}\setminus |J|_0\right)\big/ \mathbb R_+$ over $A$, where $|J|_0$ denotes the zero section of $|J|^{\widetilde{\phi}}$. Given an equivalence class $[w]$  of liftings, we say a boundary marked point or boundary half-node $x_j$ is \textit{legal}, or that $[w]$ \textit{alternates} at $x_j$, if $[w]$, as a section of $S^0$-bundle, cannot be continuously extended to $x_j$.
We say an equivalence class $[w]$ of liftings of $J$ on boundaries is a \textit{grading of a twisted $r$-spin structure on boundary components} if, for every Neveu--Schwarz boundary node, one of the two half-nodes is legal and the other is illegal. 
\begin{rmk}
    The requirement that every NS boundary node has one legal and one illegal half-node arises from the behaviour of a grading on boundary components during degenerations. See \cite{BCT1} for more details.
\end{rmk}

Let $q$ be a contracted boundary node of $C$, we say a twisted $r$-spin structure on $C$ is \textit{compatible} at $q$ if $q$ is Ramond and there exists a $\widetilde{\phi}$-invariant element $v \in |S|\big|_q$  (called a \textit{lifting} of $S$ at $q$) such that the image of $v^{\otimes r}$ under the map $|S|^{\otimes r}\big|_q \to \omega_{|C|}\big|_q$ is positive imaginary under the canonical identification of $\omega_{|C|}\big|_q$ with ${\mathbb{C}}$ given by the residue. See \cite[Definition 2.8]{BCT1} for more details.  Such a $v$ also admits a Serre-dual lifting, \textit{i.e.}  a $\widetilde{\phi}$-invariant $w \in |J|\big\vert_q$ such that $\langle v, w \rangle$ is positive imaginary.  We refer to the equivalence classes $[w]$ of such $w$ up to positively scaling as a \textit{grading at contracted boundary node $q$}.

We say a twisted $r$-spin structure is \textit{compatible} if it is compatible on boundary components and at all contracted boundary nodes.  A (total) \textit{grading} is the collection of grading on boundary components together with a grading at each contracted boundary node. We say a grading is \textit{legal} if every boundary marked point is legal.

As we will see in Section \ref{sec:orientation}, the grading is crucial in determining a canonical relative orientation for the Witten bundle, which is one key ingredient in defining open $r$-spin intersection numbers. In the sequel \cite{TZ2}, we will define canonical boundary conditions, again using the grading.

The relation between the twists and legality, and the obstructions to having a grading, are summarized in the following proposition.
\begin{prop}
\label{prop lifting}
\begin{enumerate}
\item\label{it lift odd exist} When $r$ is odd, any twisted $r$-spin structure is compatible, and there is a unique grading.
\item\label{it lift odd legal} Suppose $r$ is odd, a boundary marked point, or boundary half-node, $x_j$ in a twisted $r$-spin structure with a grading is legal if and only if its twist is odd.
\item\label{it lift even compatible} When $r$ is even, the boundary twists $b_j$ in a compatible twisted $r$-spin structure must be even.

\item\label{it Ramond boundary node}
Ramond boundary nodes can appear in a graded structure only when $r$ is odd, and in this case, their half-nodes are illegal with twists $r-1.$
\item\label{it lift compatible parity}
There exists grading that alternates precisely at a subset $D \subset \{x_j\}_{j \in B}$ if and only if
\begin{equation}\label{eq boundarywise pairty}
    |D\cap \partial_i \Sigma| \equiv  \Theta_i \mod 2
\end{equation}

for all connected components $\partial_i \Sigma$ of $\partial \Sigma$, where $\Theta_i:=0$ if $|J|^{\tilde\phi}_{\partial_i\Sigma}$ is orientable and $\Theta_i:=1$ if $|J|^{\tilde\phi}_{\partial_i\Sigma}$ is not orientable.
In particular, we have
\begin{equation}\label{eq pairty all genus}
   \frac{2\sum a_i + \sum b_j-2g + 2}{r} \equiv |D| \mod 2.
\end{equation}

\end{enumerate}
\end{prop}
\begin{proof}
    The first four items are local properties. Their proof is exactly the same as in \cite[Proposition 2.5 and Observation 2.13]{BCT1}, where the $g=0$ version is proven. 

    For the fifth item, \eqref{eq boundarywise pairty} follows directly from the definition of gradings and legality, and \eqref{eq pairty all genus} is obtained by the summation of \eqref{eq boundarywise pairty} over all boundaries $\partial_i \Sigma$. Actually,  similar to the proof of the fourth item of \cite[Proposition 2.5]{BCT1},
    we consider a $\widetilde{\phi}$-invariant meromorphic section $s_m$ of $\lvert J \rvert$, then $\Theta_i$ equals (up to modulo $2$) the number of zeros minus the number of poles of $s_m$ lying on $\partial_i\Sigma$. On the other hand, the zeros and poles of $s_m$ not lying on any $\partial \Sigma_i$ appear in pairs, therefore the degree of $\lvert J\rvert$, which is the left-hand side of \eqref{eq pairty all genus}, equals (up to modulo $2$) the summation of all $\Theta_i$, which is the right-hand side of \eqref{eq pairty all genus} by \eqref{eq boundarywise pairty}.
\end{proof}

When a Ramond contracted boundary node is normalized, the grading at this boundary node induces an additional structure at the corresponding half-node (see \cite[Definition 2.8]{BCT1} for an equivalent definition). Note that for an internal marked point with twist $r-1$, there is a map
$$\tau': \left(|S|\otimes \mathcal O\left([z_i]\right)\right)^{\otimes r}\big|_{z_i}\to \omega_{|C|}([z_i])\big|_{z_i}\cong\mathbb C,$$
where the second identification is the residue map. 

\begin{definition}\label{def: normalized contracted boundary marking}
A \textit{normalized contracted boundary marked point} on an anchored nodal marked surface with a twisted $r$-spin structure is a Ramond internal marked point with twist $r-1$, together with
\begin{enumerate}
\item an involution $\widetilde{\phi}$ on the fibre $\left(|S|\otimes {\mathcal{O}}\left([z_i]\right)\right)_{z_i}$ such that
\[
\tau'(\widetilde{\phi}(v)^{\otimes r})=-\overline{\tau'(v^{\otimes r})}\quad \text{for all $v\in\left(|S|\otimes {\mathcal{O}}\left([z_i]\right)\right)_{z_i}$},
\]
where $w\mapsto\overline{w}$ is the standard conjugation, and
\[\left\{\tau'(v^{\otimes r})\; | \; v\in \left(|S|\otimes {\mathcal{O}}\left([z_i]\right)\right)^{\widetilde\phi}_{z_i}\right\}\supseteq i{\mathbb R}_+,\]where $i$ is the root of $-1$ in the upper half-plane;
\item a connected component $V$ of $\left(|S|\otimes {\mathcal{O}}\left([z_i]\right)\right)^{\widetilde\phi}_{z_i}\setminus\{0\}$, called the \textit{positive direction}, such that $\tau'(v^{\otimes r})\in i{\mathbb R}_+$ for any $v \in V$.
\end{enumerate}
\end{definition}

We can now define the primary objects of interest in this paper:

\begin{definition}
\label{def graded rspin surface}
A \textit{stable graded $r$-spin surface} (legal stable graded $r$-spin surface respectively) is a stable anchored nodal marked surface, together with
\begin{enumerate}
\item a compatible twisted $r$-spin structure $S$ in which all contracted boundary nodes are Ramond;
\item a choice of grading (legal gradings respectively);

\item a set $NCB\subseteq \{i\in I\colon \tw(z_i)=r-1\}$ and an additional structure of normalized contracted boundary marked point at each point in $\{z_i\}_{i\in NCB}$.

\end{enumerate}
For an integer $0\le\h\le \lfloor \frac{r-2}{2}\rfloor$, we say a stable graded $r$-spin surface is of level-$\h$ if every legal boundary marked point has twist greater than or equal to $r-2-2\h$, and every illegal boundary marked point has twist smaller than or equal to $2\h$. We will omit the term "level-$\h$" when $\h$ is chosen as a fixed integer.
\end{definition}

\begin{rmk}
    Note that in \cite{BCT2}, the term "stable graded $r$-spin disk" refers to a legal stable graded genus-zero level-$0$ $r$-spin surface.
\end{rmk}

\subsection{The moduli space of graded $r$-spin surfaces} 
Denote by $\Mbar_{g,n}^{1/r}$ the moduli space of stable $r$-spin surfaces without boundary. This space is known to be a smooth Deligne--Mumford stack with projective coarse moduli. The forgetful map to \[\text{For}_{\text{spin}}:\Mbar_{g,n}^{1/r}\to\Mbar_{g,n}\] is finite (see e.g. \cite{ChiodoStable}, or~\cite{Jarvis} which works in a slightly different compactification). $\Mbar_{g,n}^{1/r}$ space admits a decomposition into connected components,
\[
\Mbar_{g,n}^{1/r} = \bigsqcup_{(a_1, \ldots, a_n)} \Mbar_{g,\{a_1, \ldots, a_n\}}^{1/r},
\]
where $\Mbar_{g,\vec{a}}^{1/r}$ denotes the substack of $r$-spin structures with twist $a_i\in\{0,\ldots,r-1\}$ at the $i$-th marked point.  

We denote by $\Mbarstar_{g,k,l}^{1/r}$ the moduli space of connected stable graded genus-$g$ $r$-spin surfaces with $k$ boundary and $l$ internal marked points.
As in the closed case, we have a set-theoretic decomposition of the space
\[\Mbarstar_{g,k,l}^{1/r} = \bigsqcup_{\vec{a}} \Mbarstar_{g,\{b_1,\ldots,b_k\}, \{a_1,\ldots,a_k\}}^{1/r},\]
in which $\Mbarstar_{g,\{b_1,\ldots,b_k\}, \{a_1,\ldots,a_k\}}^{1/r} \subset \Mbarstar^{1/r}_{g,k,l}$ consists of graded surfaces for which the $i$-th boundary (internal) marked point has twist $b_i$ ($a_i$ resp.).
We denote by \begin{equation}\label{eq for_spin_g}\text{For}_{\text{spin}}:\Mbarstar_{g,k,l}^{1/r}\to\Mbar_{g,k,l}\end{equation} the map that forgets the grading and the spin structure.
\begin{thm}\label{thm:all_g_moduli}
$\Mbarstar_{g,k,l}^{1/r}$ is a compact orbifold with corners of real dimension $3g-3+k+2l$. It is associated with a universal curve with a universal $r$-spin line and a universal grading. When $g=0$ this moduli is orientable.
\end{thm}
In \cite[Theorem 3.4]{BCT1}, this theorem is proven for $g=0,$ and a special choice of boundary twists, based on the results of \cite[Section 2]{Zernik}. For the purposes of this paper we need the more general case, but the proof is almost identical, and is again based on \cite[Section 2]{Zernik}.
For this reason, we will allow ourselves to be quite brief, and the reader should consult for more details. We will highlight the main points which differ from \cite[Theorem 3.4]{BCT1}. Our notion of orbifold with corners follows that of \cite[Section 3]{Zernik}.
\begin{proof}
We have the following sequence of maps, whose content is explained below.
\begin{equation}
\Mbarstar_{g,k,l}^{1/r}\stackrel{(E)}{\to}\widehat{\mathcal{M}}_{g,k,l}^{1/r}\stackrel{(D)}{\hookrightarrow}\widetilde{\mathcal{M}}_{g,k,l}^{1/r}\stackrel{(C)}{\to}\widetilde{\mathcal{M}}_{g,k,l}^{1/r,{\mathbb Z}_2}\stackrel{(B)}{\to}
\overline{\mathcal{M}}_{g,k+2l}^{1/r,{\mathbb Z}_2}\stackrel{(A)}{\to}\Mbar_{g,k+2l}^{'1/r}.
\end{equation}
\textbf{Step (A): }
$\overline{\mathcal{M}}_{g,k+2l}^{'1/r}$ is the suborbifold of $\Mbar_{g,k+2l}^{1/r}$ defined by the conditions
\begin{itemize}
    \item for $i\in\{k+1,\ldots,k+l\}$ the $i$-th marking's twist, and the $(i+l)$-th marking twist are equal to the same number $a_i;$
    \item  for even $r$, the twist of the $j$-th marking $b_j$ for $j\in\{1,\ldots, k\}$ satisfies
    \[b_j\equiv 0~\text{mod~2};\]
\end{itemize} 
Consider the involution on this space, defined by
\[(C;w_1, \ldots, w_{k+2l}, S) \mapsto (\overline{C}; w_1, \ldots, w_k, w_{k+l+1},\ldots, w_{k+2l},w_{k+1},\ldots,w_{k+l}, \overline{S}),\]
where $\overline{C}$ and $\overline{S}$ are the same as $C$ and $S$ but with the conjugate complex structure (more details on the fixed point functor on stacks can be found in \cite{fixedPointStack}). Then $\Mbar_{g,k+2l}^{1/r, {\mathbb Z}_2}$ denotes its fixed locus. As the fixed locus of an anti-holomorphic involution, $\Mbar_{g,k+2l}^{1/r, {\mathbb Z}_2}$ is a real orbifold.  It is the classifying space of isomorphism types of \emph{real} marked $r$-spin curves (curves with an involution $\widetilde\phi$ which covers the conjugation $\phi$ on the underlying real curve) and the prescribed twists, and it maps to $\Mbar_{g,k+2l}^{1/r}$. In particular, it inherits a universal curve via pullback.

\textbf{Step (B): }
In this step, we cut $\overline{\mathcal{M}}_{g,k+2l}^{1/r,{\mathbb Z}_2}$ along the real simple normal crossing divisors consisting of curves with at least one real node, via Zernik's "real hyperplane blow-up" \cite{Zernik}. The result, as argued in \cite[Section 3.3]{Zernik}, is an orbifold with corners $\widetilde{\mathcal{M}}_{g,k,l}^{1/r, {\mathbb Z}_2}$.  

\textbf{Step (C): }
From here, we define $\widetilde{\mathcal{M}}_{g,k,l}^{1/r}$ to be the disconnected $2$-to-$1$ cover of $\widetilde{\mathcal{M}}_{g,k,l}^{1/r,{\mathbb Z}_2}$. The generic point of the moduli space $\widetilde{\mathcal{M}}_{g,k,l}^{1/r}$ corresponds to a smooth marked real spin curve with a choice of a distinguished "half" $\Sigma$, that is a connected component of $C\setminus C^\phi$.  In the generic (smooth) situation, this also induces an orientation on $C^{\phi}$. It is important to note that this choice can be uniquely continuously extended to nodal points, see \cite[Section 2.6]{Zernik}, as opposed to being independently chosen for each boundary component.
The proof that $\widetilde{\mathcal{M}}_{g,k,l}^{1/r}$ is an orbifold with corners is identical to the proof of \cite[Thoerem 2]{Zernik}.

\textbf{Step (D): }
Inside $\widetilde{\mathcal{M}}_{g,k,l}^{1/r}$, we denote by $\widehat{\mathcal{M}}_{g,k,l}^{1/r}$ the union of connected components such that
\begin{itemize}
    \item the marked points $w_{k+1},\ldots, w_{k+l}$ lie in the distinguished stable half $\Sigma$;
    \item for even $r,$ the spin structure is compatible;
\end{itemize} 
$\widehat{\mathcal{M}}_{g,k,l}^{1/r}$ is clearly also an orbifold with corners, as a union of connected components of an orbifold with corners.

\textbf{Step (E): }
Finally, $\Mbarstar_{g,k,l}^{1/r}$ is the cover of $\widehat{\mathcal{M}}_{g,k,l}^{1/r}$ given by a choice of grading.

As a set $\Mbarstar_{g,k,l}^{1/r}$ is the space we defined above set-theoretically. It inherits the orbifold-with-corners structure from $\widehat{\mathcal{M}}_{g,k,l}^{1/r}$. It is compact since compactness is preserved at every step. It carries a universal curve and $r$-spin line since any intermediate space in the above steps carries such objects, and it has a universal grading by the construction in the last step.

In order to prove orientability in $g=0$, observe first that the moduli space $\Mbar_{0,k,l}$ of (spinless) stable marked disks is orientable \cite[Proposition 3.12]{BCT1}.
Also, by \cite[Proposition 2.15, Observation 3.9]{BCT1}, for $g=0,$ the map of \eqref{eq for_spin_g} is a bijection from any nonempty component $C$ of $\Mbarstar_{0,k,l}^{1/r}$ to $\Mbar_{0,k,l},$ which is, moreover, a diffeomorphism from $\text{Int}(C)\setminus S$ onto its image, where $S\subset C$ is a finite union of codimension-$2$ suborbifolds. We use this to pullback the orientation to $\text{Int}(C)\setminus S$. Since adding codimension-$2$ strata or boundaries does not ruin the orientability of $C,$ and hence also $\Mbarstar_{0,k,l}^{1/r}$, are orientable.
\end{proof}

The superscript $*$ indicates that there might be illegal boundary marked points. 
We denote by $\Mbar_{g,k,l}^{1/r}\subseteq \Mbarstar_{g,k,l}^{1/r}$ the submoduli space parametrizing legal graded genus-$g$ $r$-spin surfaces with $k$ boundary markings and $l$ internal markings. 

The moduli spaces of smooth stable graded genus-$g$ $r$-spin surfaces are denoted by the notation without the bar on the top, \textit{e.g.} we denote by ${\mathcal M^*}_{g,k,l}^{1/r} \subset {\Mbarstar}_{g,k,l}^{1/r}$ the moduli space of smooth stable graded genus-$g$ $r$-spin surface with $k$ boundary and $l$ internal marked points.

Assuming that the internal marked points $\{z_i\}_{i\in I}$ have twists $\{a_i\}_{i\in I}$, by an abuse of notation, we also denote the set $\{a_i\}_{i\in I}$ by $I$.
Similarly, we denote by $B$ the set $\{b_j\}_{j\in B}$ equipped with a preselected legality for each of its elements. We denote by ${\Mbarstar}_{g,{B},I}^{1/r}\subseteq \Mbarstar_{g,\lvert B\rvert,\lvert I\rvert}^{1/r}$ the components parametrizing the stable graded genus-$g$ $r$-spin surfaces whose marked points are indexed by $B$ and $I$. The superscript will be $*$ omitted as the legality is contained in the data of $B$, unless we want to emphasize the existences of illegal boundary markings. 

Given a stable graded genus-$g$ $r$-spin surface, the decomposition of $\partial \Sigma$ into connected components induces a decomposition of $B$ into several (possibly empty) subsets; moreover, the canonical orientation on its boundary $\partial\Sigma$ induces a bijective map $\sigma_2\colon B\to B$, such that every nonempty subset of $B$ in the decomposition is an orbit of $\sigma_2$. We denote by $\bar{B}$ the set $B$ equipped with such a decomposition and a bijective map. We denote by ${\Mbarstar}_{g,\bar{B},I}^{1/r}\subseteq{\Mbarstar}_{g,{B},I}^{1/r}$ the components that parametrize the stable graded genus-$g$ $r$-spin surfaces inducing the same extra data as $\bar{B}$ on $B$. 
In the case $g=0$, $B$ is trivially decomposed into only one subset, which is $B$ itself, and $\sigma_2$ induces a cyclic order on $B$; we write $\bar{B}=\overline{\{b_1,b_2,\dots,b_{\lvert B \rvert}\}}$ to make the cyclic order manifest, where $\sigma_2(b_i)=b_{i+1}$ for $1\le i \le \lvert B \rvert$ and $\sigma_2(b_{\lvert B \rvert})=b_{1}$.
In the case $g=1$, we write $\bar{B}=\{\bar{B}^1,\bar{B}^2 \}$ if $\bar{B}^1$ and $\bar{B}^2$ are the two (not necessarily non-empty) orbits of $\sigma_2$.

\begin{rmk}
    We denote by $\widehat{\mathcal{M}}_{g,\bar{B},I}^{1/r}\subseteq \widehat{\mathcal{M}}_{g,\lvert B \rvert,\lvert I \rvert}^{1/r}$ the image of ${\Mbarstar}_{g,\bar{B},I}^{1/r}\subseteq {\Mbarstar}_{g,\lvert B \rvert,\lvert I \rvert}^{1/r}$ under the covering map in Step (E) of the proof of Theorem \ref{thm:all_g_moduli}. We denote by $\text{For}_{gr}\colon {\Mbarstar}_{g,\bar{B},I}^{1/r}\to \widehat{\mathcal{M}}_{g,\bar{B},I}^{1/r}$ the restricted covering map, which can be regarded as the map forgetting the grading.
    
    We denote by $h$ the number of subsets in the decomposition of $B$ (which is a part of data in $\bar{B}$); note that $h$ is the number of boundary components for a generic $\Sigma$ parametrized by ${\Mbarstar}_{g,\bar{B},I}^{1/r}$.  The map $\text{For}_{gr}$ is an isomorphism when $r$ is odd; while it is $2^h$-to-$1$ when $r$ is even (we recall that $B$ and hence $\hat B$ are quipped with a pre-selected legality for each of their elements). 
    We should stress that, since for even $r$ the generic point of $\widehat{\mathcal{M}}_{g,\bar{B},I}^{1/r}$ has a ${\mathbb{Z}}/2{\mathbb{Z}}$ isotropy, the number of nonisomorphic choices of gradings is $2^{h-1},$ as any two choices which differ by globally multiplying the gradings by $-1$ are isomorphic.

\end{rmk}

\subsection{Graded $r$-spin graphs}\label{subsec rspin graphs}
Analogously to the more familiar setting of the moduli space of curves, also the moduli space of graded $r$-spin surfaces can be stratified according to decorated dual graphs. 
\begin{definition}\label{def graph without spin}
A \textit{pre-stable dual graph} is a tuple
\[\Gamma = (V, H, \sigma_0, \sigma_1, H^{CB},\Hat{g}, n, \sigma_2,  m),\]
in which
\begin{enumerate}[label=(\emph{\roman*})]
\item $V$ is a finite set (the \textit{vertices}) which can be decomposed as $V = V^O \sqcup V^C$, where $V^O$ and $V^C$ are the sets of \textit{open} and \textit{closed vertices};
\item $H$ is a finite set (the \textit{half-edges}) which can be decomposed as $H = H^B \sqcup H^I$, where $H^B$ and $H^I$ are the sets of \textit{boundary} and \textit{internal half-edges};
\item $\sigma_0: H \to V$ is a function mapping each half-edge to the vertex from which it emanates;
\item $\sigma_1: H \to H$ is bijection which can be decomposed into $\sigma_1^B: H^B \to H^B$ and $\sigma_1^I: H^I \to H^I$. The size of each orbit of $\sigma_1$ is required to be at most $2$. We denote by $T^B \subseteq H^B$ and $T^I \subseteq H^I$ the sets of size-$1$ orbits of $\sigma_1$;

\item $H^{CB} \subseteq T^I$ is the set of \textit{contracted boundary tails};
\item $\hat{g}\colon V\to \mathbb Z^{\ge0}$ is a function, the \emph{small genus};
\item $n\colon V\to \mathbb Z^{\ge 0},$ the \emph{number of boundaries}, is an assignment which satisfies that 
$v\in V^C$ if and only if $n(v)=0$; for each $v\in V$ there is a decomposition (into possibly empty subsets)
$$
\sigma_0^{-1}(v)\cap H^B=\bigsqcup_{i\in BD(v)}H^B_i(v)
$$
for a set $BD(v)$ with $n(v)$ elements;
\item $\sigma_2\colon H^B\to H^B$ is a function (\textit{the cyclic order map}) which induces cyclic orders on $H_i^B(v)$ for each $v\in V^O$ and $i\in BD(v)$; 

\item $m$ is a function (the \textit{marking}) given by
\[m = m^B \sqcup m^I: T^B \sqcup (T^I \setminus H^{CB}) \to \mathbb Z_+,\]
where $m^B$  (the \emph{boundary marking}) induces a bijection between $T^B$ and $\{1,2,\dots,\lvert T^B\rvert\}$, and $m^I$ (the \emph{internal marking}) induces a bijection between $T^I \setminus H^{CB}$ and $\{1,2,\dots,\lvert T^I \setminus H^{CB}\rvert\}$.
\end{enumerate}

\end{definition}

We refer to elements of $T^B$ as \textit{boundary tails} and elements of $T^I \setminus H^{CB}$ as \textit{internal tails}, and we set
$T := T^I \sqcup T^B.$
We denote by $E^B$ and $E^I$ the sets of size-$2$ orbits of $\sigma_1^B$ and $\sigma_1^I$, and we refer to them as \textit{boundary edges} and \textit{internal edges}.

We say a boundary edge $e\in E^B$ is \textit{separating} if the connected component containing $e$ separates into two after removing $e$; we denote by $E^B_{sp}$ the set of all separating boundary edges, and refer to all the edges in $E^B_{nsp}:=E^B\setminus E^B_{sp}$ as \textit{non-separating boundary edges}. 
We say an internal edge $e\in E^B$ is \textit{separating} if the connected component containing $e$ separates into two after removing $e$, and at least one of the two new components contains neither open vertices nor contracted boundary tails; we denote by $E^I_{sp}$ the set of all separating internal edges, and refer to all the edges in $E^I_{nsp}:=E^I\setminus E^I_{sp}$ as \textit{non-separating internal edges}.

For each vertex $v$, we set
$k(v) := |(\sigma_0^B)^{-1}(v)|$ and $l(v):= |(\sigma_0^I)^{-1}(v)|.$
For an open vertex $v \in V^O$, we set the genus of $v$ to be 
$$
g(v):=2\hat g (v)+n(v)-1;
$$
we say an open vertex $v \in V^O$ is \textit{stable} if $k(v) + 2l(v) > 2-2g(v)$.
For a closed vertex $v \in V^C$, we set the genus of $v$ to be 
$$
g(v):=\hat g (v);
$$
we say a closed vertex $v \in V^C$ is \textit{stable} if $l(v) > 2-2g(v)$. 
A graph is \textit{stable} 
if all of its vertices are stable, 
and it is \textit{closed} if $V^O = \emptyset.$ A graph is \emph{smooth} if there are no edges or contracted boundary tails.
The genus of a connected graph $\Gamma$ is defined as 
\begin{equation}  \label{eq genus rspin graph}
g(\Gamma):=\sum_{v\in V^O}g(v)+2\sum_{v\in V^C}g(v)+2\lvert E^I \rvert+\lvert E^B \rvert+\lvert H^{CB} \rvert-2\lvert V^C \rvert-\lvert V^O \rvert+1.
\end{equation}

\begin{definition}
An \textit{isomorphism} between two pre-stable dual graphs
\[\Gamma = (V,H, \sigma_0,\sigma_1, H^{CB},\hat{g},n,\sigma_2,  m) \; \text{ and } \; \Gamma' = (V', H', \sigma_0',\sigma_1', H'^{CB},\hat{g}',n',\sigma_2',  m')\]
is a pair $f = (f^V, f^H)$, where
$f^V: V \to V'$ and $f^H: H \to H'$
are bijections satisfying
\begin{multicols}{2}
\begin{itemize}
\item $f^H(H^{O})=H'^{O},f^H(H^{C})=H'^{C}$;
\item $f^V(V^{O})=V'^{O},f^V(V^{C})=V'^{C}$;
\item $f^{H} \circ \sigma_1 =  \sigma'_1 \circ f^H$,
\item $f^{V} \circ \sigma_0 =  \sigma'_0 \circ f^H$,
\item $f^H\circ \sigma_2=\sigma_2'\circ f^H\vert_{H^B}$;
\item $\hat g=\hat g'\circ f^V$,
\item $n=n'\circ f^V$,
\item $m = m' \circ f^H$,
\item $f^H(H^{CB}) = H'^{CB}$.
\end{itemize}
\end{multicols}
We denote by $\text{Aut}(\Gamma)$ the group of automorphisms of $\Gamma$.
\end{definition}

\begin{figure}[h]
\centering

  \begin{subfigure}{.3\textwidth}
  \centering

\begin{tikzpicture}[scale=0.3, every text node part/.style={align=center},]
   \node[rectangle, draw, very thick, minimum size=1]    (mainopen)    {$\hat g=0$\\$n=2$ };
  \node[circle, draw, very thick, minimum size=1, above=0.4cm of mainopen]    (closed)    {$\hat g=0$};
  \draw [double] (mainopen) to (closed);
   
\end{tikzpicture}

  \caption{}
\end{subfigure}
  \begin{subfigure}{.3\textwidth}
  \centering

\begin{tikzpicture}[scale=0.3, every text node part/.style={align=center},]

   \node[rectangle, draw, very thick, minimum size=1]    (mainopen)    {$\hat g=0$\\$n=1$ };
  \node[rectangle, draw, very thick, minimum size=1, left=1cm of mainopen]    (subopen)    {$\hat g=0$\\$n=1$ };
  \draw [double] (mainopen) to (subopen);
\end{tikzpicture}
\vspace{0.9cm}
  \caption{}
\end{subfigure}

\begin{subfigure}{.3\textwidth}
  \centering
\begin{tikzpicture}[scale=0.3, every text node part/.style={align=center},]
  \node[rectangle, draw, very thick, minimum size=1]    (mainopen)    {$\hat g=0$\\$n=2$ };
  \node[rectangle, draw, very thick, minimum size=1, left=1cm of mainopen]    (subopen)    {$\hat g=0$\\$n=1$ };
  \draw  (mainopen) to (subopen);
  \vspace{0cm}
\end{tikzpicture}
\vspace{0.9cm}

  \caption{}
\end{subfigure}
\begin{subfigure}{.3\textwidth}
  \centering
 \vspace{0.3cm}
\begin{tikzpicture}[scale=0.3, every text node part/.style={align=center},]
\vspace{0.15cm}

\node[rectangle, draw, very thick, minimum size=1]    (mainopen)    {$\hat g=0$\\$n=1$ };

 \draw  (mainopen) to [in=-50,out=-130, loop, distance=5cm] (mainopen);
  
\end{tikzpicture}
\vspace{-0.2cm}

  \caption{}
\end{subfigure}
\begin{subfigure}{.3\textwidth}
  \centering

\begin{tikzpicture}[scale=0.3, every text node part/.style={align=center},]
\vspace{0.15cm}
  \node[rectangle, draw, very thick, minimum size=1]    (mainopen)    {$\hat g=0$\\$n=1$ };
  \node[below = 0.5cm of mainopen,circle,fill,inner sep=1pt]  (pt) {};
  \draw  (mainopen) to (pt);
\end{tikzpicture}
\vspace{0.3cm}

  \caption{}
\end{subfigure}
\caption{Dual graphs corresponding to the surfaces in Figure \ref{fig node type}. We represent by thick rectangles the open vertices and by thick circles the open vertices. The boundary (half-)edges are represented by ordinary lines, the  internal (half-)edges are represented by double lines, and the contracted boundary tails are represented by  segments with black endpoints.}
\label{fig dual graph}
\end{figure}
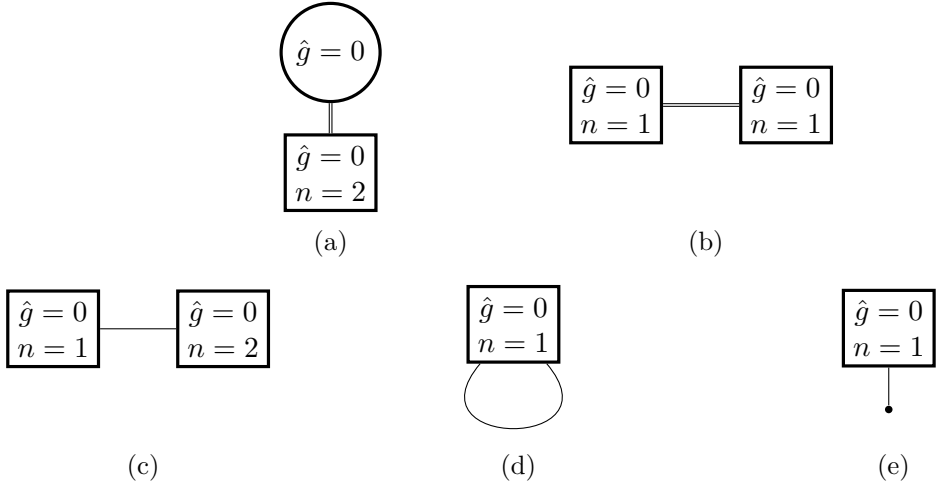

Pre-stable dual graphs encode the discrete data of a marked orbifold Riemann surface with boundary.  In order to encode the additional data of a twisted spin structure and a lifting, we must add further decorations.

\begin{definition}
\label{def rspin dual graph}
A \textit{stable graded $r$-spin graph} is a stable dual graph $\Gamma$ as above, together with maps
\[\twE: H \to \{-1, 0, 1, \ldots, r-1\}\]
(the \textit{twist}) and
\[\alt: H^B \to {\mathbb{Z}}/2{\mathbb{Z}}\]
and two disjoint subsets $T^{anc}\subseteq T^I\setminus H^{CB}$ (the \textit{anchors}) and $T^{NCB} \subseteq T^I\setminus H^{CB}$ (the \textit{normalized contracted boundary tails})
satisfying the following conditions:
\begin{enumerate}[label=(\emph{\roman*})]

\item
 \label{item} For any open vertex $v \in V^O$,
\[2\sum_{h \in (\sigma_0^I)^{-1}(v)} \twE(h) + \sum_{h \in (\sigma_0^B)^{-1}(v)} \twE(h) \equiv 2g(v)-2 \mod r\]
and
\begin{equation}\label{eq legal mod 2}
   \frac{2\sum_{h \in  (\sigma_0^I)^{-1}(v)}\ \twE(h) + \sum_{h \in (\sigma_0^B)^{-1}(v)} \twE(h)-2g(v) + 2}{r} \equiv \sum_{h \in (\sigma_0^B)^{-1}(v)} \alt(h) \mod 2.
\end{equation}

\item
For any closed vertex $v \in V^C$,
\[\sum_{h \in \sigma_0^{-1}(v)} \twE(h) \equiv 2g(v)-2 \mod r.\]

\item \label{item:condition on anchor}
    There is exactly one tail in $T^{anc}$ on each connected component without open vertices and contracted boundary edges, while there are no tails in $T^{anc}$ on other connected components. All tails $t$ with $\twE(t) =-1$ belong to $T^{anc}$.

\item
\label{cond2} For every half-edge $h \in H \setminus T$, we have
\[\twE(h) + \twE(\sigma_1(h)) \equiv r -2 \mod r.\]
No boundary half-edge $h$ satisfies $\twE(h)=-1$. 

In case $h \in H^I \setminus T^I$ satisfies $\twE(h) \equiv -1\mod~r,$ then $\twE(h)=-1$ precisely if the edge $e$ corresponding to $h$ is a separating internal edge and, after detaching $e$, the half-edge $h$ belongs to a connected component containing neither open vertices nor contracted boundary tails, nor tails in $T^{anc}$.

\item
\label{cond1} For each contracted boundary tail $t \in H^{CB}$ or normalized contracted boundary tail $t \in T^{NCB}$, we have $\twE(t) =r-1$.

\item All anchors $t\in T^{anc}$ with $\tw(t)=r-1$ belong to $T^{NCB}$.

\item
\label{it:-1}For each boundary half-edge $h \in H^B \setminus T^B$, if $\twE(h) \neq r-1 $ we have
\[\alt(h) + \alt(\sigma_1(h)) = 1\]
and if $\twE(h) =r -1$ then $\alt(h) = \alt(\sigma_1(h)) = 0.$

\item
If $r$ is odd, then for any $h \in H^B$,
\[\alt(h) \equiv \twE(h) \mod 2,\]
and if $r$ is even, then for any $h \in H^B$,
\[\twE(h) \equiv 0 \mod 2.\]

\end{enumerate}
Boundary half-edges $h$ with $\alt(h) = 1$ are called \textit{legal}, and those with $\alt(h) = 0$ are called \textit{illegal}.  Half-edges $h$ with $\twE(h) \in \{-1, r-1\}$ are called \textit{Ramond}, and those with $\twE(h) \in \{0, \ldots, r-2\}$ are called \textit{Neveu--Schwarz}.  An edge is called \textit{Ramond} if one (hence both) of its half-edges is Ramond and is called \textit{Neveu--Schwarz (NS)} otherwise.
\end{definition}

An \textit{isomorphism} between  stable  graded $r$-spin graphs consists of an isomorphism in the sense of Definition~\ref{def rspin dual graph} that respects $\tw$, $\alt$, $T^{anc}$ and $T^{NCB}$.  
We say a stable graded $r$-spin graph is \textit{legal} if every boundary tail is legal. We say a stable graded $r$-spin graph is \textit{level-$\h$} if every legal boundary tail has twist greater than or equal to $r-2-2\h$, and every illegal boundary tail has twist smaller than or equal to $2\h$.

Any connected genus-$g$ stable (legal, level-$\h$ resp.) graded $r$-spin surface $\Sigma$ induces a connected genus-$g$ (legal, level-$\h$) graded $r$-spin graph $\Gamma(\Sigma)$. If $\Gamma$ is a connected graph, we denote by ${\Mstar}^{1/r}_{\Gamma}\subseteq \Mbarstar_{g(\Gamma), \lvert T^B\rvert,\lvert T^I\setminus H^{CB}\rvert,}^{1/r}$ the moduli space parametrizing $r$-spin surface with dual graph $\Gamma$, and by $\Mbarstar^{1/r}_{\Gamma}$ its closure.   
If $\Gamma$ is disconnected, we define $\Mbarstar^{1/r}_{\Gamma}$ as the product of the moduli spaces associated to its connected components. The superscripts $1/r$ superscript $*$ will be omitted (unless emphasis is needed) as $r$ and the legality of boundary markings are parts of data in graded $r$-spin graph $\Gamma$.
\subsubsection*{Smoothing}
We can obtain a new stable graded $r$-spin dual graph by smoothing an edge or a contracted boundary tail. 
\begin{dfn}\label{def smoothing graph}
Let $\Gamma$ be a stable graded $r$-spin graph and $e$ an edge connecting vertices $v_1$ and $v_2$.  

If $v_1=v_2$, the \textit{smoothing} of $\Gamma$ along $e$ is the graph $d_e\Gamma$ obtained by erasing $e$, setting the small genus $\hat g_{d_e\Gamma}(v_1)=\hat g_{\Gamma}(v_1)+1$, and setting the number of boundaries $n_{d_e\Gamma}(v_1)=n_{\Gamma}(v_1)-1$ if $e\in E^B$, and   $n_{d_e\Gamma}(v_1)=n_{\Gamma}(v_1)$ if $e\in E^I$. The vertex $v_1$ is declared to be open in $d_e\Gamma$ if and only if it is open in $\Gamma$.

If $v_1\ne v_2$, the \textit{smoothing} of $\Gamma$ along $e$ is the graph $d_e\Gamma$ obtained by contracting $e$ and replacing $v_1$ and $v_2$ with a single vertex $v_{12}$. We set the small genus of $v_{12}$ to be $\hat g(v_{12})=\hat g(v_{1})+\hat g(v_{2})$; we set the number of boundaries  $n(v_{12})=n(v_1)+n(v_2)-1$ if $e\in E^B$, and  $n(v_{12})=n(v_1)+n(v_2)$ if $e\in E^I$.
The vertex $v_{12}$ is declared to be closed if and only if both $v_1$ and $v_2$ are closed.

The \textit{smoothing} of $\Gamma$ along $h \in H^{CB}$ is the graph $d_h\Gamma$ obtained by erasing $h$. Let $v$ be the vertex from which $h$ emanates, we set the small genus and the number of boundaries of $v$ in $d_h\Gamma$ to be $\hat g_{d_h\Gamma}(v)=\hat g_{\Gamma}(v)$ and $n_{d_h\Gamma}(v)=n_{\Gamma}(v)+1$.
The vertex $v$ is always declared to be open in $d_h\Gamma$. 

If $\Lambda$ is a smoothing of $\Gamma$, then each (half-)edge $h$ of $\Lambda$ corresponds to a unique (half-)edge of $\Gamma$, we also call it $h$ by an abuse of notation. Let $\sigma_2\colon H^B(\Gamma)\to H^B(\Gamma)$ be the cyclic order map on $\Gamma$, then the cyclic order map $\sigma'_2\colon H^B(\Lambda)\to H^B(\Lambda)$ on $\Lambda$ is given by
\[
\sigma'_2(h):=
\begin{cases}
\sigma_2(h),&\text{if $\sigma_2(h)\in H^B(\Lambda)$};\\
\sigma_2\circ\sigma_1\circ\sigma_2(h),&\text{if  $\sigma_2(h)\notin H^B(\Lambda)$ and $\sigma_2\circ\sigma_1\circ\sigma_2(h)\ne \sigma_1\circ\sigma_2(h)$};\\
\sigma_2\circ\sigma_2(h),&\text{if  $\sigma_2(h)\notin H^B(\Lambda)$ and  $\sigma_2\circ\sigma_1\circ\sigma_2(h)= \sigma_1\circ\sigma_2(h)$}.
\end{cases}
\]
The remaining graph data is kept unchanged after smoothing.

\end{dfn}
Note that the smoothing of a legal (level-$\h$ respectively) stable  graded $r$-spin graph is still legal (level-$\h$ respectively). Moreover, the smoothing of several edges (or contracted boundary tails) is independent of the order of smoothings.

\subsubsection*{Detaching}
We can also obtain a new stable graded $r$-spin graph by detaching an edge or a contracted boundary tail.

Let $\Gamma$ be a stable  graded $r$-spin graph and $e$ be an edge of $\Gamma$, we can detach $\Gamma$ along $e$ in the sense of ordinary graph and keep the extra data invariant. The object we get in this way might not be a stable graded $r$-spin graph since item \ref{item:condition on anchor} in Definition \ref{def rspin dual graph} might not be satisfied. To fix this, in the case where $e$ is a separating internal node, if one half-node $h$ of $e$ lies on a connected component without open vertices, contracted boundary edges and tails in $T^{anc}$, we add $h$ to $T^{anc}$. In this way we obtain a stable graded $r$-spin graph $\operatorname{Detach}_e \Gamma$. We should comment that this solution is good enough for our needs (which are the topological recursion relations in the sequel \cite{TZ2}, which involve the closed extended $r$-spin theory \cite{BCT_Closed_Extended}), but perhaps there is another solution for higher genus which is more geometrically motivated.

 If $h$ is a contracted boundary tail of $\Gamma$, we define $\operatorname{Detach}_h\Gamma$ to be the graph obtained by moving $h$ from $H^{CB}$ into $T^{NCB}$. In the case where $h$ is the unique contracted boundary tail lying on a closed vertex, we add $h$ to $T^{anc}$.

 Note that unless $\Gamma$ is genus-zero, the detaching of $\Gamma$ along two different edges (or contracted boundary tails) in different orders might not be the same. This is because a non-separating internal edge might become separating after detaching.

\subsection{The Witten bundle and the relative cotangent line bundles}
\label{subsec:Wittenbundle}
We now consider the important bundles on the moduli spaces of stable graded $r$-spin disks.

\subsubsection{Genus-zero Witten bundle}

We first concentrate on $g=0,$ the disk case.
To define the $r$-spin theory, in the sequel \cite{TZ2} to this work, we need to consider the \textit{Witten bundle} on the moduli space. Roughly speaking, let $\pi: \mathcal{C} \to \Mbarstar_{0,k,l}^{1/r}$ be the universal curve and $\mathcal{S} \to \mathcal{C}$ be the twisted universal spin bundle with the universal Serre-dual bundle 
\begin{equation}\label{eq universal companion bundle}
\mathcal{J}:= \mathcal{S}^{\vee} \otimes \omega_{\pi},
\end{equation}
then we~define a sheaf
\begin{equation}
\label{eq Wittenbundledef}
{\mathcal{W}}:= (R^0\pi_*\mathcal{J})_+,
\end{equation}
where the subscript $+$ denote invariant sections under the universal involution $\widetilde{\phi}: \mathcal{J} \to \mathcal{J}$.  To be more precise, defining~${\mathcal{W}}$ by \eqref{eq Wittenbundledef} would require dealing with derived pushforward in the category of orbifold-with-corners. To avoid this technicality, we define ${\mathcal{W}}$ by pullback of the analogous sheaf from a subset of the closed moduli space $\Mbar_{0,k+2l}^{1/r}$; see \cite[Section 4.1]{BCT1}. 
$\mathcal W$ is actually a vector bundle, this follows from a direct Riemann--Roch computation showing that 
\begin{equation}\label{eq dimension not jump}
    R^0\pi_*\mathcal{S}=0.
\end{equation}

 On a component of the moduli space with internal twists $\{a_i\}$ and boundary twists $\{b_j\}$, the (real) rank of the Witten bundle is
\begin{equation}\label{eq rank of witten bundle}\frac{2 \sum_{i\in I} a_i + \sum_{j\in B} b_j - (r-2)}{r}.\end{equation}

\subsubsection{Dimension-jump loci and genus-one Witten bundle}
In the $g=1$ case, we still have  the universal curve $\pi: \mathcal{C} \to \Mbarstar_{1,k,l}^{1/r}$ and the twisted universal spin bundle $\mathcal{S} \to \mathcal{C}$   with the universal Serre-dual bundle 
$\mathcal{J}:= \mathcal{S}^{\vee} \otimes \omega_{\pi}$
as in $g=0$ case. However,
\eqref{eq dimension not jump} is not true in this case. What we do is to remove the the `dimension-jump locus' $\mathcal Z^{dj}$, \textit{i.e.,} the support of $R^0\pi_*\mathcal S$,  from the $\Mbarstar_{1,k,l}^{1/r}$, and define 
\begin{equation*}
{\mathcal{W}}:= (R^0\pi_*\mathcal{J})_+,
\end{equation*}
as a vector bundle over $\Mbarstar_{1,k,l}^{1/r}\setminus \mathcal Z^{dj}$. 

\begin{dfn}
    For a $r$-spin (nodal) cylinder $\Sigma$, we say $\Sigma$ is \textit{dimension-jump} if there exists $\Sigma'\subseteq \Sigma$ such that
    \begin{itemize}
        \item $\Sigma'$ is a union of irreducible components of $\Sigma$;
        \item $\Sigma'$ is genus-one;
        \item the restriction of $S$ to $\Sigma'$ is a trivial line bundle.
    \end{itemize}
\end{dfn}
\begin{rmk}
    Note that $S\vert_{\Sigma'}$ is trivial implies that all the internal and boundary markings on $\Sigma'$ have twist $0$. Moreover, all the half-nodes on $\Sigma'$ connecting $\Sigma'$ to $\Sigma\setminus\Sigma'$ have twist $0$.
\end{rmk}

\begin{dfn}
    We define the dimension-jump loci inside $\Mbarstar^{1/r}_{1,k,l}$ to be the 
    \begin{equation*}
         \mathcal Z^{dj}:=\left\{[\Sigma]\in \Mbarstar^{1/r}_{1,k,l} \colon \Sigma \text{ is dimension-jump}\right\}.
    \end{equation*}

    We write 
    \begin{equation*}
        \overline{\mathcal{QM}^*}_{1,k,l}:=\Mbarstar^{1/r}_{1,k,l}\setminus \mathcal Z^{dj};
    \end{equation*}
    for a genus-one $r$-spin graph $\Gamma$, we write 
    \begin{equation*}
        \mathcal Z^{dj}_{\Gamma}:=\mathcal Z^{dj}\cap \Mbar_{\Gamma}\quad\oQMb_\Gamma:=\Mbar_\Gamma\setminus \mathcal Z^{dj}_{\Gamma}.
    \end{equation*}
    For genus-zero $r$-spin graph $\Gamma$, we formally write $\oQMb_\Gamma:=\Mbar_\Gamma$.
\end{dfn}

\begin{ex}\label{exa dim jump non separating}
    Note that all non-separating boundary nodes for $\Sigma$ in $\mathcal Z^{dj}$ are Ramond nodes (whose half-nodes have twist $r-1$), therefore for any $\Gamma$ with a non-separating NS edge with have $\oQMb_{\Gamma}=\Mbar_{\Gamma}$.     
   Moreover, when $r$ is even, by item \ref{it Ramond boundary node} of Proposition \ref{prop lifting}, Ramond boundary nodes can not appear; therefore, we have $\mathcal Z^{dj}=\emptyset$ when $r$ is even.
\end{ex}

\begin{ex}\label{exa dim jump rank zero}
    Let $\Delta$ be a smooth $r$-spin graph with a single genus-one vertex, $k$ boundary and $l$ internal markings with twist zero. Then topologically the moduli $\Mbar_\Delta$ is the disjoint union of $r$ copies  $\Mbar_{\operatorname{for}_{r}(\Gamma)}\subset \Mbar_{1,k,l}$, where $\operatorname{for}_{r}(\Gamma)$ is the dual graph forgetting the $r$-spin structure on $\Gamma$. If $r$ is even, we have $\oQMb_\Delta=\Mbar_{\Delta}$. If $r$ is odd, then the dimension-jump loci $\mathcal Z^{dj}\cap\Mbar_\Delta$ is one copy of these $\Mbar_{\operatorname{for}_{r}(\Gamma)}$. Therefore in this case $\oQMb_\Delta$ is the disjoint union of $r-1$ copies  $\Mbar_{\operatorname{for}_{r}(\Gamma)}$.
\end{ex}

We define the Witten $\mathcal W$ bundle over $\overline{\mathcal{QM}^*}_{1,k,l}$ as 
$${\mathcal{W}}:= (R^0\pi_*\mathcal{J})_+.$$
On a non-empty component of the moduli space which parametrizes $r$-spin cylinders with internal twists $\{a_i\}$ and boundary twists $\{b_j\}$, the (real) rank of the Witten bundle is
\begin{equation}\label{eq rank of witten bundle g=1}\frac{2 \sum_{i\in I} a_i + \sum_{j\in B} b_j }{r}.\end{equation}

\begin{dfn}
    We say a genus-one graded $r$-spin graph $\Delta$ is \textit{pre-dimension-jump} if there exists a genus-one vertex $v\in V(\Delta)$ such that all the half-edges of $v$ have twist zero.

     We say a genus-one graded $r$-spin graph $\Delta$ is \textit{completely-dimension-jump} if
     \begin{enumerate}
         \item all the edges in $E^B_{nsp}$ are Ramond;
         \item the graph obtained after smoothing all the edges in $E^B_{nsp}$ is pre-dimension-jump.
     \end{enumerate}
\end{dfn}
For a pre-dimension-jump graph $\Delta$, we define
$$
\oJMb_\Delta:=\mathcal{Z}^{dj}\cap \Mbar_{\Delta}.
$$
Note that $\Mbar_{\Delta}$ is topologically the disjoint union of $r$ copies of  $\oJMb_\Delta$.

For a completely-dimension-jump graph $\Delta$, we define $\oJMb_\Delta:=\Mbar_{\Delta}\subseteq \mathcal Z^{dj}$.

We say $\Delta$ is \textit{dimension-jump} if $\Delta$ is either pre-dimension-jump or completely-dimension-jump. When $\Delta$ is not dimension-jump, we set $\oJMb_\Delta:=\emptyset$.

\subsubsection{Relative cotangent line bundles}
Other important line bundles in open $r$-spin theory are the relative cotangent line bundles at internal marked points.  These line bundles have already been defined on the moduli space $\Mbar_{0,k,l}$ of stable marked disks (without spin structure) in \cite{PST14}, as the line bundles with fibre $T^*_{z_i}\Sigma$. Equivalently, these line bundles are the pullback of the usual relative cotangent line bundles ${\mathbb{L}}_i\to\Mbar_{g,k+2l}$ under the doubling map $\Mbar_{g,k,l} \to \Mbar_{g,k+2l}$ that sends $\Sigma$ to $C=\Sigma\sqcup_{\partial\Sigma}\overline{\Sigma}$.  The bundle $\mathbb{L}_i\to\Mbarstar_{g,k,l}^{1/r}$ is the pullback of this relative cotangent line bundle on $\Mbar_{g,k,l}$ under the morphism $\text{For}_{\text{spin}}$ that forgets the spin structure.  Note that $\mathbb{L}_i$ is a complex line bundle, hence it carries a canonical orientation.

\subsection*{Decomposition properties of the Witten bundle}

In \cite{BCT1} the genus-zero Witten bundle is proven to satisfy certain decomposition properties along nodes, and the argument used there applies to the genus-one Witten bundle without any change.  We state the analogue properties for genus-one here, further details and proofs can be found in \cite[Section 4.2]{BCT1}.

Given a stable graded $r$-spin graph $\Gamma$ of genus-$g$ for $g=0$ or $1$, let $\widehat{\Gamma}$ be obtained by detaching either an edge or a contracted boundary tail of $\Gamma$.  In order to state the decomposition properties of the Witten bundle, we need the morphisms
\begin{equation}
\label{eq Wittendecompsequence genus one}
\oQMb_{\widehat{\Gamma}}^{1/r} \xleftarrow{q} \Mbar_{\widehat{\Gamma}} \times_{\Mbar_{\Gamma}} \oQMb_{\Gamma}^{1/r} \xrightarrow{\mu} \oQMb_{\Gamma}^{1/r} \xrightarrow{i_{\Gamma}} \overline{\mathcal{QM}^*}_{g,k,l}^{1/r},
\end{equation}
where $\Mbar_{\Gamma} \subseteq \Mbar_{g,k,l}$ is the moduli space of marked surfaces (without $r$-spin structure) corresponding to the dual graph $\Gamma$.  The morphism $q$ is defined by sending the $r$-spin structure $S$ to the $r$-spin structure $\widehat{S}$ defined by \eqref{eq normalize S}; it has degree one but is not an isomorphism because it does not induce an isomorphism between isotropy groups.  The morphism $\mu$ is the projection to the second factor in the fibre product; it is a surjective morphism, and is an isomorphism when $\Gamma$ has no non-separating edges.  The morphism $i_{\Gamma}$ is the inclusion.

We denote by ${\mathcal{W}}$ and $\widehat{{\mathcal{W}}}$ the Witten bundles on $\overline{\mathcal{QM}^*}_{g,k,l}^{1/r}$ and $\oQMb_{\widehat{\Gamma}}^{1/r}$, the decomposition properties below show how these bundles are related under pullback via the morphisms \eqref{eq Wittendecompsequence genus one}.  

\begin{pr}
\label{prop decomposition g=0,1}
Let $\Gamma$ be a genus-zero stable or genus-one graded $r$-spin graph with a single edge $e$, and let $\widehat{\Gamma}$ be the detaching of $\Gamma$ along $e$.  Then the Witten bundle decomposes as follows:
\begin{enumerate}
\item\label{it NS} If $e$ is a Neveu--Schwarz edge, then \begin{equation}\label{eq NSdecompses}
\mu^*i_{\Gamma}^*{\mathcal{W}} = q^*\widehat{{\mathcal{W}}}.
\end{equation}

\item\label{it decompose Ramond boundary edge} If $e$ is a Ramond boundary edge, then there is an exact sequence
\begin{equation}
\label{eq decompose}0 \to \mu^*i_{\Gamma}^*{\mathcal{W}} \to q^*\widehat{{\mathcal{W}}} \to \underline{\mathbb{R}_+} \to 0,
\end{equation}
where $\underline{\mathbb{R}_+}$ is a trivial real line bundle.

\item If $e$ is a Ramond internal edge connecting two closed genus-zero vertices, write $q^*\widehat{\mathcal{W}} = \widehat{\mathcal{W}}_1 \boxplus \widehat{\mathcal{W}}_2$, where $\widehat{\mathcal{W}}_1$ is the Witten bundle on the component containing a contracted boundary tail or the anchor of $\Gamma$, and $\widehat{\mathcal{W}}_2$ is the Witten bundle on the other component.  Then there is an exact sequence
\begin{equation}
\label{eq decompose2}
0 \to \widehat{\mathcal{W}}_2 \to \mu^*i_{\Gamma}^*{\mathcal{W}} \to \widehat{\mathcal{W}}_1 \to 0.
\end{equation}
Furthermore, if $\widehat\Gamma'$ is defined to agree with $\widehat\Gamma$ except that the twist at each Ramond tail is $r-1$, and $q': \Mbar_{\widehat{\Gamma}} \times_{\Mbar_{\Gamma}} \Mbar_{\Gamma}^{1/r} \to \Mbar_{\widehat\Gamma'}^{1/r}$ is defined analogously to $q$, then there is an exact sequence
\begin{equation}
\label{eq decompose3}
0 \to \mu^*i_{\Gamma}^*{\mathcal{W}} \to (q')^*\widehat{{\mathcal{W}}}' \to {\underline{\mathbb{C}}}^{1/r} \to 0,
\end{equation}
where $\widehat{{\mathcal{W}}}'$ is the Witten bundle on $\Mbar_{\widehat\Gamma'}^{1/r}$ and ${\underline{\mathbb{C}}}^{1/r}$ is a line bundle whose $r$-th power is trivial.

\item If $e$ is a separating Ramond internal edge connecting an open vertex to a closed vertex, write $q^*\widehat{\mathcal{W}} = \widehat{\mathcal{W}}_1 \boxplus \widehat{\mathcal{W}}_2$, where $\widehat{\mathcal{W}}_1$ is the Witten bundle on the open component (defined via $\widehat{\mathcal{S}}|_{\mathcal{C}_1}$) and $\widehat{\mathcal{W}}_2$ is the Witten bundle on the closed component.  Then the exact sequences \eqref{eq decompose2} and \eqref{eq decompose3} both hold.
\end{enumerate}
Analogously, for genus-zero $\Gamma$ which has a single vertex, no edges, and a contracted boundary tail~$t$, and $\widehat{\Gamma}$ is the detaching of $\Gamma$ along~$t$, then there is a decomposition property:
\begin{enumerate}
\setcounter{enumi}{4}
\item\label{it decompose cb tail} If ${\mathcal{W}}$ and $\widehat{\mathcal{W}}$ denote the Witten bundles on $\Mbarstar_{0,k,l}^{1/r}$ and $\Mbar_{\widehat\Gamma}^{1/r}$, respectively, then the sequence \eqref{eq decompose} holds.
\end{enumerate}
\end{pr}

\begin{rmk}\label{rmk decompose NS boundary node}
If the edge $e$ is a Neveu--Schwarz boundary edge, then the map $q$ is an isomorphism, and in this case, the proposition implies that the Witten bundle pulls back under the gluing morphism $\mu\circ q^{-1}\colon \oQMb_{\widehat\Gamma}^{1/r} \to \overline{\mathcal{QM}^*}_{g,k,l}^{1/r}$. Note that $q$ is not an isomorphism in general, because it does not induce
an isomorphism on automorphism groups (see \cite[Remark 4.5]{BCT1}).

\end{rmk}

\begin{rmk}
The Witten bundle decomposes in a more straightforward way along Neveu--Schwarz nodes than along Ramond nodes. This occurs because the NS nodes are nodes at which the isotropy group of $C$ acts nontrivially on the fibre of $S$.  Given that sections of an orbifold line bundle must be invariant under the action of the isotropy group, nontriviality of the action results in the vanishing of sections at such nodes. This leads to a splitting in the normalization exact sequence associated with $S$.
\end{rmk}

\section{Orientation}\label{sec:orientation}
In this section, we construct a canonical relative orientation of the Witten bundles over moduli spaces of graded $r$-spin disks and cylinders. 

\subsection{Orientation of moduli space}
  
\subsubsection{Genus-zero case}
  
The orientation of $\Mbarstar^{1/r}_{0,\bar{B},I}$ is studied in \cite[Section 3.3]{BCT1}, which pull-back from orientation of $\Mbar_{0,\bar{B},I}$ via the forgetful morphism. We summarize some of the properties proven there.
\begin{prop}\label{prop:orientation moduli}
    There is an unique choice of orientations $\tilde{\mathfrak o}^{x}_{0,\bar{B},I}$ of the spaces $\Mbar_{0,\bar{B},I}$ for all $\bar{B},I$ and $x\in B$, with the following properties:
    \begin{enumerate}
        \item In the zero-dimensional case where $\vert I \vert = \vert B \vert = 1$, the orientations are positive, while when $\vert B \vert = 3$ and $\vert I \vert = 0$, the orientations are negative.
        \item When $\lvert B\rvert\ge 1$, let $\sigma_2\colon B \to B$ be the cyclic order on $B$ encoded in $\bar B$, we have $\tilde{\mathfrak o}^{\sigma_2(x)}_{0,\bar{B},I}=(-1)^{\vert B \vert-1}\tilde{\mathfrak o}^{x}_{0,\bar{B},I}$.

         \item \label{item genus zero moduli orientation forget internal marking}
        For any $a \in I$, the each fibre $F$ of the forgetful morphism forgetting the marking $a$
        $$
        \operatorname{For}_a\colon \mathcal M_{0,\bar{B},I} \to \mathcal M_{0,\bar{B},I\setminus\{a\}}
        $$
        has a canonical complex structure since it is a punctured disk, and we denote the complex orientation on the fibre by $o_F$. Then we have $$\tilde{\mathfrak o}^{x}_{0,\bar{B},I}=o_F\wedge \operatorname{For}_a^*\tilde{\mathfrak o}^{x}_{0,
        \bar{B},I\setminus\{a\}}.$$ 
        
        \item \label{item genus zero orientation moduli forget boundary marking}
        We denote by
         $$
        \operatorname{For}_{\sigma_2^{-1}(x)}\colon \mathcal M_{0,\bar{B},I} \to \mathcal M_{0,\bar{B}\setminus\{\sigma_2^{-1}(x)\},I}
        $$
        the forgetful morphism forgetting the boundary marking $\sigma_2^{-1}(x)$. The fibre $G$ of $\operatorname{For}_{\sigma_2^{-1}(x)}$ is isomorphic to an interval from $\sigma_2^{-1}(\sigma_2^{-1}(x))$ to $x$, we denote by $o_G$ the orientation of $G$ induced by the complex orientation of the disk. Then we have
        $$\tilde{\mathfrak o}^{x}_{0,\bar{B},I}=o_G\wedge \operatorname{For}_{\sigma_2^{-1}(x)}^*\tilde{\mathfrak o}^{x}_{0,\bar{B}\setminus\{\sigma_2^{-1}(x)\},I}.$$

        \item Let $\Gamma$ be the graph with two open vertices connected by an edge $e$, let $h_i$ denote the half-edges of $v_i$. Let $I_i$ be the sets of internal markings of $v_i$; we write $\bar{B}_1$, the boundary half-edges of $v_1$, in its cyclic order as $\bar{B}_1=\{b_{11},b_{12},\dots,b_{1k_1},h_1\}$; we also write $\bar{B}_2$, the boundary half-edges of $v_2$, in its cyclic order as $\bar{B}_2=\{h_2,b_{21},b_{22},\dots,b_{2k_2}\}$. Then the set of internal markings of $d_e \Gamma$ is $I=I_1\sqcup I_2$, and the set of the boundary markings of $d_e \Gamma$ written in cyclic order is $\bar{B}=\{b_{11},b_{12},\dots,b_{1k_1},b_{21},b_{22},\dots,b_{2k_2}\}$. Note that we have $\det(T\Mbar_{0,\bar{B},I})\big|_{\Mbar_\Gamma} = \det(N)\otimes \det( T\Mbar_{\Gamma}),$ where $N$ is the outward normal with canonical orientation $o_N$. Then
        \begin{equation}\label{eq orientation moduli open open}
        \tilde{\mathfrak{o}}^{b_{11}}_{0,\bar B,I}\big |_{\Mbar_\Gamma^{}}=(-1)^{(k_1-1)k_2}o_N\otimes
         \left(\tilde{\mathfrak{o}}^{b_{11}}_{0,\bar B_1,I_1}\boxtimes\tilde{\mathfrak{o}}^{h_2}_{0,\bar B_2,I_2}\right).
         \end{equation}

        \item

        Let $\Gamma$ be a graph with two vertices, an open vertex $v^o$ and a closed vertex $v^c$, connected by an edge $e$.  We denote by $I$ and $I^o$ the sets of internal half-edges of $d_e\Gamma$ and $v^o$, and by $B$ the common set of boundary half-edges of $d_e\Gamma$ and $v^o$. We have $\det(T\Mbar_{0,\bar{B},I})\big|_{\Mbar_\Gamma} = \det(N)\otimes \left(\det(T\Mbar_{v^c})\boxtimes \det(T\Mbar_{v^o})\right)$, where $N$ is again the normal bundle.  Then, for every boundary marking $b\in B$, we have
        \begin{equation}\label{eq orientation moduli closed open}
        \tilde{\mathfrak{o}}^b_{0,\bar B,I}|_{\Mbar_\Gamma} = o_N\otimes (\tilde{\mathfrak{o}}^b_{0,\bar B,I^o}\boxtimes \tilde{\mathfrak{o}}_{v^c}),
        \end{equation}
        where $o_N$ and $\tilde{\mathfrak{o}}_{v^c}$ are the canonical complex orientations.

 \end{enumerate}
\end{prop}

We also denote by $\tilde{\mathfrak o}^{x}_{0,\bar{B},I}$ the orientation pulled back to $\Mbarstar^{1/r}_{0,\bar{B},I}$ via the forgetful morphism.
\begin{rmk}
    In the case $B=\emptyset$, the orientation $\tilde{\mathfrak o}^{x}_{0,\bar{B},I}$ depends only on $\bar{B}$ and $I$, the superscript $x$ is just for the sake of maintaining consistency of notation.
\end{rmk}

\subsubsection{Genus-one case} \label{sec orientation for moduli g=1}

We construct the orientation of $\Mbarstar^{1/r}_{1,\{\bar{B}^\alpha,\bar{B}^\beta\},I}$ in this subsection. We start from the orientation of the moduli space  $\mathcal M_{1,\{\bar B^{\alpha}, \bar B^\beta\},I}$ of cylinders with out $r$-spin structure.

We write $\bar B^\alpha=\overline{\{b^\alpha_1,b^\alpha_2,\dots,b^\alpha_{k_\alpha}\}}$ and $\bar B^\beta=\overline{\{b^\beta_1,b^\beta_2,\dots,b^\beta_{k_\beta}\}}$ the sets of boundary markings on each boundaries in their cyclic order. We also denote by $I=\{a_1,\dots,a_l\}$ the set of internal markings. Each cylinder $(C,\Sigma)$ in $\mathcal M_{1,\{\bar B^{\alpha}, \bar B^\beta\},I}$ can be regarded as $\mathbb C_z=\mathbb R_x \times \mathbb R_y$ quotiented by the lattice generated by $(1,0)$ and $(2\tau,0)$, where a fundamental domain of $\Sigma$ is $\{0\le x <1\}\times \{0\le y \le \tau\}$. If we fix an order of $\alpha$ and $\beta$, we can construct a space
$$
\mathcal T^{\alpha,\beta}_{1,(\bar{B}^\alpha,\bar{B}^\beta),I}:=\{(\tau,z_1,\dots,z_l,x^\alpha_1,\dots,x^\alpha_{k_\alpha},x^\beta_1,\dots,x^\beta_{k_\beta})\}
$$
where
\begin{itemize}
    \item $\tau \in \mathbb R^{>0}$;
    \item $z_1,\dots,z_l$ are different points in $\mathring{\Sigma}\subset \frac{\mathbb C}{\left<(1,0),(0,2\tau)\right>}$,
    \item $x^\alpha_1,\dots,x^\alpha_{k_\alpha}\in S^{1,\alpha}:=\frac{\mathbb R^\alpha}{\left<1\right>}$, and they are in cyclic order with respect to natural order induced by the real line $\mathbb R^\alpha$;
     \item $x^\beta_1,\dots,x^\beta_{k_\beta}\in S^{1,\beta}:=\frac{\mathbb R^\beta}{\left<1\right>}$, and they are in reversed cyclic order with respect to natural order induced by the real line $\mathbb R^\beta$.
\end{itemize}

There is a $\mathbb R_t=\mathbb R_t^{\alpha,\beta}$ action on $\mathcal T^{\alpha,\beta}_{1,(\bar{B}^\alpha,\bar{B}^\beta),I}$, where for $t\in \mathbb R_t$, the action is given by
$$
\left(t,(\tau,z_1,\dots,z_l,x^\alpha_1,\dots,x^\alpha_{k_1},x^\beta_1,\dots,x^\beta_{k_2})\right)\mapsto(\tau,z_1+t,\dots,z_l+t,x^\alpha_1+t,\dots,x^\alpha_{k_1}+t,x^\beta_1+t,\dots,x^\beta_{k_2}+t).
$$
We denote the quotient of this action by
$$
\mathcal M^{\alpha,\beta}_{1,(\bar{B}^\alpha,\bar{B}^\beta),I}:=\mathcal T^{\alpha,\beta}_{1,(\bar{B}^\alpha,\bar{B}^\beta),I}\big/ \mathbb R_t.
$$

There is a map 
$$
\pi_{\mathcal T \to \mathcal M}\colon \mathcal T^{\alpha,\beta}_{1,(\bar{B}^\alpha,\bar{B}^\beta),I}\to \mathcal M_{1,\{\bar{B}^\alpha,\bar{B}^\beta\},I}
$$ which send $(\tau,z_1,\dots,z_l,x^\alpha_1,\dots,x^\alpha_{k_1},x^\beta_1,\dots,x^\beta_{k_2})$ to the cylinder represented by $\frac{\mathbb C}{\left<(1,0),(0,2\tau)\right>}$, with internal markings locate at $z_1,\dots,z_l$, and boundary markings locate at $(x^\alpha_1,0),\dots,(x^\alpha_{k_\alpha},0)$ and $(x^\beta_1,\tau),\dots,(x^\beta_{k_\alpha},\tau)$. Notice that the orientation on $\frac{\mathbb R_x\times \{y=0\}}{\left< (1,0)\right>}$ induced from the complex orientation on $\Sigma$ coincides with th one induced from the natural orientation on $\mathbb R_x$, while the orientation on $\frac{\mathbb R_x\times \{y=\tau\}}{\left< (1,0)\right>}$ induced from the complex orientation on $\Sigma$ is the reverse of one induced from the natural orientation on $\mathbb R_x$. Therefore, both  $(x^\alpha_1,0),\dots,(x^\alpha_{k_\alpha},0)$ and $(x^\beta_1,\tau),\dots,(x^\beta_{k_\alpha},\tau)$ are in cyclic order induce by the complex orientation of $\Sigma$ on the corresponding boundary as required.

Notice that $\pi_{\mathcal T \to \mathcal M}$ is invariant under the action of $\mathbb R_t$, it induces a map $$\operatorname{For}_{\alpha,\beta}\colon \mathcal M^{\alpha,\beta}_{1,(\bar{B}^\alpha,\bar{B}^\beta),I}\to \mathcal M_{1,\{\bar{B}^\alpha,\bar{B}^\beta\},I}.$$
In fact, the space $ \mathcal M^{\alpha,\beta}_{1,(\bar{B}^\alpha,\bar{B}^\beta),I}$ parametrizes cylinders $(C,\Sigma)$ in $ \mathcal M_{1,\{\bar{B}^\alpha,\bar{B}^\beta\},I}$ with an extra datum: the two boundaries  of $\Sigma$ are labelled by $\alpha$ and $\beta$. Therefore, in the case $\bar{B}^\alpha=\bar{B}^\beta=\emptyset$, the map $\operatorname{For}_{\alpha,\beta}$ is a double covering; in the case where at least one of $\bar{B}^\alpha$ or $\bar{B}^\beta$ is non-empty, $\operatorname{For}_{\alpha,\beta}$ is an isomorphism since for any cylinder in $ \mathcal M_{1,\{\bar{B}^\alpha,\bar{B}^\beta\},I}$, the label of two boundaries is already determined by distribution of boundary markings.

If we exchange the order of $\alpha$ and $\beta$, we have a map
$$
\operatorname{Exch}\colon\mathcal T^{\alpha,\beta}_{1,(\bar{B}^\alpha,\bar{B}^\beta),I}\to \mathcal T^{\beta,\alpha}_{1,(\bar{B}^\beta,\bar{B}^\alpha),I}
$$
which maps
$$
(\tau,z_1,\dots,z_l,x^\alpha_1,\dots,x^\alpha_{k_1},x^\beta_1,\dots,x^\beta_{k_2})$$ to 
\begin{equation}\label{eq def of exchange before quotient}
(\tau,-z_1+\tau\sqrt{-1},\dots,-z_l+\tau\sqrt{-1},-x^\beta_1,\dots,-x^\beta_{k_2},-x^\alpha_1,\dots,-x^\alpha_{k_1}).
\end{equation}
Since $\operatorname{Exch}$ commutes with the actions of $\mathbb R_t^{\alpha,\beta}$ and $\mathbb R_t^{\beta,\alpha}$ after the a morphism $t\mapsto -t$ between $\mathbb R_t^{\alpha,\beta}$ and $\mathbb R_t^{\beta,\alpha}$, it induces a map $$\operatorname{Exch}\colon\mathcal M^{\alpha,\beta}_{1,(\bar{B}^\alpha,\bar{B}^\beta),I}\to \mathcal M^{\beta,\alpha}_{1,(\bar{B}^\beta,\bar{B}^\alpha),I},$$ 

In the case where at least one of $\bar{B}^\alpha$ or $\bar{B}^\beta$ is non-empty, $For_{\alpha,\beta}\circ \operatorname{Exch} \circ For^{-1}_{\alpha,\beta}$ coincides with the identical map of $\mathcal M_{1,\{\bar{B}^\alpha,\bar{B}^\beta\},I}$.

In the case where $\bar{B}^\alpha=\bar{B}^\beta=\emptyset$, we have another morphism between $\mathcal T^{\alpha,\beta}_{1,(\emptyset^\alpha,\emptyset^\beta),I}$ and $\mathcal T^{\beta,\alpha}_{1,(\emptyset^\beta,\emptyset^\alpha),I}$ sending $(\tau,z_1,\dots,z_l)$ to $(\tau,z_1,\dots,z_l)$, which induce an isomorphism 
$$\Id_{\emptyset,\emptyset} \colon \mathcal M^{\alpha,\beta}_{1,(\bar{B}^\beta,\bar{B}^\alpha),I} \to \mathcal M^{\beta,\alpha}_{1,(\bar{B}^\beta,\bar{B}^\alpha),I};$$
the composition $\operatorname{Exch}\circ \Id_{\emptyset,\emptyset}$
 coincides with the map exchanging the two covers of $\operatorname{For}_{\alpha,\beta}$.

For any choice of $x^\alpha_i\in \bar B^\alpha$ and $x^\alpha_j\in \bar B^\beta$ (in the case $\bar B^\alpha$ or $\bar B^\beta$ is empty we choose them as formal notations), we define a form on $\mathcal T^{\alpha,\beta}_{1,(\bar{B}^\alpha,\bar{B}^\beta),I}$ by
\begin{equation}\label{eq definition orientaion moduli genus one before quotient}
\begin{split}
    \tilde{\mathfrak o}^{x^\alpha_i,x^\beta_j,\mathcal T}_{1,(\bar{B}^\alpha,\bar{B}^\beta),I}:=& dx^\alpha_{i-1}\wedge dx^\alpha_{i-2}\wedge\dots\wedge dx^\alpha_{2}\wedge dx^\alpha_{1} \wedge dx^\alpha_{k_\alpha} \wedge\dots dx^\alpha_{i+1}\wedge dx^\alpha_{i}\\
    &\wedge dx^\beta_{j-1}\wedge dx^\beta_{j-2}\wedge\dots\wedge dx^\beta_{2}\wedge dx^\beta_{1} \wedge dx^\beta_{k_\beta} \wedge\dots dx^\beta_{j+1}\wedge dx^\beta_{j}\\
    &\wedge \frac{\sqrt{-1}}{2}dz_1\wedge d\bar{z}_1\wedge  \frac{\sqrt{-1}}{2}dz_2\wedge d\bar{z}_2\wedge \dots \wedge \frac{\sqrt{-1}}{2}dz_l\wedge d\bar{z}_l \wedge d\tau.
\end{split}
    \end{equation}
Since locally $\mathcal T^{\alpha,\beta}_{1,(\bar{B}^\alpha,\bar{B}^\beta),I}$ is the product of $\mathcal M^{\alpha,\beta}_{1,(\bar{B}^\alpha,\bar{B}^\beta),I}$ and $\mathbb R_t$, if we denote by $o_{\mathbb R_t}$ the canonical orientation on $\mathbb R_t$, the orientation $\tilde{\mathfrak o}^{x^\alpha_i,x^\beta_j,\mathcal T}_{1,(\bar{B}^\alpha,\bar{B}^\beta),I}$ on $\mathcal T^{\alpha,\beta}_{1,(\bar{B}^\alpha,\bar{B}^\beta),I}$ induces an orientation $\tilde{\mathfrak o}^{x^\alpha_i,x^\beta_j}_{1,(\bar{B}^\alpha,\bar{B}^\beta),I}$ on $\mathcal M^{\alpha,\beta}_{1,(\bar{B}^\alpha,\bar{B}^\beta),I}$ satisfying
$$\tilde{\mathfrak o}^{x^\alpha_i,x^\beta_j,\mathcal T}_{1,(\bar{B}^\alpha,\bar{B}^\beta),I}=\tilde{\mathfrak o}^{x^\alpha_i,x^\beta_j}_{1,(\bar{B}^\alpha,\bar{B}^\beta),I}\wedge o_{\mathbb R_t}.$$

The orientation $\tilde{\mathfrak o}^{x^\alpha_i,x^\beta_j}_{1,(\bar{B}^\alpha,\bar{B}^\beta),I}$ on $\mathcal M^{\alpha,\beta}_{1,(\bar{B}^\alpha,\bar{B}^\beta),I}$ extends to the entire compactification $\Mbar^{\alpha,\beta}_{1,(\bar{B}^\alpha,\bar{B}^\beta),I}$ as $\Mbar^{\alpha,\beta}_{1,(\bar{B}^\alpha,\bar{B}^\beta),I}\setminus \mathcal M^{\alpha,\beta}_{1,(\bar{B}^\alpha,\bar{B}^\beta),I}$ consists of boundaries and codimensional-two strata.

\begin{prop}\label{prop orientation moduli g=1}
    The orientations $\tilde{\mathfrak o}^{x^\alpha_i,x^\beta_j}_{1,(\bar{B}^\alpha,\bar{B}^\beta),I}$ satisfy the following property.
    \begin{enumerate}
        \item \label{item orientation of moduli when change starting point genus one}
        We have
        \begin{equation}\label{eq orientation of moduli when change starting point genus one}
        \tilde{\mathfrak o}^{x^\alpha_{i+1},x^\beta_j}_{1,(\bar{B}^\alpha,\bar{B}^\beta),I}=(-1)^{\lvert B^\alpha \rvert+1}\tilde{\mathfrak o}^{x^\alpha_{i},x^\beta_j}_{1,(\bar{B}^\alpha,\bar{B}^\beta),I} \text{ and }
        \tilde{\mathfrak o}^{x^\alpha_{i},x^\beta_{j+1}}_{1,(\bar{B}^\alpha,\bar{B}^\beta),I}=(-1)^{\lvert B^\beta \rvert+1}\tilde{\mathfrak o}^{x^\alpha_{i},x^\beta_j}_{1,(\bar{B}^\alpha,\bar{B}^\beta),I}.
        \end{equation}
        \item \label{item orientation of moduli when exchanging boundary genus one} We have
        \begin{equation}\label{eq orientation moduli when exchange boundaries genus one}
            \tilde{\mathfrak o}^{x^\alpha_i,x^\beta_{j}}_{1,(\bar{B}^\alpha,\bar{B}^\beta),I}=(-1)^{(\lvert B^\alpha \rvert+1)(\lvert B^\beta \rvert+1)}\operatorname{Exch}^*\tilde{\mathfrak o}^{x^\beta_{j},x^\alpha_i}_{1,(\bar{B}^\beta,\bar{B}^\alpha),I}.
        \end{equation}

        \item \label{item orientation of moduli without boundary points id}
        When $B^\alpha=B^\beta=\emptyset$, we have
        \begin{equation}\label{eq orientation of moduli without boundary points id}
            \tilde{\mathfrak o}^{x^\beta,x^\alpha}_{1,(\bar{B}^\beta,\bar{B}^\alpha),I}=\Id_{\emptyset,\emptyset}^*\tilde{\mathfrak o}^{x^\alpha,x^\beta}_{1,(\bar{B}^\alpha,\bar{B}^\beta),I},
            \end{equation}
            where $x^\beta$ and $x^\alpha$ are just formal notations.        
        \item \label{item genus one moduli orientation forget internal marking}
        For any $a \in I$, the each fibre $F$ of the forgetful morphism forgetting the marking $a$
        $$
        \operatorname{For}_a\colon \mathcal M^{\alpha,\beta}_{1,(\bar{B}^\alpha,\bar{B}^\beta),I} \to \mathcal M^{\alpha,\beta}_{1,(\bar{B}^\alpha,\bar{B}^\beta),I\setminus\{a\}}
        $$
        has a canonical complex structure since it is a punctured cylinder, and we denote the complex orientation on the fibre by $o_F$. Then we have $$\tilde{\mathfrak o}^{x^\alpha_{i},x^\beta_j}_{1,(\bar{B}^\alpha,\bar{B}^\beta),I}=o_F\wedge \operatorname{For}_a^*\tilde{\mathfrak o}^{x^\alpha_{i},x^\beta_j}_{1,(\bar{B}^\alpha,\bar{B}^\beta),I\setminus\{a\}}.$$

        \item \label{item genus one orientation moduli forget boundary marking}
        Since $(x^\alpha_{i-1},0) \in \bar{B}^\alpha$ is a boundary marking, we denote by
         $$
        \operatorname{For}_{\alpha,i-1}\colon \mathcal M^{\alpha,\beta}_{1,(\bar{B}^\alpha,\bar{B}^\beta),I} \to \mathcal M^{\alpha,\beta}_{1,(\bar{B}^\alpha\setminus\{(x^\alpha_{i-1},0)\},\bar{B}^\beta),I}
        $$
        the forgetful morphism forgetting this boundary marking. This fibre $G$ of $\operatorname{For}_{\alpha,i-1}$ is isomorphic to an interval from $(x^\alpha_{i-2},0)$ to $(x^\alpha_{i},0)$, we denote by $o_G$ the orientation of $G$ induced by the complex orientation of the cylinder. Then we have
        $$\tilde{\mathfrak o}^{x^\alpha_{i},x^\beta_j}_{1,(\bar{B}^\alpha,\bar{B}^\beta),I}=o_G\wedge \operatorname{For}_{\alpha,i-1}^*\tilde{\mathfrak o}^{x^\alpha_{i},x^\beta_j}_{1,(\bar{B}^\alpha\setminus\{(x^\alpha_{i-1},0)\},\bar{B}^\beta),I}.$$ 
        
        \item \label{item orientation moduli open open g=1}
        Let $\Gamma$ be the graph with two open vertices: genus-zero $v_1$ and genus-one $v_2$, connected by an separating boundary edge $e$. Let $h_i$ denote the half-edges of $v_i$. Let $I_i$ be the sets of internal markings of $v_i$; we write $\bar{B}_1$, the boundary half-edges of $v_1$, in its cyclic order as $\bar{B}_1=\{b_{11},b_{12},\dots,b_{1k_1},h_1\}$; we also write $\bar{B}^1_2$ and $\bar{B}^2_2$, the sets boundary half-edges of $v_2$ on each boundaries (where $h_1\in \bar{B}_{2}^1$), in their cyclic order as $\bar{B}_{2}^\alpha=\overline{\{h_2,b_{21}^1,b_{22}^1,\dots,b_{2k_2^1}^1\}}$ and $\bar{B}_{2}^\beta=\overline{\{b_{21}^2,b_{22}^2,\dots,b_{2k_2^2}^2\}}$. Then the set of internal markings of $d_e \Gamma$ is $I=I_1\sqcup I_2$, and the sets of boundary markings of $d_e \Gamma$ on each boundaries written in their cyclic order is $\bar{B}^\alpha=\overline{\{b_{11},b_{12},\dots,b_{1k_1},b_{21}^1,b_{22}^1,\dots,b_{2k_2^1}^1\}}$ and $\bar{B}^\beta=\bar{B}_{2}^2=\overline{\{b_{21}^2,b_{22}^2,\dots,b_{2k_2^2}^2\}}$. 
        Note that we have $\det(T\Mbar_{1,\{\bar{B}^\alpha,\bar B^\beta\},I})\big|_{\Mbar_\Gamma} = \det(N)\otimes \det( T\Mbar_{\Gamma}),$ where $N$ is the outward normal with canonical orientation $o_N$. Then \footnote{In the case $B^\alpha=B^\beta=\emptyset$, by abuse of notation, we also denote by $\Mbar_{\Gamma}$ its preimage in the double cover.}
        \begin{equation}\label{eq orientation moduli open open g=1}
        \tilde{\mathfrak{o}}^{b_{11},b_{21}^2}_{1,(\bar{B}^\alpha,\bar B^\beta),I}\big |_{\Mbar_\Gamma^{}}=(-1)^{(k_1-1)k_2^1}o_N\otimes
         \left(\tilde{\mathfrak{o}}^{b_{11}}_{0,\bar B_1,I_1}\boxtimes\tilde{\mathfrak{o}}^{h_2,b_{21}^2}_{1,(\bar B_2^\alpha, \bar B_2^\beta),I_2}\right).
         \end{equation}

         \item \label{item orientation moduli nonseperating}
          Let $\Gamma$ be a graph consisting of an open genus-zero vertex $v$ and a non-separating boundary edge $e$ connecting $v$ to itself, with two half-edges $h_1$ and $h_2$.  Let $I$ be the sets of internal markings of $v$; we write $\bar{B}_v$, the boundary half-edges of $v$, in its cyclic order as $\bar{B}_v=\overline{\{b_{11},b_{12},\dots,b_{1k_1},h_2,b_{21},b_{22},\dots,b_{2k_1},h_1\}}$. Then the set of internal markings of $d_e \Gamma$ is also $I$, and the sets of boundary markings of $d_e \Gamma$ on each boundaries written in their cyclic order is $\bar{B}^\alpha=\overline{\{b_{11},b_{12},\dots,b_{1k_1}\}}$ and $\bar{B}^\beta=\overline{\{b_{21},b_{22},\dots,b_{2k_1}\}}$. 
          Note that We have $\det(T\Mbar_{1,\{\bar{B}^\alpha,\bar B^\beta\},I})\big|_{\Mbar_\Gamma} = \det(N)\otimes \det( T\Mbar_{\Gamma}),$ where $N$ is the outward normal with canonical orientation $o_N$. Then 
        \begin{equation}\label{eq orientation genus one moduli non-separating}
        \tilde{\mathfrak{o}}^{b_{11},b_{21}}_{1,(\bar{B}^\alpha,\bar{B}^\beta),I}\big |_{\Mbar_\Gamma^{}}=
         (-1)^{k_1 k_2+k_1+k_2}o_N\otimes \tilde{\mathfrak{o}}^{b_{11}}_{0,\bar B_v,I}.
         \end{equation}

        \item \label{item orientation moduli closed open g=1}

        Let $\Gamma$ be a graph with two vertices, an open genus-one vertex $v^o$ and a closed genus-zero vertex $v^c$, connected by an separating internal edge $e$.  We denote by $I$ and $I^o$ the sets of internal half-edges of $d_e\Gamma$ and $v^o$, and by $\bar B^\alpha$, $\bar B^\beta$ the common sets of boundary half-edges of $d_e\Gamma$ and $v^o$ on each boundaries. We have $\det(T\Mbar_{1,\{\bar{B}^\alpha,\bar B^\beta\},I})\big|_{\Mbar_\Gamma} = \det(N)\otimes \left(\det(T\Mbar_{v^c})\boxtimes \det(T\Mbar_{v^o})\right)$, where $N$ is again the normal bundle.  Then, for any boundary marking $b_1\in B^\alpha$ and $b_2\in B^\beta$, we have
        \begin{equation}\label{eq orientation moduli closed open g=1}
        \tilde{\mathfrak{o}}^{b_1,b_2}_{1,(\bar{B}^\alpha,\bar B^\beta),I}|_{\Mbar_\Gamma} = o_N\otimes (\tilde{\mathfrak{o}}^{b_1,b_2}_{1,(\bar{B}^\alpha,\bar B^\beta),I^o}\boxtimes \tilde{\mathfrak{o}}_{v^c}),
        \end{equation}
        where $o_N$ and $\tilde{\mathfrak{o}}_{v^c}$ are the canonical complex orientations.

    \end{enumerate}
\end{prop}

\begin{proof}
    Item \ref{item orientation of moduli when change starting point genus one} follows from a similar property for $\tilde{\mathfrak o}^{x^\alpha_i,x^\beta_j}_{1,(\bar{B}^\alpha,\bar{B}^\beta),I}$, \textit{i.e.,} by \eqref{eq definition orientaion moduli genus one before quotient} they satisfy 
    $$\tilde{\mathfrak o}^{x^\alpha_{i+1},x^\beta_j,\mathcal T}_{1,(\bar{B}^\alpha,\bar{B}^\beta),I}=(-1)^{\lvert B^\alpha \rvert+1}\tilde{\mathfrak o}^{x^\alpha_{i},x^\beta_j,\mathcal T}_{1,(\bar{B}^\alpha,\bar{B}^\beta),I}
    \text{ and }\tilde{\mathfrak o}^{x^\alpha_{i},x^\beta_{j+1},\mathcal T}_{1,(\bar{B}^\alpha,\bar{B}^\beta),I}=(-1)^{\lvert B^\beta \rvert+1}\tilde{\mathfrak o}^{x^\alpha_{i},x^\beta_j,\mathcal T}_{1,(\bar{B}^\alpha,\bar{B}^\beta),I}.$$

    For item \ref{item orientation of moduli when exchanging boundary genus one}, by \eqref{eq def of exchange before quotient} and \eqref{eq definition orientaion moduli genus one before quotient} we have
    \begin{equation*}
        \begin{split}
            \operatorname{Exch}^*\tilde{\mathfrak o}^{x^\beta_j,x^\alpha_{i},\mathcal T}_{1,(\bar{B}^\alpha,\bar{B}^\beta),I}=
    &(-1)^{k_\alpha+k_\beta}dx^\beta_{j-1}\wedge dx^\beta_{j-2}\wedge\dots\wedge dx^\beta_{2}\wedge dx^\beta_{1} \wedge dx^\beta_{k_\beta} \wedge\dots dx^\beta_{j+1}\wedge dx^\beta_{j}\\
    & \wedge  dx^\alpha_{i-1}\wedge dx^\alpha_{i-2}\wedge\dots\wedge dx^\alpha_{2}\wedge dx^\alpha_{1} \wedge dx^\alpha_{k_\alpha} \wedge\dots dx^\alpha_{i+1}\wedge dx^\alpha_{i}\\
    &\wedge \frac{\sqrt{-1}}{2}dz_1\wedge d\bar{z}_1\wedge  \frac{\sqrt{-1}}{2}dz_2\wedge d\bar{z}_2\wedge \dots \wedge \frac{\sqrt{-1}}{2}dz_l\wedge d\bar{z}_l \wedge d\tau\\
            =&(-1)^{k_\alpha k_\beta +k_\alpha+k_\beta}\tilde{\mathfrak o}^{x^\alpha_{i},x^\beta_j,\mathcal T}_{1,(\bar{B}^\alpha,\bar{B}^\beta),I}\\
            =&(-1)^{\lvert B^\alpha \rvert\lvert B^\beta \rvert+\lvert B^\alpha \rvert+\lvert B^\beta\rvert}\tilde{\mathfrak o}^{x^\alpha_{i},x^\beta_j,\mathcal T}_{1,(\bar{B}^\alpha,\bar{B}^\beta),I}.
        \end{split}
    \end{equation*}
    On the other hand, the morphism $t\mapsto -t$ between $\mathbb R_t^{\alpha,\beta}$ and $\mathbb R_t^{\beta,\alpha}$ reverts the canonical orientations, which means we have an addition $-1$ when comparing the induced orientations on $\mathcal M^{\alpha,\beta}_{1,(\bar{B}^\alpha,\bar{B}^\beta),I}$, hence \eqref{eq orientation moduli when exchange boundaries genus one} holds.

    Item \ref{item orientation of moduli without boundary points id}, \ref{item genus one moduli orientation forget internal marking} and \ref{item genus one orientation moduli forget boundary marking} also follows from their analogue versions for orientations on $\mathcal T$, which can be deduced directly from the definition \eqref{eq definition orientaion moduli genus one before quotient}.

    Items \ref{item orientation moduli open open g=1}, \ref{item orientation moduli nonseperating} and \ref{item orientation moduli closed open g=1} can be proven using the same inductive argument as in the proof of \cite[Lemma 3.15]{BCT1}.
\end{proof}

We denote by $\Mbarstar^{1/r,\alpha,\beta}_{1,(\bar{B}^\alpha,\bar{B}^\beta),I}$ the fibre product of  $\Mbarstar^{1/r}_{1,\{\bar{B}^\alpha,\bar{B}^\beta\},I}$  and  $\Mbar^{\alpha,\beta}_{1,(\bar{B}^\alpha,\bar{B}^\beta),I}$ over $\Mbar^{}_{1,\{\bar{B}^\alpha,\bar{B}^\beta\},I}$. It parametrizes graded $r$-spin cylinders with labelled boundary components. 
Since the morphism from $\Mbarstar^{1/r,\alpha,\beta}_{1,(\bar{B}^\alpha,\bar{B}^\beta),I}$ to $\Mbar^{\alpha,\beta}_{1,(\bar{B}^\alpha,\bar{B}^\beta),I}$ is the forgetful morphism forgetting the $r$-spin structures, it is locally a diffeomorphism on the coarse underlying level, and the orientation $\tilde{\mathfrak o}^{x^\alpha_i,x^\beta_j}_{1,(\bar{B}^\alpha,\bar{B}^\beta),I}$ on $\Mbar^{\alpha,\beta}_{1,(\bar{B}^\alpha,\bar{B}^\beta),I}$ induces an orientation (still denoted by $\tilde{\mathfrak o}^{x^\alpha_i,x^\beta_j}_{1,(\bar{B}^\alpha,\bar{B}^\beta),I}$) on $\Mbarstar^{1/r,\alpha,\beta}_{1,(\bar{B}^\alpha,\bar{B}^\beta),I}$, satisfying the same properties listed in Proposition \ref{prop orientation moduli g=1}.

On the other hand, if at least one of $B^\alpha$ is $B^\beta$ is non-empty, the morphism $$\operatorname{For}^{1/r}_{\alpha,\beta}\colon \Mbarstar^{1/r,\alpha,\beta}_{1,(\bar{B}^\alpha,\bar{B}^\beta),I} \to \Mbarstar^{1/r}_{1,\{\bar{B}^\alpha,\bar{B}^\beta\},I}$$ 
forgetting the label of boundary components is an isomorphism as  $\operatorname{For}^{}_{\alpha,\beta} \colon \Mbar^{\alpha,\beta}_{1,(\bar{B}^\alpha,\bar{B}^\beta),I}\to \Mbar^{}_{1,\{\bar{B}^\alpha,\bar{B}^\beta\},I}$ is an isomorphism in this case. 
In the case $B^\alpha=B^\beta=\emptyset$, the morphism $\operatorname{For}^{1/r}_{\alpha,\beta}$ is a double covering. The morphisms $\Id_{\emptyset,\emptyset}$ and $\operatorname{Exch}$ also lift to the  moduli spaces with $r$-spin structures and satisfy the same properties.

\subsection{Orientation of Witten bundle}
Let $\Mbarstar$ be a moduli space of graded $r$-spin disks {or cylinders} (with possible illegal boundary markings) and $\mathcal M^*\subseteq \Mbarstar$ be its smooth locus. Let $\mathcal W$ be the Witten bundle over $\overline{\mathcal{QM}^*}$ (recall that in genus-zero $\overline{\mathcal{QM}^*}=\Mbarstar$). In this subsection, we construct a relative orientation of $\mathcal W \to \overline{\mathcal{QM}^*}$. Since $\overline{\mathcal{QM}^*} \setminus \mathcal M^*$ consists of boundaries and strata of codimension at least $2$, it is enough to construct a relative orientation of $\mathcal W \to \mathcal M^*$.

Let $\mathcal C\to \mathcal M^*$ be the universal curve and $\mathcal J\to \mathcal C$ be the universal Serre-dual bundle as in \eqref{eq universal companion bundle}, we denote by $\lvert \mathcal C \rvert$ and $\lvert \mathcal J \rvert$ the underlying coarse curve and line bundle. We put $N=\operatorname{rank} \mathcal W$.

\subsubsection{Construction of orientation for genus-zero Witten bundles}

{When $g=0$,} let $p,q_1,q_2,\dots,q_N\colon \mathcal M^*\to \lvert \mathcal C \rvert^{\lvert \phi\rvert}$ be continuous choices of boundary points $p,q_1,q_2,\dots,q_N$ in the $\lvert \phi\rvert$-fixed locus of $\lvert \mathcal C \rvert$  such that 
\begin{itemize}
    \item $p,q_1,q_2,\dots,q_N$ do not touch each other and they are in the cyclic order induced by the canonical orientation on $\partial \Sigma\subset \lvert C\rvert$ on each fibre;
    \item $p$ does not touch any boundary marked points.
\end{itemize}

Note that for each fibre $C$ of $\mathcal C \to \mathcal M^*$, the underlying coarse curve $\lvert C\rvert$ is a rational curve, and the underlying coarse line bundle $\lvert J\rvert$ over $\lvert C\rvert$ is a degree-$(N-1)$ line bundle. Then, for each $1\le i \le N$, the line bundle $\mathcal \lvert J\rvert\otimes \mathcal O\left([q_i]-\sum_{j=1}^N [q_j]\right)$ is trivial since it is a degree-$0$ line bundle over a rational curve.
We denote by $s_i$ the unique (up to positively scaling) section of $\lvert J\rvert$ such that 
\begin{itemize}
    \item $s_i$ has simple zeros at $q_1,q_2,\dots,q_{i-1},q_{i+1},\dots,q_{N}$;
    \item $s_i$ is negative (with respect to the grading) at $p$.
\end{itemize}
Due to the canonical isomorphism $H^0(J)=H^0(\lvert J\rvert)$, we can regard $s_i$ as sections of Witten bundle $\mathcal W$ up to a positively scaling. Note that the sections $\{s_i\}_{i= 1,\dots,N}$ are linearly independent since all of them vanish at $q_j$ except for $s_j$.  We denote by $\mathfrak o(p,q_1,\dots,q_N)$   the orientation of $\mathcal W$ given by
$$
s_1\wedge s_2\dots \wedge s_N;
$$
in the case $N=0$, we define $\mathfrak o(p)$ to be the positive orientation.

\begin{lem}
    The orientation $\mathfrak o(p,q_1,\dots,q_N)$ is independent of the choice of $q_1,\dots,q_N$; it depends only on the choice of $p$.
\end{lem}

\begin{proof}
For a fixed boundary point $p$, let $p,q'_1,q'_2,\dots,q'_N$ be another continuous choice of boundary points in cyclic order (induced by the canonical orientation on $\partial \Sigma$), and let $s'_1,s'_2,\dots,s'_N$ be the corresponding sections of $\mathcal J$. We can find a  homotopy $H\colon {\mathcal M^{*}}^{\times N}\times [0,1]\to \lvert \mathcal C \rvert^{\lvert \phi\rvert}$  between $q_1,q_2,\dots,q_N$ and $q'_1,q'_2,\dots,q'_N$ which preserves the cyclic order. One such homotopy is the linear homotopy by regarding $\partial \Sigma\setminus\{p\}$ as a parametrized interval. For each $t\in [0,1]$, the induced sections $s^H_i(t)$ are linearly independent, thus $s^H_1(t)\wedge s^H_2(t)\dots \wedge s^H_N(t)$ homotopes between $s_1\wedge s_2\dots \wedge s_N$ and $s'_1\wedge s'_2\dots \wedge s'_N$ without vanishing. Thus $\mathfrak o(p,q_1,\dots,q_N)=\mathfrak o(p,q'_1,\dots,q'_N)$, which means that $\mathfrak o(p,q_1,\dots,q_N)$ depends only on the choice of $p$. 
\end{proof}
We write $\mathfrak o(p):=\mathfrak o(p,q_1,\dots,q_N)$ for any choice of $q_1,\dots,q_N$.
Let $p$ and $p'$ be two boundary points, we compare $\mathfrak o(p)$ and $\mathfrak o(p')$. If there are no legal boundary marked points on the interval between $p$ and $p'$, then we have $\mathfrak o(p)=\mathfrak o(p')$ since $s_i=s'_i$ for all $i$ under the choice of $q_i$ such that all $q_i$ are not on that interval between $p$ and $p'$. If there is one legal boundary marked point on the interval between $p$ and $p'$, then we have 
\begin{equation} \label{eq change of orientation of witten when cross a legal in g 0}
\mathfrak o(p)=(-1)^N\mathfrak o(p')
\end{equation}
since we can make $s_i=-s'_i$ for all $i$ under the above choice of $q_i$.

\begin{dfn}
    Let $\mathcal W_{0,\bar{B},I}$ be the Witten bundle over $\overline{\mathcal M^*}^{1/r}_{0,\bar{B},I}$, for each boundary marking $x\in B$, we define the orientation $\mathfrak o^{x}_{0,\bar{B},I}$ of $\mathcal W_{0,\bar{B},I}$ to be $\mathfrak o(p)$, where $p$ is a point on the arc from  $\sigma_2^{-1}(x)$ to $x$. In the case $B=\emptyset$, we define $\mathfrak o^{x}_{0,\bar{B},I}$ to be $\mathfrak o(p)$, where $p$ is an arbitrary boundary point; the superscript $x$ is only for the sake of maintaining symbol consistency in this case again.
\end{dfn}

\begin{lem}\label{lem independent of choice of point g=0}
When all boundary points are legal, the relative orientation $\tilde{\mathfrak o}^{x}_{0,\bar{B},I}\otimes {\mathfrak o}^{x}_{0,\bar{B},I}$ is independent of the choice of $x\in B$.
\end{lem}

\begin{proof}
    According to the second item of Proposition \ref{prop:orientation moduli} we have $\tilde{\mathfrak o}^{\sigma_2(x)}_{0,\bar{B},I}=(-1)^{\vert B \vert-1}\tilde{\mathfrak o}^{x}_{0,\bar{B},I}$.
    On the other hand, since $\operatorname{rank} \mathcal W\equiv \lvert  B\rvert-1 \mod 2$ by \eqref{eq pairty all genus}
     and \eqref{eq rank of witten bundle}, we have ${\mathfrak o}^{\sigma_2(x)}_{0,\bar{B},I}=(-1)^{\vert B \vert-1}{\mathfrak o}^{x}_{0,\bar{B},I}$. 

\end{proof}

\begin{dfn}
    We define a relative orientation ${ o}^{x}_{0,\bar{B},I}$ of  $\mathcal W_{0,\bar{B},I}\to \overline {\mathcal M^*}^{1/r}_{0,\bar{B},I}$ to be \begin{equation}    \label{eq definition canonical orientation}  
    { o}^{x}_{0,\bar{B},I}:=(-1)^{m^\delta(0,\bar{B},I)}\tilde{\mathfrak o}^{x}_{0,\bar{B},I}\otimes {\mathfrak o}^{x}_{0,\bar{B},I},
    \end{equation} 
    where 
    $$m^\delta(0,\bar{B},I):=\frac{\operatorname{rank}_{\mathbb R}\mathcal W_{0,\bar{B},I}+1-\#\{b\in B\colon b \text{ legal}\}}{2}.$$
    In particular, when every markings in $B$ are legal, ${ o}^{x}_{0,\bar{B},I}$ is independent of the choice of $x$. We denote by ${ o}^{}_{0,\bar{B},I}$, the \textit{canonical relative orientation} of $\mathcal W_{0,\bar{B},I}\to \overline {\mathcal M}^{1/r}_{0,\bar{B},I}$ to be ${ o}^{x}_{0,\bar{B},I}$ for any $x\in B$.
\end{dfn}

\subsubsection{Construction of orientation for genus-one Witten bundles}
In the case $g=1$, recall that in \S \ref{sec orientation for moduli g=1} we constructed a moduli space $\Mbarstar^{1/r,\alpha,\beta}_{1,(\bar{B}^\alpha,\bar{B}^\beta),I}$ parametrizing $r$-spin cylinders with addition labels $\alpha$ and $\beta$ for their two boundary components, together with  a forgetful morphism $\operatorname{For}^{1/r}_{\alpha,\beta}\colon \Mbarstar^{1/r,\alpha,\beta}_{1,(\bar{B}^\alpha,\bar{B}^\beta),I}\to \mathcal \Mbarstar^{1/r}_{1,\{\bar{B}^\alpha,\bar{B}^\beta\},I}$ which is a double covering in the case $\bar{B}^\alpha=\bar{B}^\beta=\emptyset$ and is an isomorphism otherwise. 
We denote by $\overline{\mathcal{QM}^*}^{1/r,\alpha,\beta}_{1,(\bar{B}^\alpha,\bar{B}^\beta),I}\subseteq\Mbarstar^{1/r,\alpha,\beta}_{1,(\bar{B}^\alpha,\bar{B}^\beta),I}$ the correspond subspace after removing dimension-jump loci. 
We start from constructing orientations for the (pulled back via $\operatorname{For}^{1/r}_{\alpha,\beta}$) Witten bundle over $\overline{\mathcal{QM}^*}^{1/r,\alpha,\beta}_{1,(\bar{B}^\alpha,\bar{B}^\beta),I}\cap \Mstar^{1/r,\alpha,\beta}_{1,(\bar{B}^\alpha,\bar{B}^\beta),I}$. 

We decompose the boundary of a cylinder $\Sigma$ according to the label as $\partial \Sigma=\partial_\alpha\Sigma\sqcup \partial_\beta\Sigma$. Let $M_i$ be the number of legal marked points on $\partial_i \Sigma$ for $i=\alpha,\beta$. In the case where $\operatorname{rank} \mathcal W \ge 1$, we can choice $N_\alpha,N_\beta$ such that $N_\alpha+N_\beta=N=\operatorname{rank} \mathcal W$ and 
\begin{equation}\label{eq choice of N1 N2}
    N_j\equiv M_j+1, j=\alpha,\beta. 
\end{equation}

Let $p_\alpha,q_{\alpha,1},q_{\alpha,2},\dots,q_{\alpha,{N_\alpha}}\colon \mathcal M^*\to \lvert \mathcal C \rvert^{\lvert \phi\rvert}$ and $p_\beta,q_{\beta,1},q_{\beta,2},\dots,q_{\beta,{N_2}}\colon \mathcal M^*\to \lvert \mathcal C \rvert^{\lvert \phi\rvert}$ be continuous choices of boundary points 
in the $\lvert \phi\rvert$-fixed locus of $\lvert \mathcal C \rvert$  such that for both $j=\alpha,\beta$:
\begin{itemize}
    \item  $p_j,q_{j,1},q_{j,2},\dots,q_{j,{N_j}}$ lie in $\partial_j \Sigma\subset \lvert C\rvert$ fibrewise; they do not touch each other and they are in the cyclic order induced by the canonical orientation on $\partial_i \Sigma\subset \lvert C\rvert$ on each fibre;
    \item $p_j$ does not touch any boundary marked points.
\end{itemize}

Note that in $g=1$ case, as $\Mstar$ does not contain any dimension-jump locus, we have 
$$\deg \lvert J\rvert=h^0(\lvert J\rvert)-h^1(\lvert J\rvert)+g-1=  h^0(\lvert J\rvert)=\operatorname{rank} \mathcal W =N_\alpha+N_\beta.$$ 
 For $j=\alpha,\beta$ and $1\le i \le N_j$, 
 the line bundle $\mathcal \lvert J\rvert\otimes \mathcal O\left([q_{j,i}]-\sum_{n=1}^{N_\alpha} [q_{\alpha,n}]-\sum_{n=1}^{N_\beta} [q_{\beta,n}]\right)$ is a degree-one line bundle over a smooth genus-one curve $\lvert C \rvert$, so it has a unique section (up to real scaling) invariant under the anti-holomorphic involution.
 Then there exists a unique (up to positively scaling) section of $\lvert J\rvert$, denoted by $s_{j,i}$, such that 
\begin{itemize}
    \item $s_{j,i}$ has simple zeros at all the points in the set $\{q_{\alpha,1},\dots,q_{\alpha,N_\alpha},q_{\beta,1},\dots,q_{\beta,N_\beta}\}\setminus\{q_{j,i}\}$;
    \item $s_{j,i}$ is negative (with respect to the grading) at $p_j$.
\end{itemize}

Except for zeros in $\{q_{\alpha,1},\dots,q_{\alpha,N_\alpha},q_{\beta,1},\dots,q_{\beta,N_\beta}\}\setminus\{q_{j,i}\}$, the section $s_{j,i}$ has an additional zero; we denote this zero by $\check{q}_{j,i}$. According to item \ref{it lift compatible parity} in Proposition \ref{prop lifting}, the additional zero $\check{q}_{j,i}$ lies on $\partial_{j'}\Sigma$ for all $1\le i \le N_j$, where $\{j'\}\cup\{j\}=\{\alpha,\beta\}$, and therefore $s_{j,i}$ does not vanish at $q_{j,i}$. 

We denote by $\mathfrak o(p_\alpha,p_\beta,q_{\alpha,1},\dots,q_{\alpha,N_\alpha},q_{\beta,1},\dots,q_{\beta,N_\beta})$   the orientation of $\mathcal W$ given by
$$
s_{\alpha,1}\dots \wedge s_{\alpha,N_\alpha}\wedge s_{\beta,1}\wedge \dots \wedge s_{\beta,N_\beta}.
$$
By the same argument as in the genus-zero case, the orientation $\mathfrak o(p_\alpha,p_\beta,q_{\alpha,1},\dots,q_{\alpha,N_\alpha},q_{\beta,1},\dots,q_{\beta,N_\beta})$ is independent of the choice of $q_{\alpha,1},\dots,q_{\alpha,N_\alpha},q_{\beta,1},\dots,q_{\beta,N_\beta}$; we denote it by $\mathfrak o(p_\alpha,p_\beta,N_\alpha,N_\beta)$.

If we exchange the order of the two boundaries $\partial_1\Sigma$ and $\partial_2\Sigma$, then by definition, we have 
\begin{equation}\label{eq change of orientation of witten when exchange two bdries}
\mathfrak o(p_\alpha,p_\beta,N_\alpha,N_\beta)=(-1)^{N_\alpha N_\beta}\operatorname{Exch}^*\mathfrak o(p_\beta,p_\alpha,N_\beta,N_\alpha).
\end{equation}

Similar to \eqref{eq change of orientation of witten when cross a legal in g 0}, if there is one legal boundary marked point on the interval between $p_1$ and $p'_1$, then we have 
\begin{equation}\label{eq change of orientation of witten when cross a legal in g 1 1st bdry}
\mathfrak o(p_\alpha,p_\beta,N_\alpha,N_\beta)=(-1)^{N_\alpha}\mathfrak o(p'_\alpha,p_\beta,N_\alpha,N_\beta);
\end{equation} 
if there is one legal boundary marked point on the interval between $p_2$ and $p'_2$, then we have 
\begin{equation}\label{eq change of orientation of witten when cross a legal in g 1 2nd bdry}
\mathfrak o(p_\alpha,p_\beta,N_\alpha,N_\beta)=(-1)^{N_\beta}\mathfrak o(p_\alpha,p'_\beta,N_\alpha,N_\beta).
\end{equation}

\begin{lem}
    If $N_\alpha \ge 2$, we have  $\mathfrak o(p_\alpha,p_\beta,N_\alpha,N_\beta)=\mathfrak o(p_\alpha,p_\beta,N_\alpha-2,N_\beta+2)$.
\end{lem}
\begin{proof}
    We only need to check it for the fibre of $\mathcal W$ over a fixed cylinder $(C,\Sigma)$. According to \eqref{eq change of orientation of witten when cross a legal in g 1 1st bdry} and \eqref{eq change of orientation of witten when cross a legal in g 1 2nd bdry}, we only need to prove it for a specific choice of $p_\alpha,p_\beta$.

    For each cylinder $(C,\Sigma)$ in $\mathcal M^*$, we choice a coordinate system and represent $C$ as $\mathbb C=\mathbb R_x\times R_y$ quotiented by the lattice generated by $(1,0)$ and $(0,2\tau)$, and the fundamental domain of $\Sigma$ is $\{0\le x < 1\}\times \{0\le y < \tau\}$. We assume $(0,0)$ and $(0,\tau)$ are now boundary markings, otherwise we can shit the $x$-coordinator a little bit.

    We fix $p_1=(0,0)$ and $p_2=(0,\tau).$ If we choice $q_{\alpha,1},q_{\alpha,2},\dots,q_{\alpha,{N_\alpha}}$ and $q_{\beta,1},q_{\beta,2},\dots,q_{\beta,{N_\beta}}$ to define $\mathfrak o(p_\alpha,p_\beta,N_\alpha,N_\beta)$ as 
    $$
    s_{\alpha,1}\dots \wedge s_{\alpha,N_\alpha}\wedge s_{\beta,1}\wedge \dots \wedge s_{\beta,N_\beta},
    $$
    then we have $q_{\alpha,i}=(x_{\alpha,i},0)$ for each $0\le i \le N_\alpha$, where $0< x_{\alpha,1}<x_{\alpha,2}<\dots<x_{\alpha,N_\alpha}<1$; we also have $q_{\beta,i}=(x_{\beta,i},\tau)$ for each $1\le i \le N_\beta$, where $1> x_{\beta,1}>x_{\beta,2}>\dots>x_{\beta,N_\beta}>0$. We write $\check{q}_{\alpha,i}=(\check{x}_{\alpha,i},\tau)$ and $\check{q}_{\beta,i}=(\check{x}_{\beta,i},0)$.
    Since $s_{j,i}$ are sections of a fixed bundle $\lvert J \rvert$, the sum of (the $x$-coordinators) its zeros 
    \begin{equation}\label{eq sum of zero of x coordinator of a section}
        X:=\sum_{k=1}^{N_\alpha} x_{\alpha,k}+  \sum_{k=1}^{N_\beta} x_{\beta,k} -x_{j,i} + \check{x}_{j,i}
    \end{equation}
    is a constant modulo $\mathbb Z$ for a fixed coordinate system. We can assume $0< X < 1$, otherwise we can shift the $x$-coordinator in our coordinate system.

    We take  $q_{\alpha,1},q_{\alpha,2},\dots,q_{\alpha,{N_\alpha}}$ and $q_{\beta,1},q_{\beta,2},\dots,q_{\beta,{N_\beta}}$ in a way that 
    $$0< x_{\alpha,1}<x_{\alpha,2}<\dots<x_{\alpha,N_\alpha}<\frac{\operatorname{min}\{X,1-X\}}{2N}$$ 
    and 
    $$1> x_{\beta,1}>x_{\beta,2}>\dots>x_{\beta,N_\beta}>1-\frac{\operatorname{min}\{X,1-X\}}{2N},$$ 
    are require then do not touch the boundary markings. Then by \eqref{eq sum of zero of x coordinator of a section} we have 
    $$1> x_{\beta,1}>x_{\beta,2}>\dots>x_{\beta,N_\beta}> \check{x}_{\alpha,N_\alpha}> \check{x}_{\alpha,N_\alpha-1}>0.$$

    Now we choice $q'_{\alpha,1},q'_{\alpha,2},\dots,q'_{\alpha,{N_\alpha-1}}$ and $q'_{\beta,1},q'_{\beta,2},\dots,q'_{\beta,{N_\beta}},q'_{\beta,{N_\beta+1}},q'_{\beta,{N_\beta+2}}$ to define the orientation $\mathfrak o(p_\alpha,p_\beta,N_\alpha-2,N_\beta+2)$ as 
    $$
    s'_{\alpha,1}\dots \wedge s'_{\alpha,N_\alpha-2}\wedge s'_{\beta,1}\wedge \dots \wedge s'_{\beta,N_\beta+2}.
    $$
    We take $q'_{\alpha,i}=q_{\alpha,i}$ for $1\le i \le N_\alpha-2$, $q'_{\beta,i}=q_{\beta,i}$ for $1\le i \le N_\beta$,  $q'_{\beta,N_\beta+1}=\check{q}_{\alpha,N_\alpha}$ and $q'_{\beta,N_\beta+2}=\check{q}_{\alpha,N_\alpha-1}$. 
    Under this choice, the sections $s_{j,i}$ and $s'_{j,i}$ satisfy the following properties.
    \begin{itemize}
        \item For any $1\le i \le N_\alpha-2$, both $s_{\alpha,i}$ and $s'_{\alpha,i}$ vanish when evaluated at $\{q_{\alpha,1},\dots,q_{\alpha,N_\alpha},q_{\beta,1},\dots,q_{\beta,N_\beta}\}\setminus\{q_{\alpha,i}\}$. Moreover, $s_{\alpha,i}(q_{\alpha,i})$ and $s'_{\alpha,i}(q_{\alpha,i})$ have the same sign with respect to the grading: $s_{\alpha,i}$ and $s'_{\alpha,i}$ have the same number of zeros on the arc from $p_\alpha$ to $q_{\alpha,i}$, and both $s_{\alpha,i}(p_\alpha)$ and $s'_{\alpha,i}(p_\alpha)$ are negative by definition. Therefore, for each $0\le t \le 1$, the section
        $$
        s^t_{\alpha,i}:=(1-t)s_{\alpha,i}+ts'_{\alpha,i}
        $$
        vanishes when evaluated at $\{q_{\alpha,1},\dots,q_{\alpha,N_\alpha},q_{\beta,1},\dots,q_{\beta,N_\beta}\}\setminus\{q_{\alpha,i}\}$, and does not vanish when  evaluated at $q_{\alpha,i}$.
        
        \item For any $1\le i \le N_\beta$, similar to the above item, for each for each $0\le t \le 1$, the section
        $$
        s^t_{\beta,i}:=(1-t)s_{\beta,i}+ts'_{\beta,i}
        $$
        vanishes when evaluated at $\{q_{\alpha,1},\dots,q_{\alpha,N_\alpha},q_{\beta,1},\dots,q_{\beta,N_\beta}\}\setminus\{q_{\beta,i}\}$, and does not vanish when  evaluated at $q_{\beta,i}$.
        
        \item The zeros of $s_{\alpha,N_\alpha-1}$ and $s'_{\beta,N_\beta+1}$ are the same: at $q_{\alpha,i}$ for $1\le i \le N_\alpha-2$, $q_{\beta,i}$ for $1\le i \le N_\beta$, $q_{\alpha,N_\alpha}$ and $\check{q}_{1,N_\alpha-1}$. Thus $\frac{s_{\alpha,N_\alpha-1}}{s'_{\beta,N_\beta+1}}$ is a non-zero real number  $\frac{s_{\alpha,N_\alpha-1}(p_\beta)}{s'_{\beta,N_\beta+1}(p_\beta)}$. Similarly, $\frac{s_{\alpha,N_\alpha}}{s'_{\beta,N_\beta+2}}$ is a non-zero real number  $\frac{s_{\alpha,N_\alpha}(p_\beta)}{s'_{\beta,N_\beta+2}(p_\beta)}$. Note that $s'_{\beta,N_\beta+1}(p_\beta)$ and $s'_{\beta,N_\beta+2}(p_\beta)$ are always negative by definition. 
        Moreover, we can show that $s_{\alpha,N_\alpha}(p_\beta)$ and $s_{\alpha,N_\alpha-1}(p_\beta)$ always have the same sign. In fact, we can construct a homotopy $s^\nu$ between $s_{\alpha,N_\alpha}$ and $s_{\alpha,N_\alpha-1}$ with a homotopy parameter $0\le \nu \le 1$: the zeros of $s^\nu$ are at $q_{\alpha,i}$ for $1\le i \le N_\alpha-2$, $q_{\beta,i}$ for $1\le i \le N_\beta$, $\nu q_{\alpha,N_\alpha}+(1-\nu)q_{\alpha,N_\alpha-1}$ and $\nu\check{q}_{\alpha,N_\alpha-1}+(1-\nu)\check{q}_{\alpha,N_\alpha}$. 
        Then for any $0\le \nu \le 1$, $s^\nu$ does not vanish at $p_\beta$, which means $s_{\alpha,N_\alpha}(p_\beta)=s^{\nu=0}(p_\beta)$ and $s_{\alpha,N_\alpha-1}(p_\beta)=s^{\nu=1}(p_\beta)$ have the same sign. Therefore up to a positive scaling we have
        $$
        s_{\alpha,N_\alpha-1}\wedge s_{\alpha,N_\alpha}=s'_{\beta,N_\beta+1} \wedge s'_{\beta,N_\beta+2}.
        $$
    \end{itemize}
    Then the family of wedge products (over $0\le t \le 1$)
    $$
    s^t_{\alpha,1}\dots \wedge s^t_{\alpha,N_\alpha-2}\wedge s_{\alpha,N_\alpha-1}\wedge s_{\alpha,N_\alpha}\wedge s^t_{\beta,1}\wedge \dots \wedge s^t_{\beta,N_\beta},
    $$
    homotopes between 
    $$\mathfrak o(p_\alpha,p_\beta,N_\alpha,N_\beta)= s_{\alpha,1}\dots \wedge s_{\alpha,N_\alpha-2}\wedge s_{\alpha,N_\alpha-1}\wedge s_{\alpha,N_\alpha}\wedge s_{\beta,1}\wedge \dots \wedge s_{\beta,N_\beta}$$
    and 
    \begin{equation*}
    \begin{split}
        \mathfrak o(p_\alpha,p_\beta,N_\alpha-2,N_\beta+2)=& s'_{\alpha,1}\dots \wedge s'_{\alpha,N_\alpha-2}\wedge s'_{\beta,1}\wedge \dots \wedge s'_{\beta,N_\beta} \wedge s'_{\beta,N_\beta+1}\wedge s'_{\beta,N_\beta+2}\\
        =&s'_{\alpha,1}\dots \wedge s'_{\alpha,N_\alpha-2}\wedge s'_{\beta,1}\wedge \dots \wedge s'_{\beta,N_\beta} \wedge s_{\alpha,N_\alpha-1}\wedge s_{\alpha,N_\alpha}\\
        =&s'_{\alpha,1}\dots \wedge s'_{\alpha,N_\alpha-2} \wedge s_{\alpha,N_\alpha-1}\wedge s_{\alpha,N_\alpha}\wedge s'_{\beta,1}\wedge \dots \wedge s'_{\beta,N_\beta}.
    \end{split}
     \end{equation*}
    It remains to show that for any $0\le t \le 1$, the sections $s^t_{\alpha,1},\dots, s^t_{\alpha,N_\alpha-2}$, $s^t_{\beta,1}, \dots , s^t_{\beta,N_\beta}$, $ s_{\alpha,N_\alpha-1},$ $s_{\alpha,N_\alpha}$ are linearly independent. Assuming a linear combination of them are zero, \textit{i.e.,}
    $$
    \lambda^t_{\alpha,1}s^t_{\alpha,1}\dots + \lambda^t_{\alpha,N_\alpha-2}s^t_{\alpha,N_\alpha-2}+ \lambda_{\alpha,N_\alpha-1} s_{\alpha,N_\alpha-1}+ \lambda_{\alpha,N_\alpha} s_{\alpha,N_\alpha}+ \lambda^t_{\beta,1} s^t_{\beta,1}+ \dots + \lambda^t_{\beta,N_\beta} s^t_{\beta,N_\beta}=0,
    $$
    the evaluation at $q_{\alpha,i}$ forces $\lambda^t_{\alpha,i}=0$ for all $1\le i \le N_\alpha-2$ since $s^t_{\alpha,i}$ is the only one dose not vanish at $q_{\alpha,i}$. Similarly we have $\lambda^t_{\beta,i}=0$ for all $1\le i \le N_\beta$. Then the only terms left are 
    $$\lambda_{\alpha,N_\alpha-1} s_{\alpha,N_\alpha-1}+ \lambda_{\alpha,N_\alpha} s_{\alpha,N_\alpha}=0,$$ this forces $\lambda_{\alpha,N_\alpha-1}=\lambda_{\alpha,N_\alpha}=0$ since $s_{\alpha,N_\alpha-1}$ and $s_{\alpha,N_\alpha}$ have different zeros.
\end{proof}

Thus we write
$$
\mathfrak o(p_\alpha,p_\beta):=\mathfrak o(p_\alpha,p_\beta,N_\alpha,N_\beta)
$$
for any choice of $N_\alpha,N_\beta$ such that $N_\alpha+N_\beta=N=\operatorname{rank} \mathcal W$ and $N_j\equiv M_j+1, j=\alpha,\beta,$ where $M_j$ is the number of legal boundary markings on the boundary where $p_j$ lies.

In the case $\operatorname{rank} \mathcal W =0$, we can still construct the orientation $\mathfrak o (p_\alpha,p_\beta)$ (positive or negative) of the Witten bundle (away from the dimension-jump locus) for each choice of boundaries points $p_\alpha,p_\beta$ lying on different boundaries of $\Sigma$ which are not marked points. There are two possibilities.
\begin{enumerate}
    \item Both $M_\alpha$ and $M_\beta$ are odd. In this case we define $\mathfrak o (p_\alpha,p_\beta)$ to be positive for any choice of $p_\alpha,p_\beta$.
    \item Both $M_\alpha$ and $M_\beta$ are even. In this case, assuming $\bar{B}^\alpha$ and $\bar{B}^\beta$ the sets of boundary markings on two boundaries. We write $\bar{B}^\alpha\cup\{p_\alpha\}$ and $\bar{B}^\beta\cup\{p_\beta\}$  in their cyclic order as $\bar{B}^\alpha\cup\{p_\alpha\}=\overline{\{p_\alpha,b_{11},b_{12},\dots,b_{1k_1}\}}$ and $\bar{B}^\beta\cup\{p_\beta\}=\overline{\{p_\beta,b_{21},b_{22},\dots,b_{2k_1}\}}$. 
    We consider the strata $\mathcal M_{\Gamma}$ with a single non-separating boundary node where the arcs $p_\alpha$ and $p_\beta$ lying on meet each other, or in other words, $\Gamma$ is a graded $r$-spin graph with a single genus-zero open vertex $v$, and a non-separating boundary edge $e$ connecting $v$ to itself, where the set of boundary half-edges of $v$ in its cyclic order is $\bar{B}_v=\overline{\{b_{11},b_{12},\dots,b_{1k_1},h_2,b_{21},b_{22},\dots,b_{2k_1},h_1\}}$. Since $\Sigma$ is not in the dimension-jump locus, the edge $e$ is NS (see Example \ref{exa dim jump non separating}). Moreover, such $\Gamma$ is uniquely determined by the connected component of $\Mbar^{1/r,\alpha,\beta}_{{1,(\bar{B}^\alpha,\bar{B}^\beta),I}}$ containing $\Mbar_{\Gamma}$ since  $\operatorname{rank} \mathcal W =0$ implies all the marking have twist zero (see Example \ref{exa dim jump rank zero}). If $h_1$ is illegal and $h_2$ is legal, then we define $\mathfrak o (p_\alpha,p_\beta)$ to be positive on this connected component; if $h_1$ is legal and $h_2$ is illegal, then we define $\mathfrak o (p_\alpha,p_\beta)$ to be negative on this connected component.
    
\end{enumerate}

\begin{dfn}
    Let $\mathcal W_{1,(\bar{B}^\alpha,\bar{B}^\beta),I}$ be the Witten bundle over $\overline{\mathcal {QM}^*}^{1/r,\alpha,\beta}_{1,(\bar{B}^\alpha,\bar{B}^\beta),I}$, for each ordered pair $(x,y)$ of boundary markings such that $x\in B^\alpha,y\in B^\beta$, we define the orientation $\mathfrak o^{x,y}_{1,(\bar{B}^\alpha,\bar{B}^\beta),I}$ of $\mathcal W_{1,(\bar{B}^\alpha,\bar{B}^\beta),I}$ to be $\mathfrak o(p_x,p_y)$, where $p_x$ is a point on the arc from  $\sigma_2^{-1}(x)$ to $x$, and $p_y$ is a point on the arc from  $\sigma_2^{-1}(y)$ to $y$. In the case $B^\alpha$ (respectively $B^\beta$) is empty, we define $\mathfrak o^{x,y}_{1,(\bar{B}^\alpha,\bar{B}^\beta),I}$ to be $\mathfrak o(p_x,p_y)$, where $p_x$ (respectively $p_y$) is an arbitrary boundary point on the boundary  component $\partial_\alpha \Sigma$ labelled by $\alpha$ (respectively $\partial_\beta \Sigma$ labelled by $\beta$); the superscript $x,y$ is only for the sake of maintaining symbol consistency in this case again.
\end{dfn}

\begin{lem}\label{lem independent of choice of point g=1}
When all boundary markings in $B^\alpha$ are legal, the relative orientation $\tilde{\mathfrak o}^{x,y}_{1,(\bar B^\alpha, \bar B^\beta ),I}\otimes {\mathfrak o}^{x,y}_{1,(\bar B^\alpha, \bar B^\beta ),I}$ is independent of the choice of $x\in B^\alpha$; When all boundary markings in $B^\beta$ are legal, the relative orientation $\tilde{\mathfrak o}^{x,y}_{1,(\bar B^\alpha, \bar B^\beta ),I}\otimes {\mathfrak o}^{x,y}_{1,(\bar B^\alpha, \bar B^\beta ),I}$ is independent of the choice of $y\in B^\beta$.
\end{lem}

\begin{proof}
       Similar to Lemma \ref{lem independent of choice of point g=0}, the proof follows from \eqref{eq change of orientation of witten when cross a legal in g 1 1st bdry}, \eqref{eq change of orientation of witten when cross a legal in g 1 2nd bdry}, item \ref{item orientation of moduli when change starting point genus one} in Proposition \ref{prop orientation moduli g=1}, and \eqref{eq choice of N1 N2}.
\end{proof}

\begin{dfn}
    We define a relative orientation ${ o}^{x,y}_{1,(\bar B^\alpha, \bar B^\beta ),I}$ of  $\mathcal W_{1,(\bar B^\alpha, \bar B^\beta ),I}\to \overline {\mathcal {QM}^*}^{1/r,\alpha,\beta}_{1,(\bar B^\alpha, \bar B^\beta ),I}$ to be \begin{equation}    \label{eq definition canonical orientation g=1}  
    { o}^{x,y}_{1,(\bar B^\alpha, \bar B^\beta ),I}:=(-1)^{m^\delta(1,\{\bar B^\alpha, \bar B^\beta \},I)+m^{g=1}(1,\{\bar B^\alpha, \bar B^\beta \},I)}\tilde{\mathfrak o}^{x,y}_{1,(\bar B^\alpha, \bar B^\beta ),I}\otimes {\mathfrak o}^{x,y}_{1,(\bar B^\alpha, \bar B^\beta ),I},
    \end{equation} 
    where 
    $$m^\delta(1,\{\bar B^\alpha, \bar B^\beta \},I):=\frac{\operatorname{rank}_{\mathbb R}\mathcal W_{1,\{\bar B^\alpha, \bar B^\beta \},I}-\#\{b\in B^\alpha\sqcup B^\beta\colon b \text{ legal}\}}{2}$$
    and 
    $$
    m^{g=1}(1,\{\bar B^\alpha, \bar B^\beta \},I):=(\#\{b\in B^\alpha\colon b \text{ legal}\}+1)\cdot (\#\{b\in B^\beta\colon b \text{ legal}\}+1)-1
    $$
\end{dfn}

In the case $B^\alpha$ or $B^\beta$ is non-empty, since $\operatorname{For}^{1/r}_{\alpha,\beta}\colon \overline {\mathcal {QM}}^{1/r,\alpha,\beta}_{1,(\bar B^\alpha, \bar B^\beta ),I} \to \overline {\mathcal {QM}}^{1/r,}_{1,\{\bar B^\alpha, \bar B^\beta \},I}$ is an isomorphism, when there is no ambiguity, we will also denote by ${ o}^{x,y}_{1,(\bar B^\alpha, \bar B^\beta ),I}$ the relative orientation
$$(\operatorname{For}_{\alpha,\beta}^{1/r})_* { o}^{x,y}_{1,(\bar B^\alpha, \bar B^\beta ),I}:=\left((\operatorname{For}_{\alpha,\beta}^{1/r})^{-1}\right)^* { o}^{x,y}_{1,(\bar B^\alpha, \bar B^\beta ),I}$$
of the Witten bundle $\mathcal W_{1,\{\bar B^\alpha, \bar B^\beta \},I}\to \overline {\mathcal {QM}^*}^{1/r}_{1,\{\bar B^\alpha, \bar B^\beta \},I}$.
When all markings in $B^\alpha$ and $B^\beta$ are legal, according to \eqref{eq orientation moduli when exchange boundaries genus one}, \eqref{eq change of orientation of witten when exchange two bdries} and \eqref{eq choice of N1 N2}, we have 
\begin{equation}\label{eq exchange boundaries relative}
{ o}^{x,y}_{1,(\bar B^\alpha, \bar B^\beta ),I}=\operatorname{Exch}^*{ o}^{y,x}_{1,(\bar B^\beta, \bar B^\alpha ),I},
\end{equation}
which means the relative orientation $(\operatorname{For}_{\alpha,\beta}^{1/r})_* { o}^{x,y}_{1,(\bar B^\alpha, \bar B^\beta ),I}$
is independent of the choice of the order of the labels $\alpha,\beta$ as $\operatorname{Exch}\circ(\operatorname{For}_{\alpha,\beta}^{1/r})^{-1}=(\operatorname{For}_{\beta,\alpha}^{1/r})^{-1}$. It is also independent of the choice of $x\in B^\alpha$ and $y\in B^\beta$ according to Lemma \ref{lem independent of choice of point g=1}.

In the case $B^\alpha=B^\beta=\emptyset$, notice that by construction we have 
$$
{\mathfrak o}^{y,x}_{1,(\bar B^\beta, \bar B^\alpha ),I}=\Id_{\emptyset,\emptyset}^*{\mathfrak o}^{x,y}_{1,(\bar B^\alpha, \bar B^\beta ),I},
$$
together with \eqref{eq orientation of moduli without boundary points id} we have 
$$
{ o}^{y,x}_{1,(\bar B^\beta, \bar B^\alpha ),I}=\Id_{\emptyset,\emptyset}^*{o}^{x,y}_{1,(\bar B^\alpha, \bar B^\beta ),I},
$$
then \eqref{eq exchange boundaries relative} means the relative orientation   ${o}^{x,y}_{1,(\bar B^\alpha, \bar B^\beta ),I}$ on $\overline {\mathcal {QM}}^{1/r,\alpha,\beta}_{1,(\bar B^\alpha, \bar B^\beta ),I}$ is preserved by the automorphism $\Id_{\emptyset,\emptyset}\circ \operatorname{Exch}$ exchanging the two sheets of the double covering $\operatorname{For}_{\alpha,\beta}^{1/r} \colon \overline {\mathcal {QM}}^{1/r,\alpha,\beta}_{1,(\bar B^\alpha, \bar B^\beta ),I} \to \overline {\mathcal {QM}}^{1/r}_{1,\{\bar B^\alpha, \bar B^\beta \},I}$, therefore ${ o}^{x,y}_{1,(\bar B^\alpha, \bar B^\beta ),I}$ induces an relative orientation $(\operatorname{For}_{\alpha,\beta}^{1/r})_* { o}^{x,y}_{1,(\bar B^\alpha, \bar B^\beta ),I}$ of the Witten bundle on $ \overline {\mathcal {QM}}^{1/r}_{1,\{\bar B^\alpha, \bar B^\beta \},I}$, which is again independent of the choice of the order of the labels $\alpha,\beta$ by \eqref{eq exchange boundaries relative}.

\begin{dfn}
    Let $\bar B^1$ and $\bar B^2$ be two sets of legal boundary markings with cyclic order, and $I$ be a set of internal markings, we define the \textit{canonical relative orientation} ${ o}^{}_{1,\{\bar B^1, \bar B^2 \},I}$ of $\mathcal W_{1,\{\bar B^1, \bar B^2 \},I}\to \overline {\mathcal {QM}}^{1/r}_{1,\{\bar B^\alpha, \bar B^\beta \},I}$ as 
    $$
    { o}^{}_{1,\{\bar B^1, \bar B^2 \},I}:=(\operatorname{For}_{\alpha,\beta}^{1/r})_* { o}^{x,y}_{1,(\bar B^\alpha, \bar B^\beta ),I} 
    $$
    for any way to label $\{\bar B^1,\bar B^2\}$ as $\{\bar B^\alpha, \bar B^\beta\}$, and any choice of $x\in B^\alpha$, $y\in B^\beta$.

\end{dfn}

\subsubsection{Properties of orientation of Witten bundles at boundary strata}
Recall that Witten bundles decompose at the boundary strata (see Proposition \ref{prop decomposition g=0,1}). {The Witten bundle of $\Mbarstar^{1/r,\alpha,\beta}_{1,(\bar{B}^\alpha,\bar{B}^\beta),I}$ decomposes in exactly the same way.} We study the behaviour of the orientations constructed above under these decompositions.

\begin{prop}\label{prop: orientation open open}
      Let $\Gamma$ be a genus-zero stable graded $r$-spin graph with two open vertices connected by an NS boundary edge $e$, let $h_i$ denote the half-edges of $v_i$.
      Let $I_i$ be the sets of internal markings of $v_i$; we write $\bar{B}_1$, the boundary half-edges of $v_1$, in its cyclic order as $\bar{B}_1=\overline{\{b_{11},b_{12},\dots,b_{1k_1},h_1\}}$; we also write $\bar{B}_2$, the boundary half-edges of $v_2$, in its cyclic order as $\bar{B}_2=\overline{\{h_2,b_{21},b_{22},\dots,b_{2k_2}\}}$. Then the set of internal markings of $d_e \Gamma$ is $I=I_1\sqcup I_2$, and the set of boundary markings of $d_e \Gamma$ written in its cyclic order is $\bar{B}=\overline{\{b_{11},b_{12},\dots,b_{1k_1},b_{21},b_{22},\dots,b_{2k_2}\}}$. Under the identification 
      \begin{equation}\label{eq NS boundary node det decomp}
          \mu^*\det(\mathcal W_{0,\bar{B},I})\big|_{\Mbar_\Gamma} = q^*\left(\det(\mathcal W_{v_1})\boxtimes \det(\mathcal W_{v_2})\right)
      \end{equation}
       we have
        \begin{equation}\label{eq orientation bundle open open}
        \mu^*{\mathfrak{o}}^{b_{11}}_{0,\bar B,I}\big |_{\Mbar_\Gamma^{}}=
         q^*\left({\mathfrak{o}}^{b_{11}}_{0,\bar B_1,I_1}\boxtimes{\mathfrak{o}}^{h_2}_{0,\bar B_2,I_2}\right).
         \end{equation}
         We will omit $q^*$ and $\mu^*$ from the notation in \eqref{eq orientation bundle open open} for simplicity as $q$ and $\mu$ are isomorphisms in this case.
\end{prop}
\begin{proof}
    
    Let $N_1=\operatorname{rank} \mathcal W_{v_1}$ and $N_2=\operatorname{rank} \mathcal W_{v_2}$, then $\operatorname{rank} \mathcal W_{0,\bar{B},I}=N_1+N_2$. To prove \eqref{eq orientation bundle open open}, we define the orientation $\mathfrak{o}^{b_{11}}_{0,\bar B,I}$ with a specific choice of $p,q_1,q_2,\dots,q_{N_1+N_2}\colon \mathcal M_{0,\bar B,I}\to \mathcal C$, such that 
    \begin{itemize}
        \item $p$ lies on the arc from $b_{2k_2}$ to $b_{11}$;
        \item when approaching $\mathcal M_\Gamma$, the limits of $q_1,\dots,q_{N_1}$ lie on the irreducible component corresponding to $v_1$, and  the limits of $q_{N_1+1},\dots,q_{N_1+N_2}$ lie on the irreducible component corresponding to $v_2$.
    \end{itemize}  
    We denote the corresponding limits by \[q^1_1,\dots,q^1_{N_1}\colon \mathcal M^*_{v_1} \to \lvert \mathcal C_{v_1} \rvert^{\lvert \phi\rvert}~\text{ and  }q^2_{N_1+1},\dots,q^2_{N_1+N_2}\colon \mathcal M_{v_2} \to \lvert \mathcal C_{v_2} \rvert^{\lvert \phi\rvert}.\] We choose $p_1\colon \mathcal M^*_{v_1} \to \lvert \mathcal C_{v_1} \rvert^{\lvert \phi\rvert}$ on the arc from $h_1$ to $b_{11}$ and $p_2\colon \mathcal M_{v_2} \to \lvert \mathcal C_{v_2} \rvert^{\lvert \phi\rvert}$ on the arc from $b_{2k_2}$ to $h_2$, such that both $p_1,q^1_1,\dots,q^1_{N_1}$ and $p_2,q^2_{N_1+1},\dots,q^2_{N_1+N_2}$ are in the corresponding cyclic order. We denote by $(s_1,\dots,s_{N_1+N_2})$, $(s^1_1,\dots,s^1_{N_1})$ and $(s^2_{N_1+1},\dots,s^2_{N_1+N_2})$ the sections determined by $(p,q_1,q_2,\dots,q_{N_1+N_2})$, $(p_1,q^1_1,\dots,q^1_{N_1})$ and $(p_2,q^2_{N_1+1},\dots,q^2_{N_1+N_2})$ correspondingly.

    We denote by $\hat s_i$ the limit of $s_i$ when approaching $\mathcal M_\Gamma$,  perhaps after scaling by a positive function. For $1\le i \le N_1$, the restriction of $\hat s_i$ to the irreducible component corresponding to $v_2$ is the zero section since it vanishes at $N_2=\operatorname{rank} \mathcal W_{v_2}$ points $q^2_{N_1+1},\dots, q^2_{N_1+N_2}$;  the restriction of $\hat s_i$ to the irreducible component corresponding to $v_1$ vanishes at $q^1_1,q^1_2\dots, q^1_{i-1},q^1_{i+1},\dots,q^2_{N_1+N_2}$, by comparing the sign at $p_1$ it coincides with $s^1_i$. 
    Similarly for $N_1+1\le i \le N_1+N_2$, the restriction of $\hat s_i$ to the irreducible component corresponding to $v_1$ is the zero section, the restriction of $\hat s_i$ to the irreducible component corresponding to $v_2$ coincides with $s^2_i$.
    Then we have
    $$
    {\mathfrak{o}}^{b_{11}}_{0,\bar B,I}\big |_{\Mbar_\Gamma^{}}=
         \bigwedge_{i=1}^{N_1+N_2}\hat s_i  =\bigwedge_{i=1}^{N_1} s^1_i \boxtimes \bigwedge_{i=N_1}^{N_1+N_2} s^2_i={\mathfrak{o}}^{b_{11}}_{0,\bar B_1,I_1}\boxtimes{\mathfrak{o}}^{h_2}_{0,\bar B_2,I_2}.
    $$

\end{proof}

\begin{cor}\label{cor:orientation open open}
    With the same notation as in Proposition \ref{prop: orientation open open}, assuming every boundary marking in $B$ is legal, in the case $h_1$ is illegal and $h_2$ is legal, we have 
    \begin{equation}
        {{o}}^{}_{0,\bar B,I}\big |_{\Mbar_\Gamma^{}}=o_N\otimes\left(
         {{o}}^{b_{11}}_{0,\bar B_1,I_1}\boxtimes{{o}}^{}_{0,\bar B_2,I_2}\right),
         \end{equation}
         where $N$ is the outward normal with canonical orientation $o_N$.
\end{cor}
\begin{proof}
    The corollary follows from \eqref{eq definition canonical orientation}, \eqref{eq orientation moduli open open} and \eqref{eq orientation bundle open open}.  The factors of the form $(-1)^{m^\delta}$ disappear since
    $$
        m^\delta(0,\bar{B},I)=m^\delta(0,\bar{B_1},I_1)+m^\delta(0,\bar{B_2},I_2).
    $$
    The factor $(-1)^{(k_1-1)k_2}$ in \eqref{eq orientation moduli open open} disappears when commuting $\tilde{\mathfrak{o}}^{h_2}_{0,\bar B_2,I_2}$ with ${\mathfrak{o}}^{b_{11}}_{0,\bar B_1,I_1}$.
\end{proof}

\begin{prop}\label{prop: orientation open open g=1 separating}
      Let $\Gamma$ be a genus-one stable graded $r$-spin graph an open genus-zero vertex $v_1$ and an open genus-one vertex $v_2$, connected by an NS separating boundary edge $e$, with half-edges $h_i$ on $v_i$. Let $I_i$ be the sets of internal markings of $v_i$; we write $\bar{B}_1$, the boundary half-edges of $v_1$, in its cyclic order as $\bar{B}_1=\overline{\{b_{11},b_{12},\dots,b_{1k_1},h_1\}}$; we also write $\bar{B}_{2}^\alpha$ and $\bar{B}_{2}^\beta$, the boundary half-edges of $v_2$ on each boundaries (where $h_1\in \bar{B}_{2}^1$), in their cyclic order as $\bar{B}_{2}^\alpha=\overline{\{h_2,b_{21}^1,b_{22}^1,\dots,b_{2k_2^1}^1\}}$ and $\bar{B}_{2}^\beta=\overline{\{b_{21}^2,b_{22}^2,\dots,b_{2k_2^2}^2\}}$. Then the set of internal markings of $d_e \Gamma$ is $I=I_1\sqcup I_2$, and the sets of boundary markings of $d_e \Gamma$ on each boundaries written in their cyclic order is $\bar{B}^\alpha=\overline{\{b_{11},b_{12},\dots,b_{1k_1},b_{21}^1,b_{22}^1,\dots,b_{2k_2^1}^1\}}$ and $\bar{B}^\beta=\bar{B}_{2}^\beta=\overline{\{b_{21}^2,b_{22}^2,\dots,b_{2k_2^2}^2\}}$. Under the identification 
      \begin{equation}\label{eq g=1 separating NS boundary node det decomp}
         \mu^* \det(\mathcal W_{1,(\bar{B}^\alpha,\bar{B}^\beta),I})\big|_{\oQMb_\Gamma} = q^*\left(\det(\mathcal W_{v_1})\boxtimes \det(\mathcal W_{v_2})\right);
      \end{equation}
       we have
        \begin{equation}\label{eq orientation genus one bundle separating open open}
        \mu^*  {\mathfrak{o}}^{b_{11},b_{21}^2}_{1,(\bar{B}^\alpha,\bar{B}^\beta),I}\big |_{\oQMb_\Gamma^{}}=
        q^*\left( {\mathfrak{o}}^{b_{11}}_{0,\bar B_1,I_1}\boxtimes{\mathfrak{o}}^{h_2, b_{21}^2}_{1,(\bar{B}_2^\alpha,\bar{B}_2^\beta),I_2}\right).
         \end{equation}
         We will omit $q^*$ and $\mu^*$ from the notation in \eqref{eq orientation genus one bundle separating open open} for simplicity as $q$ and $\mu$ are isomorphisms in this case.
\end{prop}

\begin{proof}

    We denote by $M_1$, $M_2^1$ and $M_2^2$ the number of legal half-edges in $\bar B_1$, $\bar B_2^\alpha \setminus \{h_2\}$ and $\bar B_2^\beta$.
    
    In the case $\operatorname{rank} \mathcal W_{v_2}\ge 1$, or in the case $\operatorname{rank} \mathcal W_{v_2}=0$ and $M_2^2$ is odd, the argument is similar to the proof of Proposition \ref{prop: orientation open open}. We sketch this argument and point out the different point.

    We take $N_1=\operatorname{rank} \mathcal W_{v_1}$, and take $N_2^1$ and $N_2^2$ such that 
    \begin{equation}\label{eq pairty in proof of separating orientation}
    N_2^1+N_2^2=\operatorname{rank} \mathcal W_{v_2}, \text{ } N_2^1\equiv M_2^1 \mod 2,  \text{ and } N_2^2\equiv M_2^2+1 \mod 2.
    \end{equation}
    Then by \eqref{eq legal mod 2} and \eqref{eq rank of witten bundle} we have $N_1+N_2^1 \equiv M_1+M_2^1+1 \mod 2$, notice that $M_1+M_2^1$ is the number of legal half-edges in $\bar B^\alpha$.
    
    We define the orientation ${\mathfrak{o}}^{b_{11},b_{21}^2}_{1,(\bar{B}^1,\bar{B}^2),I}$ with a specific choice of $$p_\alpha,p_\beta,q_1,q_2,\dots,q_{N_1+N_2^1+N_2^2}\colon \mathcal M^{1/r,\alpha,\beta}_{1,(\bar{B}^\alpha,\bar{B}^\beta),I}\cap \oQMb\to \mathcal C$$ such that 
    \begin{itemize}
        \item $p_\alpha$ lies on the arc from $b_{2\,k_2^1}$ to $b_{11}$, $p_\beta$ lies on the arc from $b_{2\,k_2^2}^2$ to $b_{21}^2$;
        \item $p_\alpha,q_1,q_2,\dots,q_{N_1+N_2^1}$ lie on the boundary corresponding to $\bar{B}^\alpha$ in cyclic order;
        \item $p_\beta,q_{N_1+N_2^1+1},q_{N_1+N_2^1+2},\dots,q_{N_1+N_2^1+N_2^2}$ lie on the boundary corresponding to $\bar{B}^\beta$ in cyclic order;
        \item when approaching $\mathcal M_\Gamma$, the limits of $q_1,\dots,q_{N_1}$ lie on the irreducible component corresponding to $v_1$, and  the limits of $q_{N_1+1},\dots,q_{N_1+N_2^1+N_2^2}$ lie on the irreducible component corresponding to $v_2$.
    \end{itemize}  
    With such a choice, we have the corresponding sections $s_{\alpha,1},\dots,s_{\alpha,N_1+N_2^1}$ and $s_{\beta,1},\dots,s_{\beta,N_2^2}$. As in the proof of Proposition \ref{prop: orientation open open}, when approaching $\mathcal M_\Gamma$,  we want to show that the limits of $s_{\alpha,1},\dots,s_{\alpha,N_1}$ are zero when restricted to the irreducible component corresponding to $v_2$, while when restricted to the irreducible component corresponding to $v_1$ their limits are exactly the sections we used to define ${\mathfrak{o}}^{b_{11}}_{0,\bar B_1,I_1}$; similarly we want to show that the limits of $s_{\alpha,N_1+1},\dots, s_{\alpha,N_1+N_2^1}$ and $s_{\beta,1},\dots,s_{\beta,N_2^2}$  are zero when restricted to the irreducible component corresponding to $v_1$, while when restricted to the irreducible component corresponding to $v_2$ their limits are exactly the sections we used to define
    ${\mathfrak{o}}^{h_2, b_{21}^2}_{1,(\bar{B}_2^\alpha,\bar{B}_2^\beta),I_2}.$ 

    The only different point from the proof of Proposition \ref{prop: orientation open open} is, for $1\le i \le N_1$, how to show the limit of $s_{\alpha,i}$ is zero when restricted to the irreducible component corresponding to $v_2$. Actually, the limit of $s_{\alpha,i}$,  when restricted to the irreducible component corresponding $v_2$, have a total of $N_2^1+N_2^2=\operatorname{rank} \mathcal W_{v_2}$ zeros on the boundaries; this is not enough to force such restriction to be constant zero as in the genus-zero case. However, notice that, among those $N_2^1+N_2^2$ zeros, $N_2^1$ of them lie on the boundary corresponding to $\bar{B}_2^\alpha$ and $N_2^2$ of them lie on the boundary corresponding to $\bar{B}_2^\beta$, which is impossible for a non-zero section because of \eqref{eq pairty all genus} and \eqref{eq pairty in proof of separating orientation}.

     In the case where $\operatorname{rank} \mathcal W_{v_2}= 0$ and $M_2^2$ is even, the above argument doesn't work since we can not find non-negative $N_2^2$ satisfying \eqref{eq pairty in proof of separating orientation}. However, we can always reduce the problem to the case where $\operatorname{rank} \mathcal W_{1,(\bar{B}^\alpha,\bar{B}^\beta),I}\le 1$ as follows.
     
     In fact, if $\operatorname{rank} \mathcal W_{1,(\bar{B}^1,\bar{B}^2),I}= \operatorname{rank} \mathcal W_{v_1}\ge 2$, we can always find a graph $\Lambda\in \partial v_1$ consist of $n+1$ genus-zero open vertices $v'_1$ and $u_1,\dots,u_n$, where the half-edge $h_1$ is attached to $v'_1$, each $u_i$ is connected to $v'_1$ via a NS boundary edge $f_i$ (we denote its half-edge on the $u_i$ by $h_{u_i}$ and the other half by $h'_{u_i}$), and $\operatorname{rank} \mathcal W_{v'_1}=1$. We denote by $\Xi\in \partial \Gamma$ the graph obtained by connecting $\Lambda$ and $v_2$ via the edge $e$ with half-edges $h_1$ and $h_2$; we also denote by $\Xi_1$ the genus-one component of $\Xi$ after detaching all the edges $f_i$, \textit{i.e.} a graph consisting of two vertices $v'_1$ and $v_2$ connected by the edge $e$.
     When we restrict our orientation problem to the Witten bundle over $\oQMb_\Xi\subset \oQMb_{\Gamma}$, we can apply \eqref{eq orientation bundle open open} repeatedly to the edges $f_i$ to write the orientation ${\mathfrak{o}}^{b_{11}}_{0,\bar B_1,I_1}\big \vert_{\Mbar_{\Lambda}}$ of $\mathcal W_{v_1}$ as 
     $$
     {\mathfrak{o}}^{b_{11}}_{0,\bar B_1,I_1}\big \vert_{\Mbar_{\Lambda}}=(-1)^{\epsilon_1}{\mathfrak{o}}^{h_{u_1}}_{0,\bar B_{u_1},I_{u_1}}\boxtimes {\mathfrak{o}}^{h_{u_2}}_{0,\bar B_{u_2},I_{u_2}}\boxtimes \dots \boxtimes {\mathfrak{o}}^{h_{u_i}}_{0,\bar B_{u_i},I_{u_i}}\boxtimes {\mathfrak{o}}^{b'_{11}}_{0,\bar B_{v'_1},I_{v'_1}},
     $$
     where $b'_{11}$ is either $b_{11}$ itself when $b_{11}$ is attached to $v'_1$, or the half-edge on $v'_1$ corresponding to the vertex $u_i$ which $b_{11}$ attached to,
     and the sign correction $(-1)^{\epsilon_1}$ determined by \eqref{eq change of orientation of witten when cross a legal in g 0} when changing the chosen point. 
     On the other hand, since $\operatorname{rank} \mathcal W_{v'_1}=1$, the Witten bundle over the genus-one part of $\operatorname{Detach}_{S}\Xi$ for any subset $S \subseteq \{f_i\}_{1\le i \le n}$ has positive rank, therefore we can apply \eqref{eq orientation genus one bundle separating open open} (already proven in the case where the Witten bundle has positive rank on the genus-one component) for the edges $f_i$ repeatedly to write 
     $${\mathfrak{o}}^{b_{11},b_{21}^2}_{1,(\bar{B}^\alpha,\bar{B}^\beta),I}\big |_{\oQMb_\Xi}=(-1)^{\epsilon_2}{\mathfrak{o}}^{h_{u_1}}_{0,\bar B_{u_1},I_{u_1}}\boxtimes {\mathfrak{o}}^{h_{u_2}}_{0,\bar B_{u_2},I_{u_2}}\boxtimes \dots \boxtimes {\mathfrak{o}}^{h_{u_i}}_{0,\bar B_{u_i},I_{u_i}}\boxtimes {\mathfrak{o}}^{b'_{11},b_{21}^2}_{1,(\bar{B}^\alpha_{d_e\Xi_1},\bar{B}^\beta_{d_e\Xi_1}=\bar{B}^\beta),I_{d_e\Xi_1}}\big |_{\oQMb_{\Xi_1}},$$ 
     where the sign correction $(-1)^{\epsilon_2}$ is determined by \eqref{eq change of orientation of witten when cross a legal in g 0} and \eqref{eq change of orientation of witten when cross a legal in g 1 1st bdry} when changing the chosen point. Assuming we can show $(-1)^{\epsilon_1}=(-1)^{\epsilon_2}$ and \eqref{eq orientation genus one bundle separating open open} in the reduce case $${\mathfrak{o}}^{b'_{11},b_{21}^2}_{1,(\bar{B}^\alpha_{d_e\Xi_1},\bar{B}^\beta_{d_e\Xi_1}=\bar{B}^\beta),I_{d_e\Xi_1}}\big |_{\oQMb_{\Xi_1}}={\mathfrak{o}}^{b'_{11}}_{0,\bar B_{v'_1},I_{v'_1}}\boxtimes {\mathfrak{o}}^{h_2, b_{21}^2}_{1,(\bar{B}_2^\alpha,\bar{B}_2^\beta),I_2}$$
     then \eqref{eq orientation genus one bundle separating open open} will follow.

    Now we compute the corrections $(-1)^{\epsilon_1}$. We start from the case $n=1$. We denote take $p$ to be a point on the boundary of $\Sigma_{u_1}\in \mathcal M_{u_1}$ such that there exist no markings on the arc from $h_{u_1}$ to $p$. Then by \eqref{eq orientation bundle open open} we have 
    $$
    {\mathfrak{o}}^{p}_{0,\bar B_1,I_1}\big \vert_{\Mbar_{\Lambda}}={\mathfrak{o}}^{p}_{0,\bar B_{u_1},I_{u_1}}\boxtimes {\mathfrak{o}}^{h'_{u_1}}_{0,\bar B_{v'_1},I_{v'_1}}.
    $$
    By \eqref{eq change of orientation of witten when cross a legal in g 0} we have
    $$
    {\mathfrak{o}}^{b_{11}}_{0,\bar B_1,I_1}=(-1)^{Z_{b_{11}\to p}\cdot(M_{B_1}+1)}{\mathfrak{o}}^{p}_{0,\bar B_1,I_1},
    $$
    where $Z_{b_{11}\to p}$ is the number of legal boundary markings on the arc from $b_{11}$ (included) to $p$, and $M_{B_1}$ is the number of legal half-edges in $\bar B_1$;
    we also have
     $$
    {\mathfrak{o}}^{b'_{11}}_{0,\bar B_{v'_1},I_{v'_1}}=(-1)^{Z_{b'_{11}\to h'_{u_1}}\cdot (M_{B_{v'_1}}+1)}{\mathfrak{o}}^{h'_{u_1}}_{0,\bar B_{v'_1},I_{v'_1}},
    $$
     where $Z_{b'_{11}\to h'_{u_1}}$ is the number of legal boundary markings on the arc from $b'_{11}$ (included) to $h'_{u_1}$ (not included), and $M_{B_{v'_1}}$ is the number of legal half-edges in $\bar B_{v'_1}$;
    and we also have  
    $$
    {\mathfrak{o}}^{h_{u_1}}_{0,\bar B_{u_1},I_{u_1}}= (-1)^{\alt(h_{u_1})}{\mathfrak{o}}^{p}_{0,\bar B_{u_1},I_{u_1}}
    $$
    where $ \alt(h_{u_1})=0$ if  $h_{u_1}$ is illegal, and $ \alt(h_{u_1})=1$ if  $h_{u_1}$ is legal. 
    Note that if $b'_{11}=b_{11}$ then $Z_{b_{11}\to p}=Z_{b'_{11}\to h'_{u_1}}$;  if $b'_{11}\ne b_{11}$ then $b'_{11} =h'_{u_1}$, which implies $Z_{b_{11}\to p}=Z_{b'_{11}\to h'_{u_1}}=0$. In any case we have 
    $$
    \epsilon_1\equiv Z_{b'_{11}\to h'_{u_1}} \cdot (M_{B_1}-M_{B_{v'_1}})+\alt(h_{u_1})\equiv Z_{b'_{11}\to h'_{u_1}} \cdot (M_{B_{u_1}}-\alt(h_{u_1}))+\alt(h_{u_1})\mod 2,
    $$
    where $M_{B_{u_1}}$ is the number of legal half-edges in $\bar B_{u_1}$. Therefore, $(-1)^{\epsilon_1}$ only depends on the information of boundary markings lying on the arc from $b'_{11}$ to $h'_{u_1}$ and the half-edges of $u_1$. Similarly, when $n>1$, assuming $h'_{u_1},h'_{u_2},\dots,h'_{u_n}$ are in cyclic order, then $(-1)^{\epsilon_1}$ only depends on the information of boundary markings lying on the arc from $b'_{11}$ to $h'_{u_n}$, and the half-edges of $u_1,\dots,u_n$.

    On the other hand, we can $(-1)^{\epsilon_2}$ using  \eqref{eq orientation genus one bundle separating open open}, \eqref{eq change of orientation of witten when cross a legal in g 0} and \eqref{eq change of orientation of witten when cross a legal in g 1 1st bdry}. $(-1)^{\epsilon_2}$ depends on the same information as $(-1)^{\epsilon_1}$ dose, in the exactly same way, which means $(-1)^{\epsilon_2}=(-1)^{\epsilon_2}$.
    
     To conclude the proof, we need to verify \eqref{eq orientation genus one bundle separating open open} in the case $\operatorname{rank} \mathcal W_{1,(\bar{B}^\alpha,\bar{B}^\beta),I}=0 $ or $1$. Both of them can be checked easily by definition.
\end{proof}

\begin{cor}\label{cor:orientation open open g=1}
    With the same notation as in Proposition \ref{prop: orientation open open g=1 separating}, assuming every boundary marking in $B^1$ and $B^2$ is legal, in the case $h_1$ illegal and $h_2$ legal we have 
    \begin{equation}\label{eq relative open open g=1 version 1}
        {{o}}_{1,\{\bar{B}^1,\bar{B}^2\},I}\big |_{\Mbar_\Gamma^{}}=o_N\otimes\left(
         {{o}}^{b_{11}}_{0,\bar B_1,I_1}\boxtimes{{o}}^{}_{1,\{\bar{B}_2^1,\bar{B}_2^2\},I_2}\right),
         \end{equation}
         where $N$ is the outward normal with canonical orientation $o_N$; in the case $h_2$ illegal and $h_1$ legal we have 
         \begin{equation}\label{eq relative open open g=1 version 2}
        {{o}}_{1,\{\bar{B}^1,\bar{B}^2\},I}\big |_{\Mbar_\Gamma^{}}=o_N\otimes\left( {{o}}^{b_{21}^1, b_{21}^2}_{1,(\bar{B}_2^1,\bar{B}_2^2),I_2} \boxtimes{{o}}^{}_{0,\bar B_1,I_1}
           \right).
         \end{equation}
\end{cor}
\begin{proof}
    The equation \eqref{eq relative open open g=1 version 1}  follows from \eqref{eq definition canonical orientation}, \eqref{eq definition canonical orientation g=1}, \eqref{eq orientation moduli open open g=1} and \eqref{eq orientation genus one bundle separating open open}.  The factors of the form $(-1)^{m^\delta}$ disappear since
    $$
        m^\delta(1,\{\bar{B}^1,\bar{B}^2\},I)=m^\delta(0,\bar{B_1},I_1)+m^\delta(1,\{\bar{B}_2^1,\bar{B}_2^2\},I_2).
    $$
    The factor $(-1)^{(k_1-1)k_2^1}$ in \eqref{eq orientation moduli open open g=1} disappears because 
    $$m^{g=1}(1,\{\bar{B}^1,\bar{B}^2\},I)-m^{g=1}(1,\{\bar{B}_2^1,\bar{B}_2^2\},I_2)=(k_2^2+1)(k_1-1),$$
    while commuting $\tilde{\mathfrak{o}}^{h_2, b_{21}^2}_{1,\{\bar{B}_2^1,\bar{B}_2^2\},I_2}$ with ${\mathfrak{o}}^{b_{11}}_{0,\bar B_1,I_1}$ gives $(-1)^{(k_1-1)(k_2^1+k_2^2+1)}$. The proof of equation \eqref{eq relative open open g=1 version 2} is similar.
\end{proof}

\begin{prop}\label{prop: orientation open open g=1 non-separating}
      Let $\Gamma$ be a genus-one stable graded $r$-spin graph consisting of an open genus-zero vertex $v$ and a NS non-separating boundary edge $e$ connecting $v$ to itself, with two half-edges $h_1$ and $h_2$. We assume $h_1$ is illegal and $h_2$ is legal. Let $I$ be the sets of internal markings of $v$; we write $\bar{B}_v$, the boundary half-edges of $v$, in its cyclic order as $\bar{B}_v=\overline{\{b_{11},b_{12},\dots,b_{1k_1},h_2,b_{21},b_{22},\dots,b_{2k_1},h_1\}}$. Then the set of internal markings of $d_e \Gamma$ is also $I$, and the sets of boundary markings of $d_e \Gamma$ on each boundaries written in their cyclic order is $\bar{B}^\alpha=\overline{\{b_{11},b_{12},\dots,b_{1k_1}\}}$ and $\bar{B}^\beta=\overline{\{b_{21},b_{22},\dots,b_{2k_1}\}}$. Under the identification
      \begin{equation}\label{eq g=1 non-separating NS boundary node det decomp}
          \mu^* \det(\mathcal W_{1,(\bar{B}^\alpha,\bar{B}^\beta),I})\big|_{\Mbar_\Gamma} = \det(\mathcal W_{v})
      \end{equation}
       we have
        \begin{equation}\label{eq orientation genus one bundle non-separating}
        \mu^* {\mathfrak{o}}^{b_{11},b_{21}}_{1,(\bar{B}^\alpha,\bar{B}^\beta),I}\big |_{\Mbar_\Gamma^{}}=
         q^*{\mathfrak{o}}^{b_{11}}_{0,\bar B_v,I}.
         \end{equation}
         We will omit $q^*$ and $\mu^*$ from the notation in \eqref{eq orientation genus one bundle non-separating} for simplicity as $q$ is an isomorphism in this case, and the (surjective) gluing morphism $\mu\circ q^{-1}\colon \Mbar_v\to \Mbar_{\Gamma}$ is an isomorphism when restricted as a morphism  $\mu\circ q^{-1}\colon \mathcal{M}_v\to \mathcal{M}_\Gamma$ between the dense open subspaces.
\end{prop}

\begin{proof}
    
    When $\operatorname{rank} \mathcal W_{1,(\bar{B}^\alpha,\bar{B}^\beta),I}=0$, the equation \eqref{eq orientation genus one bundle non-separating} is equivalent to the definition of orientation in this case. We assume $\operatorname{rank} \mathcal W_{1,(\bar{B}^\alpha,\bar{B}^\beta),I}\ge 1$.

    We denote by $M_1$ and $M_2$ the number of legal half-edges in $\bar B^\alpha$ and $\bar B^\beta$.
    We take integers $N_1$ and $N_2$ such that $N_1+N_2=\operatorname{rank} \mathcal W$, $N_1\equiv M_1 +1 \mod 2$, and $N_2 \equiv M_2 +1 \mod 2$.

    Unlike the proof of Proposition \ref{prop: orientation open open} and \ref{prop: orientation open open g=1 separating}, we start from a specific choice of $$p_\alpha,p_\beta,q_1,q_2,\dots,q_{N_1+N_2}\colon \mathcal M_\Gamma \to \mathcal C$$ such that 
    \begin{itemize}
        \item $p_\alpha$ lies on the arc from $h_1$ to $b_{11}$, $p_\beta$ lies on the arc from $h_2$ to $b_{21}$;
        \item $q_{1},q_{2},\dots,q_{N_1}$ lie (in order) on the arc from $p_\alpha$ to $h_2$;
        \item $q_{N_1+1},q_{N_1+2},\dots,q_{N_1+N_2}$ lie (in order) on the arc from $p_\beta$ to $h_1$;
    \end{itemize}  
    Then we can pull back them to $\mathcal M_v$ via the gluing morphism $\mathcal M_v\to \mathcal M_{\Gamma}$ and construct sections $s_1,\dots,s_{N_1+N_2}$ of $\lvert \mathcal J \rvert \to \mathcal M_v=\mathcal M_{0,\bar B_v,I}$ according to such choice of $p_\alpha,q_1,\dots,q_{N_1+N_2}$, and write 
    $${\mathfrak{o}}^{b_{11}}_{0,\bar B_v,I}=s_1\wedge s_2 \wedge \dots \wedge s_{N_1+N_2}.$$

    We extend the above choice of $p_\alpha,p_\beta,q_1,q_2,\dots,q_{N_1+N_2}$ to a neighbourhood of $\mathcal{M}_\Gamma$ in $\mathcal M_{1,\{\bar{B}^1,\bar{B}^2\},I}$, then we get sections $s_{1,1},\dots,s_{1,N_1}$ and $s_{2,1},\dots, s_{2,N_2}$ of $\lvert \mathcal J \rvert \to \mathcal M_{1,\{\bar{B}^1,\bar{B}^2\},I}$ in this neighbourhood accordingly, and we can write
    $$
    {\mathfrak{o}}^{b_{11},b_{21}}_{1,(\bar{B}^\alpha,\bar{B}^\beta),I}=s_{1,1}\wedge\dots \wedge s_{1,N_1} \wedge s_{2,1}\wedge \dots \wedge s_{2, N_2}.
    $$

    Note that for each $1\le i \le N_1$, both $s_i$ and $s_{1,i}$ vanish at $\{q_1,q_2,\dots,q_{N_1+N_2}\}\setminus \{q_i\}$, and evaluate negatively at $p_\alpha$, thus $s_{1,i}$ are extensions of (the gluing of) $s_i$ in a neighbourhood of $\mathcal M_\Gamma$ for $1\le i \le N_1$.

    To complete the proof by showing $s_{2,i}$ are extensions of $s_{N_1+i}$ in a neighbourhood of $\mathcal M_\Gamma$ for $1\le i \le N_2$, we need to show $s_{N_1+i}$ evaluate negatively at $p_\beta$. Actually, we know that $s_{N_1+i}$ evaluate negatively at $p_\alpha$; on the arc form $p_\alpha$ to $p_\beta$ there are $M_1+1$ legal boundary markings (which are legal markings in $\bar B^1$ with an additional $h_2$) where the grading alters, and $N_1$ zeros of $s_{N_1+i}$, which means $s_{N_1+i}(p_\alpha)$ and $s_{N_1+i}(p_\beta)$ have the same sign with respect to the grading since $M_1+1+N_1\equiv 0 \mod 2$. 
\end{proof}

\begin{cor}\label{cor:orientation relative nonseperating}
    With the same notation as in Proposition \ref{prop: orientation open open g=1 non-separating}, assuming every boundary marking in $B^1$ and $B^2$ is legal, we have 
    \begin{equation}\label{eq relative orientation genus one non-separating}
         {{o}}^{}_{1,\{\bar{B}^\alpha,\bar{B}^\beta\},I}\big |_{\Mbar_\Gamma^{}}=
         {{o}}^{b_{11}}_{0,\bar B_v,I}.
         \end{equation}
         where $N$ is the outward normal with canonical orientation $o_N$.
\end{cor}
\begin{proof}
    The equation \eqref{eq relative orientation genus one non-separating}  follows from \eqref{eq definition canonical orientation}, \eqref{eq definition canonical orientation g=1}, \eqref{eq orientation genus one moduli non-separating} and \eqref{eq orientation genus one bundle non-separating}.  The factors of the form $(-1)^{m^\delta}$ disappear since
    $$
        m^\delta(1,\{\bar{B}^\alpha,\bar{B}^\beta\},I)=m^\delta(0,\bar{B}_v,I).
    $$
    The sign in \eqref{eq orientation genus one moduli non-separating} is the same as
    $m^{g=1}(1,\{\bar{B}^\alpha,\bar{B}^\beta\},I)$.
\end{proof}

\begin{prop}\label{prop:orientation closed open}
    Let $\Gamma$ be a stable graded $r$-spin graph with two vertices, an open vertex $v^o$ and a closed vertex $v^c$, connected by an edge $e$.  We denote by $I,I^o$ and $I^c$ the sets of internal half-edges of $d_e\Gamma,v^o$, and $v^c$, by $B$ the common set of boundary half-edges of $d_e\Gamma$ and $v^o$. We have \begin{equation}\label{eq det decomp open open}
        \mu^*\det(\mathcal W_{0,\bar{B},I})\big|_{\Mbar_\Gamma} =  q^*\left(\det(\mathcal W_{v^o})\boxtimes \det(\mathcal W_{v^c})\right).
    \end{equation}
    Moreover, for any boundary marking $b\in B$, we have 
        \begin{equation}\label{eq orientation bundle closed open}
        \mu^*{\mathfrak{o}}^b_{0,\bar B,I}|_{\Mbar_\Gamma} = (-1)^{m^c(I^c)}q^*\left({\mathfrak{o}}^b_{0,\bar B,I^o}\boxtimes {\mathfrak{o}}_{v^c}\right),
        \end{equation}
       where  
        $$
        m^c(I^c):=\operatorname{rank}_{\mathbb C} \mathcal W_{v^c}=\frac{1}{2}\operatorname{rank}_{\mathbb R} \mathcal W_{v^c},
        $$
        and ${\mathfrak{o}}_{v^c}$ is the canonical complex orientations. We will omit $q^*$ and $\mu^*$ from the notation in \eqref{eq orientation bundle closed open} for simplicity as $\mu$ is an isomorphism in this case, and even though the degree-one morphism $q$ may not induce an isomorphism on automorphism groups in this case, these actions preserve orientation.

\end{prop}
\begin{proof}

     Now we prove \eqref{eq orientation bundle closed open}. By the exact same argument as in the proof of Lemma 5.14 in \cite{BCT1}, we can reduce the problem to the case $I=I^c=\{a_1,a_2\}$ and $B=\{b\}$, where $r\le a_1+a_2\le 2r-1$ and $2a_1+2a_2+b=3r-2$. We prove the proposition by writing down, in these cases, the explicit sections of $\mathcal W\to \mathcal M^{1/r}_{0,\bar{B},I}$ in the upper half-plane model.

    We denote by $x_b,z_1,z_2$ the marked points that correspond to $b,a_1,a_2$, and let $p,q_1,q_2$ be the chosen boundary points to define $\mathfrak o(p,q_1,q_2)=s_1\wedge s_2$. We assume $p<x_b<q_1<q_2$ without loss of generality. By setting $q_b=x_b$ for convenience, we define global $\phi$-invariant sections of $\omega_{\lvert C\rvert}$ by 
    $$
    \xi_{ij}:=\frac{(q_j-q_i)d\omega}{(\omega-q_i)(\omega-q_j)}, i,j\in\{1,2,b\},i\ne j
    $$
    and
    $$
    \xi_i:=\frac{\sqrt{-1}(\bar z_i -z_i)d\omega}{(\omega-z_i)(\omega-\bar z_i)}, i\in \{1,2\},
    $$
    where $\sqrt{-1}$ is the root in the upper half-plane. They are well-defined since the above formulae are invariant under the $PSL(2,R)$-action.
    
    We define
    $$
    \Xi_1:=\xi_1^{a_1} \xi_2^{a_2} \xi_{b1}^{\frac{b+r}{2}} \xi_{b2}^{\frac{b-r}{2}}\xi_{12}^{-\frac{b+r}{2}},
    $$
    which is a global section (over $\mathcal M^{1/r}_{0,\bar{B},I}$) of 
    $$
    \omega^{\otimes r-1}_{\lvert C\rvert}\otimes \mathcal O\left(a_1[z_1]+a_2[z_2]+a_1[\bar z_1]+a_2[\bar z_2]+b[x_b] \right)\cong \lvert J\rvert^{\otimes r}.
    $$ 
    Note that $\Xi_1$ only has one zero at $q_2$ with order $r$, thus $s_1$ is an $r$-th root of $\Xi_1$. Similarly, $s_2$ is an $r$-th root of 
    $$
    \Xi_2:=\xi_1^{a_1} \xi_2^{a_2} \xi_{b1}^{\frac{b-r}{2}} \xi_{b2}^{\frac{b+r}{2}}\xi_{12}^{-\frac{b+r}{2}}.
    $$
     The stratum $\overline{\mathcal M}_\Gamma$ consists of a single point where $z_1=z_2$. We denote by $\hat s_1,\hat s_2$ the limit of $s_1,s_2$ at this point. Since $\operatorname{rank} W_{v^o}=0$, by an abuse of notation, we also denote by $\hat s_1,\hat s_2$ their restriction to the closed irreducible component. Note that $ {\mathfrak{o}}^b_{0,\bar B,I^o}$ is always positive by definition in the $\operatorname{rank} \mathcal W=0$ case, we need to compare $\hat s_1\wedge\hat s_2$ with the complex orientation of the closed Witten bundle $\mathcal W_{v^c}$. Since $\operatorname{rank_{\mathbb C}} \mathcal W_{v^c}=1$, the ratio $\frac{\hat s_2}{\hat s_1}$ is a well-defined complex number. Moreover, the orientation $\hat s_1\wedge\hat s_2$ is opposite to the complex orientation if and only if $\frac{\hat s_2}{\hat s_1}$ lies in the lower half-plane.  By construction we have 
     $$
        \frac{\hat s_2}{\hat s_1}=\lim_{z_2\to z_1}\frac{s_2}{s_1}\Big\vert_{\omega=z_1}.
     $$
     Since $\frac{s_2}{s_1}$ is an $r$-th root of $\frac{\Xi_2}{\Xi_1}=\left(\frac{q_2-x_b}{q_1-x_b}\cdot\frac{\omega-q_1}{\omega-q_2}\right)^r$ which is positive on $p$, so we have
     $$
      \frac{s_2}{s_1}=\frac{q_2-x_b}{q_1-x_b}\cdot\frac{\omega-q_1}{\omega-q_2}.
     $$
     For all $\omega$ in the upper half-plane, $\frac{s_2}{s_1}$ lies in the  lower half-plane, so does $\frac{\hat s_2}{\hat s_1}$. Therefore we have
     $$
     {\mathfrak{o}}^b_{0,\bar B,I}|_{\Mbar_\Gamma} =\hat s_1\wedge\hat s_2=-{\mathfrak{o}}_{v^c} =(-1)^{m^c(I^c)}{\mathfrak{o}}^b_{0,\bar B,I^o}\boxtimes {\mathfrak{o}}_{v^c}.
     $$
     
\end{proof} 

\begin{cor}\label{cor:orientation closed open}
     With the same notation as in Proposition \ref{prop:orientation closed open}, when all boundary markings in $B$ are legal, we have 
     \begin{equation}\label{eq orientation relative closed open}
        {{o}}_{0,\bar B,I}|_{\Mbar_\Gamma} = o_N\otimes ({{o}}_{0,\bar B,I^o}\boxtimes {{o}}_{v^c}),
        \end{equation}
        where $o_N$ and ${{o}}_{v^c}$ are the canonical complex orientations.
\end{cor}
\begin{proof}
    The corollary follows from the combination of \eqref{eq definition canonical orientation}, \eqref{eq orientation moduli closed open}, and \eqref{eq orientation bundle closed open}. The factors $(-1)^{m^\delta}$ and $(-1)^{m^c}$ disappear because
    $$
    m^\delta(0,\bar B, I)=m^c(I^c)+ m^\delta(0,\bar B, I^o).
    $$
    No signs appears when commuting $ \tilde{\mathfrak{o}}_{v^c}$ with ${\mathfrak{o}}^b_{0,\bar B,I^o}$ because the dimension of $\mathcal M_{v^c}$ is even.
\end{proof}

\begin{prop}\label{prop:orientation bundle closed open genus one}
    Let $\Gamma$ be a genus-one  stable graded $r$-spin graph with two vertices, an open genus-one vertex $v^o$ and a closed genus-zero vertex $v^c$, connected by an separating internal edge $e$.  We denote by $I,I^o$ and $I^c$ the sets of internal half-edges of $d_e\Gamma,v^o$, and $v^c$, by $\bar B^\alpha$ and $\bar B^\beta$ the common set of boundary half-edges of $d_e\Gamma$ and $v^o$ on each boundary of the cylinder. Under the identification \begin{equation}\label{eq det decomp open close genus one}
        \mu^*\det(\mathcal W_{1,(\bar{B}^\alpha,\bar{B}^\beta),I})\big|_{\oQMb_\Gamma} =  q^*\left(\det(\mathcal W_{v^o})\boxtimes \det(\mathcal W_{v^c})\right),
    \end{equation}
    for any boundary markings $x\in \bar B_1$ and $y\in \bar B_2$, we have 
        \begin{equation}\label{eq orientation bundle closed open genus one}
        \mu^*{\mathfrak{o}}^{x,y}_{1,(\bar{B}^\alpha,\bar{B}^\beta),I}|_{\oQMb_\Gamma} = (-1)^{m^c(I^c)}q^*\left({\mathfrak{o}}^{x,y}_{1,(\bar{B}^\alpha,\bar{B}^\beta),I^o}\boxtimes {\mathfrak{o}}_{v^c}\right),
        \end{equation}
       where  
        $$
        m^c(I^c):=\operatorname{rank}_{\mathbb C} \mathcal W_{v^c}=\frac{1}{2}\operatorname{rank}_{\mathbb R} \mathcal W_{v^c},
        $$
        and ${\mathfrak{o}}_{v^c}$ is the canonical complex orientations. We will omit $q^*$ and $\mu^*$ from the notation in \eqref{eq orientation bundle closed open genus one} for simplicity as $\mu$ is an isomorphism in this case, and even though the degree-one morphism $q$ may not  induce an isomorphism on automorphism groups in this case, these actions preserve orientation.

\end{prop}

\begin{proof}
    We only need to check \eqref{eq orientation bundle closed open genus one} for the fibre of $\mathcal W$ over a specific point in $\oQMb_\Gamma$ and a specific choice of $x,y$. We take a point in     $\Mbar_{\Xi}\subset \oQMb_{\Gamma}$,     where $\Xi\in \partial \Gamma$ consist of two vertices: a closed genus-zero vertex $v^c$ (which is the same as the closed vertex in $\Gamma$) and an open genus-zero vertex $v^o_0$, where $v^c$ and $v^o_0$ are connected by an internal edge $e^I$, while $v^o_0$ is also connected to itself via a NS non-separating  boundary edge $e^B$. Such $\Xi$ always exists when $\Gamma$ in not dimension-jump.

    We denote by $h_1,h_2$ the half-edges of $e^B$. Similar to the notation in Proposition \ref{prop: orientation open open g=1 non-separating}, we write $\bar{B}^\alpha=\overline{\{b_{11},b_{12},\dots,b_{1k_1}\}}$ and $\bar{B}^\beta=\overline{\{b_{21},b_{22},\dots,b_{2k_1}\}}$, and write $\bar{B}_{v^o}$, the set of boundary half-edges of $v^o$, in its cyclic order as $\bar{B}_{v^o}=\overline{\{b_{11},b_{12},\dots,b_{1k_1},h_2,b_{21},b_{22},\dots,b_{2k_1},h_1\}}$. Note that $d_{e^I}\Xi$ consist of a single genus-zero open vertex, which we denote by ${v^o}'$, and the non-separating NS boundary edge $e^B$ connect ${v^o}'$ to itself; moreover, the set of boundary half-edges of ${v^o}'$ is also $\bar{B}_{v^o}$, and the set of internal half-edges of ${v^o}'$ is $I$.

    On one hand we have
    \begin{equation*}
    \begin{split}
        {\mathfrak{o}}^{b_{11},b_{21}}_{1,(\bar{B}^\alpha,\bar{B}^\beta),I}|_{\Mbar_\Xi} =&\left({\mathfrak{o}}^{b_{11},b_{21}}_{1,(\bar{B}^\alpha,\bar{B}^\beta),I}|_{\oQMb_\Gamma}\right)\big\vert_{\Mbar_{\Xi}},
    \end{split}
    \end{equation*}
   On the other hand, \eqref{eq orientation bundle closed open} and \eqref{eq orientation genus one bundle non-separating} implies 
    \begin{equation*}
    \begin{split}
        {\mathfrak{o}}^{b_{11},b_{21}}_{1,(\bar{B}^\alpha,\bar{B}^\beta),I}|_{\Mbar_\Xi} =&\left({\mathfrak{o}}^{b_{11},b_{21}}_{1,(\bar{B}^\alpha,\bar{B}^\beta),I}|_{\Mbar_{d_{e^I}\Xi}}\right)\big\vert_{\Mbar_{\Xi}}\\
        =&  \left( {\mathfrak{o}}^{b_{11}}_{0,\bar B_v,I} \right)\big \vert_{\Mbar_{\operatorname{Detach}_{e^B} \Xi}}
        \\=&(-1)^{m^c(I^c)}{\mathfrak{o}}^{b_{11}}_{0,\bar{B}_v,I^o}\boxtimes {\mathfrak{o}}_{v^c}.
    \end{split}
    \end{equation*}
    Thus  \eqref{eq orientation bundle closed open genus one} hold when restricted to $\Mbar_\Xi$, hence on entire the $\oQMb_\Gamma$.

\end{proof}

\begin{cor}\label{cor:orientation closed open g=1}
     With the same notation as in Proposition \ref{prop:orientation bundle closed open genus one}, when all boundary markings in $B^\alpha$ and $B^\beta$ are legal, we have 
     \begin{equation}\label{eq orientation relative closed open g=1}
        {\mathfrak{o}}_{1,\{\bar{B}^\alpha,\bar{B}^\beta\},I}|_{\oQMb_\Gamma} = {\mathfrak{o}}_{1,\{\bar{B}^\alpha,\bar{B}^\beta\},I^o}\boxtimes {\mathfrak{o}}_{v^c},
        \end{equation}
        where $o_N$ and ${{o}}_{v^c}$ are the canonical complex orientations.
\end{cor}
\begin{proof}
    The corollary follows from the combination of \eqref{eq definition canonical orientation g=1}, \eqref{eq orientation moduli closed open g=1}, and \eqref{eq orientation bundle closed open genus one}. Note that we have
    $$
    m^\delta(1,\{\bar B^\alpha, \bar B^\beta\}, I)=m^c(I^c)+ m^\delta(1,\{\bar B^\alpha, \bar B^\beta \}, I^o),
    $$
    and
    $$
    m^{g=1}(1,\{\bar B^\alpha, \bar B^\beta\}, I)=m^{g=1}(1,\{\bar B^\alpha, \bar B^\beta \}, I^o).
    $$
    No signs appears when commuting $ \tilde{\mathfrak{o}}_{v^c}$ with ${\mathfrak{o}}^{x,y}_{1,(\bar{B}^\alpha,\bar{B}^\beta),I^o}$ because the dimension of $\mathcal M_{v^c}$ is even.
\end{proof}

\section{Point insertion}\label{sec:point_ins}

In this section, we introduce the point insertion technique.

\subsection{Motivation and examples}
The genus-zero open $r$-spin theory studies the intersection theory over the moduli of $r$-spin disks $\Mbar^{1/r}_{0,B,I},$ just as the closed $r$-spin theory considers the intersection theory over the moduli spaces of $r$-spin curves. However, since $\Mbar^{1/r}_{0,B,I}$ is an orbifold with corners, the intersection theory is not well-defined. The grading structure allows us to deal in different ways with the problems that occur due to the presence of the boundaries.
As we shall see in the sequel \cite{TZ2}, certain boundary components are treated via some positivity phenomenon, which relies on the grading. The procedure of \textit{point insertion}, which also depends on the grading, is aimed to deal with the remaining boundaries: we can glue another moduli to $\Mbar^{1/r}_{0,B,I}$ along boundaries that cannot be dealt with via positivity to cancel out these boundaries. 
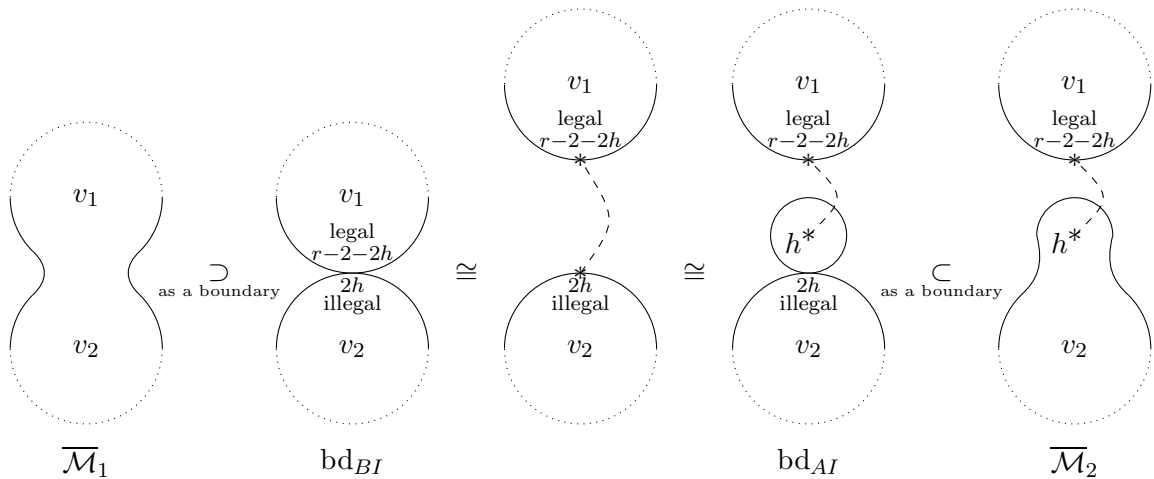
\begin{figure}[h]
    \centering
    \begin{tikzpicture}
        \draw (0,0) arc (0:180:1);
        \draw[dotted] (0,0) arc (0:-180:1);
        \node at (-1,1) {$*$};
        \node at (-1,0.7) {$\substack{2h\\ \text{illegal}}$};
        \node at (-1,0) {$v_2$};
                
        \draw[dotted] (0,3.5) arc (0:180:1);
        \draw (0,3.5) arc (0:-180:1);
        \node at (-1,2.5) {$*$};
        \node at (-1,2.9) {$\substack{\text{legal}\\r-2-2h}$};
        \node at (-1,3.5) {$v_1$};

        \draw[dashed] (-1,1) .. controls (-0.5,1.7) .. (-1,2.5);

        \draw (-3,0) arc (0:180:1);
        \draw[dotted] (-3,0) arc (0:-180:1);
  
        \node at (-4,0.7) {$\substack{2h\\ \text{illegal}}$};
        \node at (-4,0) {$v_2$};
        \node at (-4,-1.5) {$\text{bd}_{BI}$};

        \draw[dotted] (-3,2) arc (0:180:1);
        \draw (-3,2) arc (0:-180:1);
  
        \node at (-4,1.4) {$\substack{\text{legal}\\r-2-2h}$};
        \node at (-4,2) {$v_1$};

        \node at (-2.5,1) {$\cong$};

        \draw (3,0) arc (0:180:1);
        \draw[dotted] (3,0) arc (0:-180:1);
        \node at (2,1.5) {$*$};
        \node at (1.8,1.4) {$h$};
        \node at (2,0.7) {$\substack{2h\\ \text{illegal}}$};
        \node at (2,0) {$v_2$};
        \node at (2,-1.5) {$\text{bd}_{AI}$};
                
        \draw[dotted] (3,3.5) arc (0:180:1);
        \draw (3,3.5) arc (0:-180:1);
        \node at (2,2.5) {$*$};
        \node at (2,2.9) {$\substack{\text{legal}\\r-2-2h}$};
        \node at (2,3.5) {$v_1$};
        
        \draw (2.5,1.5) arc (0:360:0.5);
        \draw[dashed] (2,1.5) .. controls (2.5,2) .. (2,2.5);

        \node at (0.5,1) {$\cong$};

         \draw (-6.5,0) arc (0:45:1);
         \draw (-8.5,0) arc (180:135:1);
        \draw[dotted] (-6.5,0) arc (0:-180:1);
    
        \node at (-7.5,0) {$v_2$};
        \node at (-7.5,-1.5) {$\Mbar_1$};

        \draw[dotted] (-6.5,2) arc (0:180:1);
        \draw (-6.5,2) arc (0:-45:1);
        \draw (-8.5,2) arc (-180:-135:1);
    
        \node at (-7.5,2) {$v_1$};
        \draw (-6.793,1.293) .. controls (-7,1.1) and (-7,0.9).. (-6.793,0.707);
        \draw (-8.207,1.293) .. controls (-8,1.1) and (-8,0.9).. (-8.207,0.707);

        \node at (-5.75,1) {$\supset$};
        \node at (-5.75,0.75) {\tiny{as a boundary}};

        \draw (6.5,0) arc (0:45:1);
        \draw (4.5,0) arc (180:135:1);
        \draw[dotted] (6.5,0) arc (0:-180:1);
        \node at (5.5,1.5) {$*$};
        \node at (5.3,1.4) {$h$};
  
        \node at (5.5,0) {$v_2$};
        \node at (5.5,-1.5) {$\Mbar_2$};
                
        \draw[dotted] (6.5,3.5) arc (0:180:1);
        \draw (6.5,3.5) arc (0:-180:1);
        \node at (5.5,2.5) {$*$};
        \node at (5.5,2.9) {$\substack{\text{legal}\\r-2-2h}$};
        \node at (5.5,3.5) {$v_1$};
        
        \draw (6,1.5) arc (0:180:0.5);
  
        \draw[dashed] (5.5,1.5) .. controls (6,2) .. (5.5,2.5);
        
        \draw (6.207,0.707) .. controls (6,0.9) and (5.9,1.1).. (6,1.5);
        \draw (4.793,0.707) .. controls (5,0.9) and (5.1,1.1).. (5,1.5);

        \node at (3.75,1) {$\subset$};
        \node at (3.75,0.75) {\tiny{as a boundary}};
        
    \end{tikzpicture}
    \caption{In point insertion procedure we glue $\Mbar_1$ and $\Mbar_2$ together along their isomorphic boundaries $\text{bd}_{BI}$ and $\text{bd}_{AI}$. The first isomorphism follows from the decomposition property for the boundary NS nodes; the second isomorphism holds because the moduli $\Mbar^{1/r}_{0.\{r-2-2h\},\{h\}}$ (represented by the smallest bubble in the figure) is a single point. The new marked points coming from the point insertion procedure are represented by $*$; the dashed line between the new marked points indicates that they come from the same node.} 
    \label{fig:point insertion demonstration}
\end{figure}

More precisely, as shown in Figure \ref{fig:point insertion demonstration}, let $\Mbar_1$ be a moduli of $r$-spin disks, and $\text{bd}_{BI}\subset \Mbar_1$ be a boundary corresponding to an NS boundary node with twist $2h$ at the illegal half-node. We can glue to $\Mbar_1$, along the boundary $\text{bd}_{BI}$, another moduli $\Mbar_2$ which has a boundary $\text{bd}_{AI}$ isomorphic to $\text{bd}_{BI}$. Note that $\Mbar_2$ is a moduli of two disconnected $r$-spin disks, obtained by first detaching the boundary node, then "inserting" the illegal twist-$2h$ boundary marked point into the interior as a twist-$h$ internal marked point. 
At corners we can perform several point insertions at the same time, in a consistent way.

By applying this procedure, which we term \emph{the point insertion scheme}, repeatedly and iteratively, we get a glued moduli whose only remaining boundaries can be dealt with using the positivity phenomenon mentioned above.

We will define $\lfloor \frac{r}{2}\rfloor$ different point insertion theories, indexed by an integer $\h\in \{0,1,\dots,\lfloor \frac{r-2}{2}\rfloor\}$. For a chosen $\h$, we apply the point insertion scheme at an NS boundary node $n$ if and only if the twist of the illegal half-node of $n$ is less than or equal to $2\h$.

\begin{ex}
    Consider the case where $r=9$, $I=\emptyset$ and $B$ consists of three markings of twist $5$ and one marking of twist $1$ and. In this case  $\Mbar^{1/r}_{0, B,I}$ is a disjoint union of six segments, each of them corresponds to a cyclic order of $B$. We write $B$ as $\{1,5_1,5_2,5_3\}$ to make the cyclic order manifest. We refer to the new marked points coming from point insertion by their twists with a hat. The procedure of point insertion (with $\h=3$) is shown in Figure \ref{fig:r=7 B=1,5,5,5}. Together with six additional segments, $\Mbar^{1/9}_{0,\bar B,I}$ is glued to obtain a space which is homeomorphic to a circle. 
    
    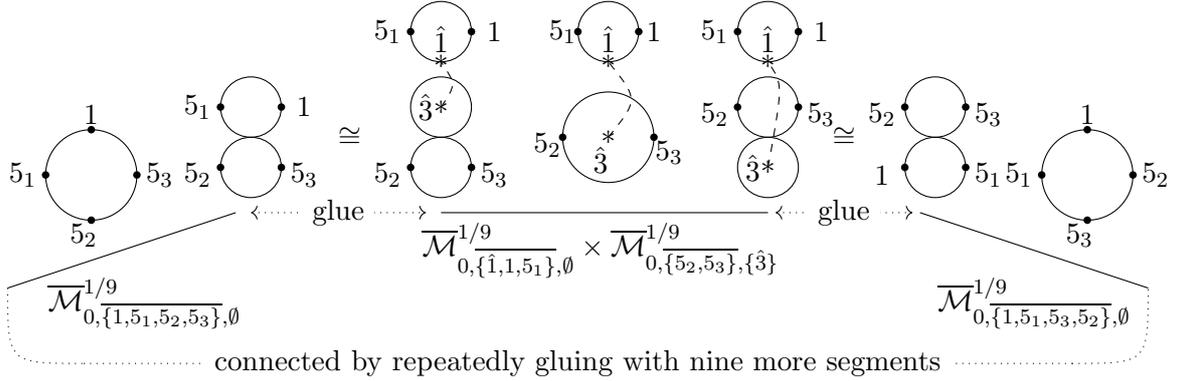
\begin{figure}[h]
        \centering
        \begin{tikzpicture}
            \draw (0,0)-- (3,1);
            \draw (1.7,1.5) arc (0:360:0.6);
            \node at (1.7,1.5)[circle,fill,inner sep=1pt]{};         \node at (2,1.5) {$5_3$};
            \node at (1.1,0.9)[circle,fill,inner sep=1pt]{};         \node at (1,0.7) {$5_2$};
            \node at (0.5,1.5)[circle,fill,inner sep=1pt]{};         \node at (0.2,1.5) {$5_1$};
            \node at (1.1,2.1)[circle,fill,inner sep=1pt]{};         \node at (1.1,2.3) {$1$};

                \draw (3.2,2) arc (90:450:0.4);
            \draw (3.2,2) arc (-90:270:0.4);
            \node at (2.8,1.6)[circle,fill,inner sep=1pt]{};   
                \node at (2.5,1.5) {$5_2$};
            \node at (2.8,2.4)[circle,fill,inner sep=1pt]{};  
                   \node at (2.5,2.4) {$5_1$};
            \node at (3.6,1.6)[circle,fill,inner sep=1pt]{};   
                 \node at (3.9,1.5) {$5_3$};
            \node at (3.6,2.4)[circle,fill,inner sep=1pt]{};  
                \node at (3.9,2.4) {$1$};

             \node at (4.5,2) {$\cong$};    

             \draw[<->,dotted] (3.2,1) to ["\text{glue}"]  (5.5,1) ;
                 
                 \draw (5.7,1)-- (10,1);

                 \draw (5.7,2) arc (90:450:0.4);
            \draw (5.7,2) arc (-90:270:0.4);
            \node at (5.3,1.6)[circle,fill,inner sep=1pt]{};   
                \node at (5.0,1.5) {$5_2$};
            \node at (5.3,3.4)[circle,fill,inner sep=1pt]{};  
                   \node at (5.0,3.4) {$5_1$};
            \node at (6.1,1.6)[circle,fill,inner sep=1pt]{};   
                 \node at (6.4,1.5) {$5_3$};
            \node at (6.1,3.4)[circle,fill,inner sep=1pt]{};  
                \node at (6.4,3.4) {$1$};
            \draw (5.7,3) arc (-90:270:0.4);

            \node at (5.7,2.4){$*$};
            \node at (5.7,3){$*$};
            \draw[dashed] (5.7,2.4) .. controls (5.9,2.7) .. (5.7,3);

            \node at (5.5,2.4) {$\hat 3$};
                  \node at (5.7,3.3) {$\hat 1$};

            \draw (10.0,2) arc (90:450:0.4);
            \draw (10.0,2) arc (-90:270:0.4);
            \node at (9.6,2.4)[circle,fill,inner sep=1pt]{};   
                \node at (9.3,2.3) {$5_2$};
            \node at (9.6,3.4)[circle,fill,inner sep=1pt]{};  
                   \node at (9.3,3.4) {$5_1$};
            \node at (10.4,2.4)[circle,fill,inner sep=1pt]{};   
                 \node at (10.7,2.3) {$5_3$};
            \node at (10.4,3.4)[circle,fill,inner sep=1pt]{};  
                \node at (10.7,3.4) {$1$};
            \draw (10.0,3) arc (-90:270:0.4);

            \node at (10.0,1.6){$*$};
            \node at (10.0,3){$*$};
            \draw[dashed] (10.0,1.6) .. controls (10.2,2.7) .. (10.0,3);

            \node at (9.8,1.6) {$\hat 3$};
                  \node at (10,3.3) {$\hat 1$};

            \draw (8.5,2) arc (0:360:0.6);
                \node at (8.5,2)[circle,fill,inner sep=1pt]{};   
                 \node at (8.7,1.8) {$5_3$};
                 \node at (7.3,2)[circle,fill,inner sep=1pt]{};   
                 \node at (7.1,1.9) {$5_2$};
                \node at (7.9,2) {$*$};

                \draw (7.9,3) arc (-90:270:0.4);
                \node at (7.5,3.4)[circle,fill,inner sep=1pt]{};  
                \node at (8.5,3.4) {$1$};
                \node at (8.3,3.4)[circle,fill,inner sep=1pt]{};  
                   \node at (7.3,3.4) {$5_1$};
                   \node at (7.9,3) {$*$};
                   \draw[dashed] (7.9,2) .. controls (8.3,2.5) .. (7.9,3);

                  \node at (7.8,1.7) {$\hat 3$};
                  \node at (7.9,3.3) {$\hat 1$};

                   \draw (12,1)-- (15,0);
                   \node at (11.0,2) {$\cong$};    

             \draw[<->,dotted] (10.1,1) to ["\text{glue}"]  (11.9,1) ;

                   \draw (12.2,2) arc (90:450:0.4);
            \draw (12.2,2) arc (-90:270:0.4);
            \node at (11.8,2.4)[circle,fill,inner sep=1pt]{};   
                \node at (11.5,2.3) {$5_2$};
            \node at (11.8,1.6)[circle,fill,inner sep=1pt]{};  
                   \node at (11.5,1.5) {$1$};
            \node at (12.6,2.4)[circle,fill,inner sep=1pt]{};   
                 \node at (12.9,2.3) {$5_3$};
            \node at (12.6,1.6)[circle,fill,inner sep=1pt]{};  
                \node at (12.9,1.5) {$5_1$};

            \draw (14.8,1.5) arc (0:360:0.6);
            \node at (14.8,1.5)[circle,fill,inner sep=1pt]{};         \node at (15.1,1.5) {$5_2$};
            \node at (14.2,0.9)[circle,fill,inner sep=1pt]{};         \node at (14.1,0.7) {$5_3$};
            \node at (13.6,1.5)[circle,fill,inner sep=1pt]{};         \node at (13.3,1.5) {$5_1$};
            \node at (14.2,2.1)[circle,fill,inner sep=1pt]{};         \node at (14.2,2.3) {$1$};

            \draw[dotted] (1,-1)  to ["\text{connected by repeatedly gluing with nine more segments}"] (14,-1);
            \draw[dotted] (0,0) .. controls (0,-1) .. (1,-1);
            \draw[dotted] ((14,-1) .. controls (15,-1) .. (15,0);

            \node at (1.8,-0.2) {$\Mbar^{1/9}_{0,\overline{\{1,5_1,5_2,5_3\}},\emptyset}$};

            \node at (13.5,-0.2) {$\Mbar^{1/9}_{0,\overline{\{1,5_1,5_3,5_2\}},\emptyset}$};

            \node at (7.8,0.5) {$\Mbar^{1/9}_{0,\overline{\{\hat 1,1,5_1\}},\emptyset}\times \Mbar^{1/9}_{0,\overline{\{5_2,5_3\}},\{\hat 3\}}$};
            
        \end{tikzpicture}
        \caption{An example of gluing $1$-dimensional moduli spaces by point insertion}
        \label{fig:r=7 B=1,5,5,5}

    \end{figure}
\end{ex}

\begin{ex}
    Consider the example where $r=2$, $I=\{0_{int}\}$ and $ B={\{0_1,0_2,0_3\}}$. As shown in Figure \ref{fig:point insertion for r=2 k=3 l=1}, the moduli $\Mbar^{1/2}_{0,\overline{\{0_1,0_2,0_3\}},0_{int}}$ is combinatorially equivalent to a hexagon. Following the point insertion procedure, another hexagon is glued to $\Mbar^{1/2}_{0,\overline{\{0_1,0_2,0_3\}},0_{int}}$ along three out of the six boundaries. Each of the other three boundaries is also connected to the other hexagon via two extra digons. All together we get a space homeomorphic to a sphere, which is glued from two hexagons and six digons. Note that $\Mbar^{1/2}_{0,\{0_1,0_2,0_3\},0_{int}}$ consists of two disconnected components, the other component $\Mbar^{1/2}_{0,\overline{\{0_1,0_3,0_2\}},0_{int}}$ is glued into another sphere, and we get two spheres in the end.
    
    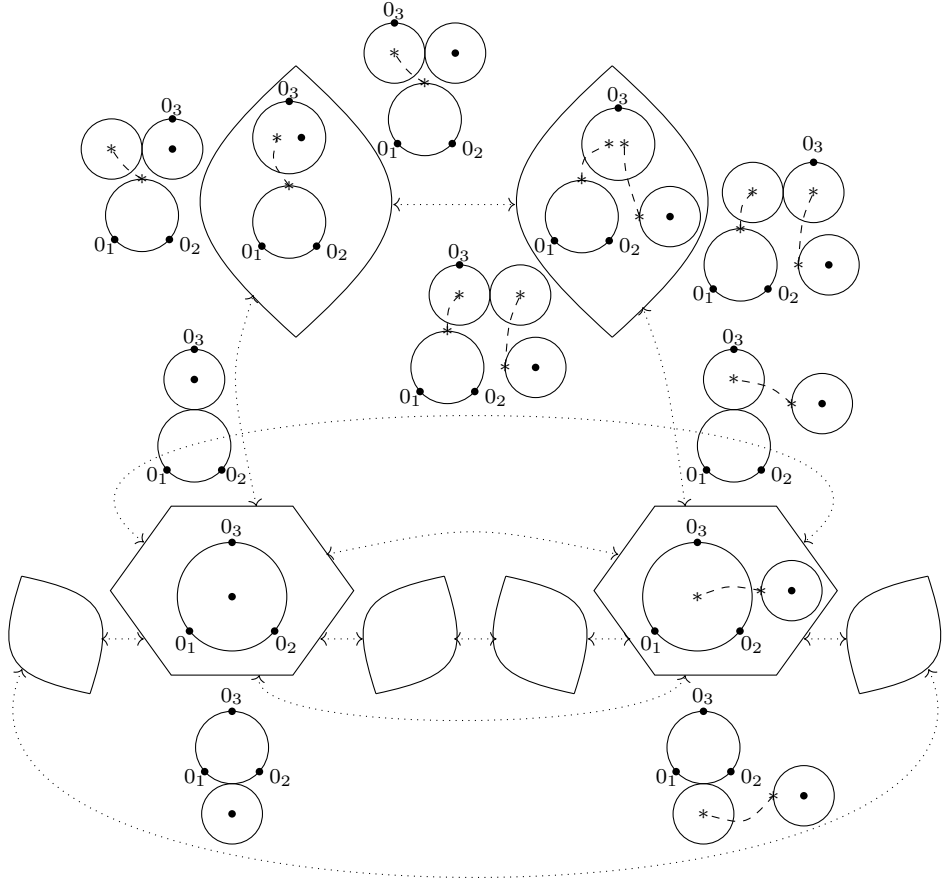
\begin{figure}[h]
         \centering
        \begin{tikzpicture}[scale=0.8,every node/.style={font=\fontsize{8}{10}\selectfont}]
    \draw (-1+1.05,0) -- (1+1.05,0) -- (2+1.05,1.4) -- (1+1.05,2.8) -- (-1+1.05,2.8) -- (-2+1.05,1.4) -- cycle;
    
    \draw (0.9+1.05,1.3) arc (0:360:0.9);
    \node at (0+1.05,1.3)[circle,fill,inner sep=1pt]{};
  
    \node at (-0.7+1.05,0.73)[circle,fill,inner sep=1pt]{};
    \node at (-0.75+1.05,0.5) {$0_{1}$};
    \node at (0.7+1.05,0.73)[circle,fill,inner sep=1pt]{};
    \node at (0.9+1.05,0.5) {$0_{2}$};
    \node at (0+1.05,2.2)[circle,fill,inner sep=1pt]{};
    \node at (0+1.05,2.45) {$0_{3}$};

    \draw (0.6+1.05,-1.2) arc (0:360:0.6);
    \node at (-0.45+1.05,-1.6)[circle,fill,inner sep=1pt]{};
    \node at (-0.6+1.05,-1.7) {$0_{1}$};
    \node at (0.45+1.05,-1.6)[circle,fill,inner sep=1pt]{};
    \node at (0.8+1.05,-1.7) {$0_{2}$};
    \node at (0+1.05,-0.6)[circle,fill,inner sep=1pt]{};
    \node at (0+1.05,-0.35) {$0_{3}$};
    
    \draw (0.5+1.05,-2.3) arc (0:360:0.5);
    \node at (0+1.05,-2.3)[circle,fill,inner sep=1pt]{};

    \draw (8,0) -- (10,0) -- (11,1.4) -- (10,2.8) -- (8,2.8) -- (7,1.4) -- cycle;
    
    \draw (9.6,1.3) arc (0:360:0.9);
    \node at (8.7,1.3){$*$};
  
    \node at (8,0.73)[circle,fill,inner sep=1pt]{};
    \node at (7.95,0.5) {$0_{1}$};
    \node at (9.4,0.73)[circle,fill,inner sep=1pt]{};
    \node at (9.6,0.5) {$0_{2}$};
    \node at (8.7,2.2)[circle,fill,inner sep=1pt]{};
    \node at (8.7,2.45) {$0_{3}$};

    \draw (10.75,1.4) arc (0:360:0.5);
    \node at (10.25,1.4)[circle,fill,inner sep=1pt]{};
  
    \node at (9.75,1.4) {$*$};

    \draw[dashed] (9.75,1.4) .. controls (9.2,1.5) .. (8.7,1.3);

    \draw (9.4,-1.2) arc (0:360:0.6);
    \node at (8.35,-1.6)[circle,fill,inner sep=1pt]{};
    \node at (8.2,-1.7) {$0_{1}$};
    \node at (9.25,-1.6)[circle,fill,inner sep=1pt]{};
    \node at (9.6,-1.7) {$0_{2}$};
    \node at (8.8,-0.6)[circle,fill,inner sep=1pt]{};
    \node at (8.8,-0.35) {$0_{3}$};
    
    \draw (9.3,-2.3) arc (0:360:0.5);
    \node at (8.8,-2.3){$*$};

    \draw (10.95,-2) arc (0:360:0.5);
    \node at (10.45,-2)[circle,fill,inner sep=1pt]{};
  
    \node at (9.95,-2) {$*$};
    \draw[dashed] (9.95,-2) .. controls (9.5,-2.5) .. (8.8,-2.3);

    \draw[dotted,<->] (1.5,0) .. controls (1.25,-1) and (8.25,-1) .. (8.5,0);

    \draw (9.9,3.8) arc (0:360:0.6);
    \node at (8.85,3.4)[circle,fill,inner sep=1pt]{};
    \node at (8.7,3.3) {$0_{1}$};
    \node at (9.75,3.4)[circle,fill,inner sep=1pt]{};
    \node at (10.1,3.3) {$0_{2}$};

    \draw (9.8,4.9) arc (0:360:0.5);
    \node at (9.3,4.9){$*$};
    \node at (9.3,5.4)[circle,fill,inner sep=1pt]{};
    \node at (9.3,5.6) {$0_{3}$};

    \draw (11.25,4.5) arc (0:360:0.5);
    \node at (10.75,4.5)[circle,fill,inner sep=1pt]{};
  
    \node at (10.25,4.5) {$*$};
    \draw[dashed] (10.25,4.5) .. controls (9.9,4.8) .. (9.3,4.9);

    \draw (9.9+0.11, 6.8) arc (0:360:0.6);
    \node at (8.85+0.11, 6.4)[circle,fill,inner sep=1pt]{};
    \node at (8.7+0.11, 6.3) {$0_{1}$};
    \node at (9.75+0.11, 6.4)[circle,fill,inner sep=1pt]{};
    \node at (10.1+0.11, 6.3) {$0_{2}$};
    
    \node at (9.3+0.11, 7.4) {$*$};
    \draw[dashed] (9.3+0.11, 7.4) .. controls (9.3+0.11, 7.8) .. (9.5+0.11, 8);
    
    \draw (10+0.11, 8) arc (0:360:0.5);
    \node at (9.5+0.11, 8){$*$};

    \draw (11.25+0.11, 6.8) arc (0:360:0.5);
    \node at (10.75+0.11, 6.8)[circle,fill,inner sep=1pt]{};
  
    \node at (10.25+0.11, 6.8) {$*$};
    \draw[dashed] (10.25+0.11, 6.8) .. controls (10.3+0.11, 7.55) .. (10.5+0.11, 8);

    \draw (11+0.11, 8) arc (0:360:0.5);
    \node at (10.5+0.11, 8){$*$};
    
    \node at (10.5+0.11, 8.5)[circle,fill,inner sep=1pt]{};
    \node at (10.5+0.11, 8.8) {$0_{3}$};

   \draw (4.4+0.2+0.59, 5.6-0.5) arc (0:360:0.6);
    \node at (3.35+0.2+0.59, 5.2-0.5)[circle,fill,inner sep=1pt]{};
    \node at (3.2+0.2+0.59, 5.1-0.5) {$0_{1}$};
    \node at (4.25+0.2+0.59, 5.2-0.5)[circle,fill,inner sep=1pt]{};
    \node at (4.6+0.2+0.59, 5.1-0.5) {$0_{2}$};

    \node at (3.8+0.2+0.59, 6.2-0.5) {$*$};
    \draw[dashed] (3.8+0.2+0.59, 6.2-0.5) .. controls (3.8+0.2+0.59, 6.6-0.5) .. (4+0.2+0.59, 6.8-0.5);
    
    \draw (4.5+0.2+0.59, 6.8-0.5) arc (0:360:0.5);
    \node at (4+0.2+0.59, 6.8-0.5){$*$};
    \node at (4+0.2+0.59, 7.3-0.5)[circle,fill,inner sep=1pt]{};
    \node at (4+0.2+0.59, 7.5-0.5) {$0_{3}$};

    \draw (5.75+0.2+0.59, 5.6-0.5) arc (0:360:0.5);
    \node at (5.25+0.2+0.59, 5.6-0.5)[circle,fill,inner sep=1pt]{};
  
    \node at (4.75+0.2+0.59, 5.6-0.5) {$*$};
    \draw[dashed] (4.75+0.2+0.59, 5.6-0.5) .. controls (4.8+0.2+0.59, 6.35-0.5) .. (5+0.2+0.59, 6.8-0.5);
    
    \draw (5.5+0.2+0.59, 6.8-0.5) arc (0:360:0.5);
    \node at (5+0.2+0.59, 6.8-0.5){$*$};

    \draw (7.4,7.6) arc (0:360:0.6);
\node at (6.35,7.2)[circle,fill,inner sep=1pt]{};
\node at (6.2,7.1) {$0_{1}$};
\node at (7.25,7.2)[circle,fill,inner sep=1pt]{};
\node at (7.6,7.1) {$0_{2}$};

\node at (6.8,8.2) {$*$};
\draw[dashed] (6.8,8.2) .. controls (6.8,8.6) .. (7.25,8.8);

\node at (7.25,8.8){$*$};
\node at (7.4,9.4)[circle,fill,inner sep=1pt]{};
\node at (7.4,9.6) {$0_{3}$};

\draw (8.75,7.6) arc (0:360:0.5);
\node at (8.25,7.6)[circle,fill,inner sep=1pt]{};

\node at (7.75,7.6) {$*$};
\draw[dashed] (7.75,7.6) .. controls (7.5,8.35) .. (7.5,8.8);

\draw (8,8.8) arc (0:360:0.6);
\node at (7.5,8.8){$*$};

\draw (7.3,10.1) .. controls (9.4,8.2) and (9.4,7.5) .. (7.3,5.6);
\draw (7.3,10.1) .. controls (5.2,8.2) and (5.2,7.5) .. (7.3,5.6);

\draw (2.1,10.1) .. controls (4.2,8.2) and (4.2,7.5) .. (2.1,5.6);
\draw (2.1,10.1) .. controls (0.0,8.2) and (0.0,7.5) .. (2.1,5.6);

  \draw (4.4+0.2-2.01, 5.6+1.91) arc (0:360:0.6);
    \node at (3.35+0.2-2.01, 5.2+1.91)[circle,fill,inner sep=1pt]{};
    \node at (3.2+0.2-2.01, 5.1+1.91) {$0_{1}$};
    \node at (4.25+0.2-2.01, 5.2+1.91)[circle,fill,inner sep=1pt]{};
    \node at (4.6+0.2-2.01, 5.1+1.91) {$0_{2}$};

    \node at (3.8+0.2-2.01, 6.2+1.91) {$*$};
    \draw[dashed] (3.8+0.2-2.01, 6.2+1.91) .. controls (3.5+0.2-2.01, 6.6+1.91) .. (3.8-2.01, 7+1.91);
    
    \draw (4.6-2.01, 7+1.91) arc (0:360:0.6);
    \node at (3.8-2.01, 7+1.91){$*$};
    \node at (4-2.01, 7.6+1.91)[circle,fill,inner sep=1pt]{};
    \node at (4-2.01, 7.8+1.91) {$0_{3}$};

    \node at (4.2-2.01, 7+1.91)[circle,fill,inner sep=1pt]{};

        \draw (4.4+0.2+0.22, 5.6+3.61) arc (0:360:0.6);
    \node at (3.35+0.2+0.22, 5.2+3.61)[circle,fill,inner sep=1pt]{};
    \node at (3.2+0.2+0.22, 5.1+3.61) {$0_{1}$};
    \node at (4.25+0.2+0.22, 5.2+3.61)[circle,fill,inner sep=1pt]{};
    \node at (4.6+0.2+0.22, 5.1+3.61) {$0_{2}$};

    \node at (3.8+0.2+0.22, 6.2+3.61) {$*$};
    \draw[dashed] (3.8+0.2+0.22, 6.2+3.61) .. controls (3.5+0.2+0.22, 6.4+3.61) .. (3.5+0.22, 6.7+3.61);
    
    \draw (4.0+0.22, 6.7+3.61) arc (0:360:0.5);
    \node at (3.5+0.22, 6.7+3.61){$*$};
    \node at (3.5+0.22, 7.2+3.61)[circle,fill,inner sep=1pt]{};
    \node at (3.5+0.22, 7.4+3.61) {$0_{3}$};

    \node at (4.5+0.22, 6.7+3.61)[circle,fill,inner sep=1pt]{};

    \draw (5.0+0.22, 6.7+3.61) arc (0:360:0.5);

    \draw (0 +1.03, 3.8) arc (0:360:0.6);
\node at (-1.05 +1.03,3.4)[circle,fill,inner sep=1pt]{};
\node at (-1.2 +1.03,3.3) {$0_{1}$};
\node at (-0.15 +1.03,3.4)[circle,fill,inner sep=1pt]{};
\node at (0.1 +1.03,3.3) {$0_{2}$};

\draw (-0.1 +1.03,4.9) arc (0:360:0.5);
\node at (-0.6 +1.03,4.9)[circle,fill,inner sep=1pt]{};
\node at (-0.6 +1.03,5.4)[circle,fill,inner sep=1pt]{};
\node at (-0.6 +1.03,5.6) {$0_{3}$};

   \draw (4.4+0.2-4.43, 5.6+2.02) arc (0:360:0.6);
    \node at (3.35+0.2-4.43, 5.2+2.02)[circle,fill,inner sep=1pt]{};
    \node at (3.2+0.2-4.43, 5.1+2.02) {$0_{1}$};
    \node at (4.25+0.2-4.43, 5.2+2.02)[circle,fill,inner sep=1pt]{};
    \node at (4.6+0.2-4.43, 5.1+2.02) {$0_{2}$};

    \node at (3.8+0.2-4.43, 6.2+2.02) {$*$};
    \draw[dashed] (3.8+0.2-4.43, 6.2+2.02) .. controls (3.5+0.2-4.43, 6.4+2.02) .. (3.5-4.43, 6.7+2.02);
    
    \draw (4.0-4.43, 6.7+2.02) arc (0:360:0.5);
    \node at (3.5-4.43, 6.7+2.02){$*$};
    \node at (4.5-4.43, 7.2+2.02)[circle,fill,inner sep=1pt]{};
    \node at (4.5-4.43, 7.4+2.02) {$0_{3}$};

    \node at (4.5-4.43, 6.7+2.02)[circle,fill,inner sep=1pt]{};

    \draw (5.0-4.43, 6.7+2.02) arc (0:360:0.5);

    \draw[dotted, <->] (0.7+0.75,2.8) .. controls (1,5) .. (1.4,6.3);

    \draw[dotted, <->] (-0.7+9.2,2.8) .. controls (-1+9.2,5) .. (-1.4+9.2,6.1);

    \draw[dotted, <->] (3.7,7.8) -- (5.7,7.8);

    \draw (3.65+1.05-7.11, 5.05-3.41) .. controls ({3.65 + 0.866*(4.7-3.65) - 0.500*(4.1-5.05)+1.05-7.11}, {5.05 + 0.500*(4.7-3.65) + 0.866*(4.1-5.05)-3.41}) and ({3.65 + 0.866*(4.7-3.65) - 0.500*(3.75-5.05)+1.05-7.11}, {5.05 + 0.500*(4.7-3.65) + 0.866*(3.75-5.05)-3.41}) .. ({3.65 + 0.866*(3.65-3.65) - 0.500*(2.8-5.05)+1.05-7.11}, {5.05 + 0.500*(3.65-3.65) + 0.866*(2.8-5.05)-3.41});
    \draw (3.65+1.05-7.11, 5.05-3.41) .. controls ({3.65 + 0.866*(2.6-3.65) - 0.500*(4.1-5.05)+1.05-7.11}, {5.05 + 0.500*(2.6-3.65) + 0.866*(4.1-5.05)-3.41}) and ({3.65 + 0.866*(2.6-3.65) - 0.500*(3.75-5.05)+1.05-7.11}, {5.05 + 0.500*(2.6-3.65) + 0.866*(3.75-5.05)-3.41}) .. ({3.65 + 0.866*(3.65-3.65) - 0.500*(2.8-5.05)+1.05-7.11}, {5.05 + 0.500*(3.65-3.65) + 0.866*(2.8-5.05)-3.41});

    \draw (3.65+8.84, 5.05-3.41) .. controls ({3.65 + 0.866*(4.7-3.65) + 0.500*(4.1-5.05)+8.84}, {5.05 - 0.500*(4.7-3.65) + 0.866*(4.1-5.05)-3.41}) and ({3.65 + 0.866*(4.7-3.65) + 0.500*(3.75-5.05)+8.84}, {5.05 - 0.500*(4.7-3.65) + 0.866*(3.75-5.05)-3.41}) .. ({3.65 + 0.866*(3.65-3.65) + 0.500*(2.8-5.05)+8.84}, {5.05 - 0.500*(3.65-3.65) + 0.866*(2.8-5.05)-3.41});
    \draw (3.65+8.84, 5.05-3.41) .. controls ({3.65 + 0.866*(2.6-3.65) + 0.500*(4.1-5.05)+8.84}, {5.05 - 0.500*(2.6-3.65) + 0.866*(4.1-5.05)-3.41}) and ({3.65 + 0.866*(2.6-3.65) + 0.500*(3.75-5.05)+8.84}, {5.05 - 0.500*(2.6-3.65) + 0.866*(3.75-5.05)-3.41}) .. ({3.65 + 0.866*(3.65-3.65) + 0.500*(2.8-5.05)+8.84}, {5.05 - 0.500*(3.65-3.65) + 0.866*(2.8-5.05)-3.41});

     \draw (3.65+1.90, 5.05-3.41) .. controls ({3.65 + 0.866*(4.7-3.65) - 0.500*(4.1-5.05)+1.90}, {5.05 + 0.500*(4.7-3.65) + 0.866*(4.1-5.05)-3.41}) and ({3.65 + 0.866*(4.7-3.65) - 0.500*(3.75-5.05)+1.90}, {5.05 + 0.500*(4.7-3.65) + 0.866*(3.75-5.05)-3.41}) .. ({3.65 + 0.866*(3.65-3.65) - 0.500*(2.8-5.05)+1.90}, {5.05 + 0.500*(3.65-3.65) + 0.866*(2.8-5.05)-3.41});
    \draw (3.65+1.90, 5.05-3.41) .. controls ({3.65 + 0.866*(2.6-3.65) - 0.500*(4.1-5.05)+1.90}, {5.05 + 0.500*(2.6-3.65) + 0.866*(4.1-5.05)-3.41}) and ({3.65 + 0.866*(2.6-3.65) - 0.500*(3.75-5.05)+1.90}, {5.05 + 0.500*(2.6-3.65) + 0.866*(3.75-5.05)-3.41}) .. ({3.65 + 0.866*(3.65-3.65) - 0.500*(2.8-5.05)+1.90}, {5.05 + 0.500*(3.65-3.65) + 0.866*(2.8-5.05)-3.41});

    \draw (3.65+1.05-0.16, 5.05-3.41) .. controls ({3.65 + 0.866*(4.7-3.65) + 0.500*(4.1-5.05)+1.05-0.16}, {5.05 - 0.500*(4.7-3.65) + 0.866*(4.1-5.05)-3.41}) and ({3.65 + 0.866*(4.7-3.65) + 0.500*(3.75-5.05)+1.05-0.16}, {5.05 - 0.500*(4.7-3.65) + 0.866*(3.75-5.05)-3.41}) .. ({3.65 + 0.866*(3.65-3.65) + 0.500*(2.8-5.05)+1.05-0.16}, {5.05 - 0.500*(3.65-3.65) + 0.866*(2.8-5.05)-3.41});
    \draw (3.65+1.05-0.16, 5.05-3.41) .. controls ({3.65 + 0.866*(2.6-3.65) + 0.500*(4.1-5.05)+1.05-0.16}, {5.05 - 0.500*(2.6-3.65) + 0.866*(4.1-5.05)-3.41}) and ({3.65 + 0.866*(2.6-3.65) + 0.500*(3.75-5.05)+1.05-0.16}, {5.05 - 0.500*(2.6-3.65) + 0.866*(3.75-5.05)-3.41}) .. ({3.65 + 0.866*(3.65-3.65) + 0.500*(2.8-5.05)+1.05-0.16}, {5.05 - 0.500*(3.65-3.65) + 0.866*(2.8-5.05)-3.41});

    \draw[dotted, <->] (2.5,0.6) -- (3.2,0.6);

    \draw[dotted, <->] (4.7,0.6) -- (5.4,0.6);

    \draw[dotted, <->] (6.9,0.6) -- (7.6,0.6);

     \draw[dotted, <->] (10.45,0.6) -- (11.15,0.6);

    \draw[dotted, <->] (-1.1,0.6) -- (-0.4,0.6);

    \draw[dotted, <->] (2.6,2) .. controls (5,2.5) ..  (7.4,2);

    \draw[dotted, <->] (-2.4,0.1) .. controls (-4.4,-4.5) and (14.5,-4.5) .. (12.5,0.1);

    \draw[dotted, <->] (-0.4, 2.2) .. controls (-4,5) and (14.05,5) .. (10.45,2.2);

\end{tikzpicture} 
        \caption{An example of gluing $2$-dimensional moduli spaces by point insertion. We omit some twists of markings because they are all equal to $0$.}
        \label{fig:point insertion for r=2 k=3 l=1}
    \end{figure}

\end{ex}

\subsection{Reduced and unreduced $(r,\h)$-surfaces}

The moduli spaces of disconnected $r$-spin surfaces together with dashed lines encoding the point insertion procedure, as the ones illustrated in Figures \ref{fig:r=7 B=1,5,5,5} and \ref{fig:point insertion for r=2 k=3 l=1}, play a crucial role in our construction.
\begin{dfn}\label{dfn rh graph}
    an $(r,\h)$-surface is a collection of legal connected level-$\h$ stable graded $r$-spin surfaces (the components), together with
    \begin{enumerate}
        \item a bijection (denoted by dashed lines) between a subset $B^{PI}$ of the boundary tails and  a subset $I^{PI}$ of the internal tails  which are not contracted boundary nodes,  where the twist $a$ and $b$ for paired internal tail and boundary tail satisfies $a+2b=r-2$ and $0\le a\le \h$;
        \item markings on the set of unpaired internal tails $I^{up}$ (which are not contracted boundary nodes) and boundary tails $B^{up}$, \textit{i.e.} identifications $I^{up}=\{1,2,\dots,\lvert I^{up}\rvert\}$ and $B^{up}=\{1,2,\dots,\lvert B^{up}\rvert\}$.
    \end{enumerate}
    We require that the union of all stable graded $r$-spin surfaces in the collection are connected via the dashed lines. We also require that, in the collection, there exists no genus-zero stable graded $r$-spin surfaces with only two tails, where both of these two tails are in $I^{PI}\sqcup B^{PI}$.
\end{dfn}

\begin{dfn}
    An \textit{after-insertion (AI) node} of an $(r,\h)$-surface is an NS boundary node whose twist on the illegal half-node is greater than $2\h$, and the irreducible component containing the legal half-node only contains this half-node and an internal tail in $I^{PI}$.

    A \textit{before-insertion (BI) node} of an $(r,\h)$-surface is an NS boundary node whose twist on the illegal half-node is greater than $2\h$, but is not an after-insertion node.
\end{dfn}

\begin{dfn}
    We say two $(r,\h)$-surfaces $S_C$ and $S_D$ are \textit{related by point insertion} if 
    \begin{itemize}
        \item  $S_C$ has a component $C_1^{BI}$ with a before-insertion node $n$, by normalizing this node we obtain a nonlegal connected level-$\h$ stable graded $r$-spin surfaces $\hat D_1$ and a legal connected level-$\h$ stable graded $r$-spin surfaces $D_2$; we denote by $h_1\in \hat D_1$ and $h_2\in D_2$ the illegal and legal half-nodes corresponding to $n$;
        \item $S_D$ has a component isomorphic to $D_2$ and a component $D_1^{AI}$ with an after-insertion node $n'$ such that, after normalizing $n'$,  the irreducible component containing the illegal half-node is isomorphic to $\hat D_1$, and the internal tail on the irreducible component containing the legal half-node is paired with $h_2\in D_2$;
        \item all components of $S_C$ and $S_D,$ but $C_1^{BI}$, $D_1^{AI}$ and $D_2,$ are one-to-one identical.
    \end{itemize}
\end{dfn}

\begin{dfn}\label{def reduced rh surface}
    We define an equivalence relation $\sim_{PI}$ for $(r,\h)$-surfaces:
     two $(r,\h)$-surfaces $S_1$ and $S_2$ are equivalent under $\sim_{PI}$ if there exists a chain of $(r,\h)$-surfaces $T_1,T_2,\dots,T_i$ such that $T_1=S_1$, $T_i=S_2$, $T_j$, and $T_{j+1}$ are related by point insertion for all $1\le j\le i-1$. 
     
    A \textit{reduced $(r,\h)$-surface} is an equivalence class of $\sim_{PI}$.
\end{dfn}
 
 \begin{figure}[h]
         \centering

         \begin{subfigure}{.45\textwidth}
  \centering
\begin{tikzpicture}[scale=0.45]
\draw (0,-0.5) circle (1);
\draw (-1,-0.5) arc (180:360:1 and 0.333);
\draw[dashed](1,-0.5) arc (0:180:1 and 0.333);

\draw (-1.5,-3) arc (180:360:1.5 and 0.5);
\draw[dashed](1.5,-3) arc (0:180:1.5 and 0.5);
\draw(1.5,-3) arc (0:180:1.5);

\draw (1.5,-3) arc (180:360:1 and 0.333);
\draw[dashed](3.5,-3) arc (0:180:1 and 0.333);
\draw(3.5,-3) arc (0:180:1);

\node at (1.5,-3) [circle,fill,inner sep=1pt]{};

\draw (4.5,-3) arc (180:360:2 and 0.667);
\draw[dashed](8.5,-3) arc (0:180:2 and 0.667);
\draw(8.5,-3) arc (0:180:2);

\node at (3.5,-3) [circle,fill,inner sep=1pt]{};
\node at (6.5,-2){*};
\draw[dashed] (3.5,-3) .. controls (5.5,-1.5) ..(6.5,-2);

\node at (1.5,-4.3){$C_1^{BI}$};

\end{tikzpicture} 
\end{subfigure}
\begin{subfigure}{.45\textwidth}
  \centering
\begin{tikzpicture}[scale=0.45]
\draw (0,-0.5) circle (1);
\draw (-1,-0.5) arc (180:360:1 and 0.333);
\draw[dashed](1,-0.5) arc (0:180:1 and 0.333);

\draw (-1.5,-3) arc (180:360:1.5 and 0.5);
\draw[dashed](1.5,-3) arc (0:180:1.5 and 0.5);
\draw(1.5,-3) arc (0:180:1.5);

\draw (1.5,-3) arc (180:360:0.5 and 0.167);
\draw[dashed](2.5,-3) arc (0:180:0.5 and 0.167);
\draw(2.5,-3) arc (0:180:0.5);

\fill[color = gray, opacity = 0.5] (1.5,-3) arc (180:360:0.5 and 0.167) arc (0:180:0.5);

\draw (4,-3) arc (180:360:1 and 0.333);
\draw[dashed](6,-3) arc (0:180:1 and 0.333);
\draw(6,-3) arc (0:180:1);

\draw (7,-3) arc (180:360:2 and 0.667);
\draw[dashed](11,-3) arc (0:180:2 and 0.667);
\draw(11,-3) arc (0:180:2);

\node at (6,-3) [circle,fill,inner sep=1pt]{};
\node at (9,-2){*};
\draw[dashed] (6,-3) .. controls (8,-1.5) ..(9,-2);

\node at (4,-3) [circle,fill,inner sep=1pt]{};

\node at (2,-2.8){*};

\draw[dashed] (4,-3) .. controls (3,-2) ..(2,-2.8);

\node at (1,-4.3){$ D_1^{AI}$};
\node at (5,-4.3){$ D_2$};

\end{tikzpicture} 
\end{subfigure}

        \caption{An example of two $(r,\h)$-surfaces related by point insertion, where we shaded the irreducible component which only contains a legal half-node (corresponding to an after-insertion node) and an internal tail in $I^{PI}$. }
        \label{fig rh surface}
    \end{figure}
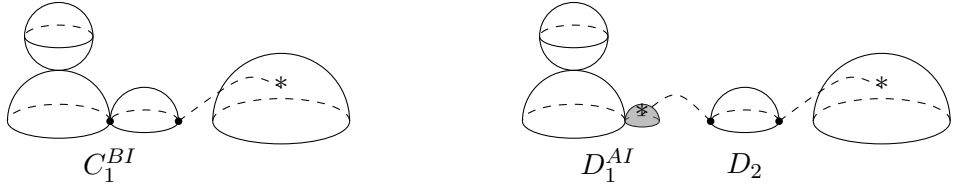

Note that $\sim_{PI}$ does not lift to the universal curve, but, as we saw, it does lift to the associated Witten and relative cotangent bundles.
\subsection{$(r,\h)$-graphs}

 We define the combinatorial objects associated with $(r,\h)$-surfaces.

\begin{dfn}
For $0\le \h \le \lfloor\frac{r}{2}\rfloor-1$, a \textit{$(r,\h)$-graph}  $\mathbf{G}$ consists of 
\begin{itemize} 
    \item a set $V(\mathbf{G})$ of connected legal level-$\h$ stable graded $r$-spin graphs with at least one open vertex or contracted boundary tail;
    \item two partitions of sets
    $$
    \bigsqcup_{\Gamma\in V(\mathbf{G})}\left(T^I(\Gamma)\backslash H^{CB}(\Gamma)\right)=I(\mathbf{G})\sqcup I^{PI}(\mathbf{G})
    $$
    and
    $$
    \bigsqcup_{\Gamma\in V(\mathbf{G})}T^B(\Gamma)=B(\mathbf{G})\sqcup B^{PI}(\mathbf{G});
    $$
    \item a set of edges (the dashed lines) $$E(\mathbf{G})\subseteq \{(a,b)\colon a\in I^{PI}(\mathbf{G}),b\in B^{PI}(\mathbf{G}), 2a+b=r-2\}$$ which induces an one-to-one correspondence $\delta$ between $I^{PI}(\mathbf{G})$ and $B^{PI}(\mathbf{G})$;
    \item a labelling of the set $I(\mathbf{G})$ by $\{1,2,\dots,l(\mathbf{G}):=\lvert I(\mathbf{G})\rvert\}$ and a labelling of the set $B(\mathbf{G})$ by $\{1,2,\dots,k(\mathbf{G}):=\lvert B(\mathbf{G})\rvert\}$.
\end{itemize}
    We require that 
    \begin{enumerate}
        \item there exists no $\Gamma\in V(\mathbf{G}),~g(\Gamma)=0$ satisfying 
    $
    H^I(\Gamma)\subseteq I^{PI}(\mathbf{G}), H^B(\Gamma)\subseteq B^{PI}(\mathbf{G})$ 
    and $
         \lvert H^I(\Gamma)\rvert+ \lvert H^B(\Gamma\rvert \le 2;
    $
    \end{enumerate}
    We define an auxiliary graph (in the normal sense) $\hat{\mathbf{G}}$ in the following way: the set of vertices of $\hat{\mathbf{G}}$ is $V(\mathbf{G})$, the set of edges of $\hat{\mathbf{G}}$ is $E(\mathbf{G})$; an element $(a,b)\in E(\mathbf{G})$ corresponds to an edge between the vertices $\Gamma_a$ and $\Gamma_b$, where  $a\in T^I(\Gamma_a)$ and $b\in T^B(\Gamma_b)$. We also require that
\begin{enumerate}[resume*]
    \item  the graph $\hat{\mathbf{G}}$ is a connected graph.
\end{enumerate}

\end{dfn}

 We define the genus of $\mathbf{G}$ to be
$$
g(\mathbf{G}):=\sum_{\Gamma\in V(\mathbf{G})} g(\Gamma) +g(\hat{\mathbf{G}});
$$
where $g(\hat{\mathbf{G}})$ is the genus of graph in the normal sense.

\begin{figure}[h]
  \centering

\begin{tikzpicture}[scale=1, every text node part/.style={align=center},]
\vspace{0.15cm}
  \node[rectangle, draw, very thick, minimum size=1]    (open1)    {$\hat g=0$\\$n=1$};
  \node[above = 0.5cm of open1, circle, draw, very thick, minimum size=1]    (closed)    {$\hat g=0$};
  \node[right = 0.5cm of open1, rectangle, draw, very thick, minimum size=1]    (opensmall)    {$\hat g=0$\\$n=1$};
  \draw [double] (open1) to (closed);
  \draw  (open1) to (opensmall);
  \node[right = 0.3cm of opensmall]    (internal1)    {};
  \draw [double] (opensmall) to (internal1);
  \node[right = 1cm of internal1]    (boundary1)    {};
  \node[left = 0cm of internal1]    (internal1p)    {};
  \node[left = 0cm of boundary1]    (boundary1p)    {};
  \draw [dashed] (boundary1) to (internal1p);

  \node[right = 0.3cm of boundary1p, rectangle, draw, very thick, minimum size=1]    (open2)    {$\hat g=0$\\$n=1$};
  \draw (boundary1p) to (open2);

    \node[right = 0.3cm of open2]    (boundary2)    {};
    \draw (boundary2) to (open2);

    \node[left = 0cm of boundary2]    (boundary2p)    {};
    \node[right = 1cm of boundary2]    (internal2)    {};
    \draw [dashed] (boundary2p) to (internal2);
    \node[left = 0cm of internal2]    (internal2p)    {};
    \node[right = 0.3cm of internal2, rectangle, draw, very thick, minimum size=1]    (open3)    {$\hat g=0$\\$n=1$};
     \draw [double] (internal2p) to (open3);

\end{tikzpicture}
\vspace{0.15cm}

  \caption{The $(r,\h)$-graph corresponding to the $(r,\h)$-surface on the right in Figure \ref{fig rh surface}.}
\end{figure}
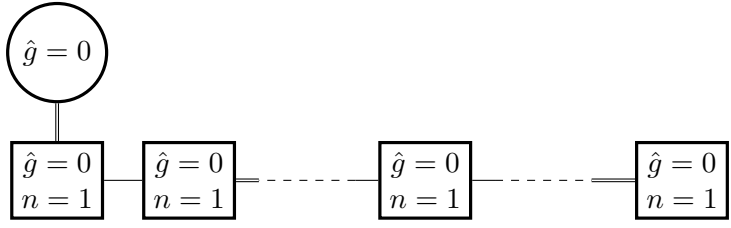

\begin{dfn}
    An isomorphism between two $(r,\h)$-graphs $\mathbf{G_1}$ and $\mathbf{G_2}$ consists of a collection of isomorphism of stable graded $r$-spin graphs between elements of $V(\mathbf{G_1})$ and $V(\mathbf{G_2})$, which induces a bijection between $V(\mathbf{G_1})$ and $V(\mathbf{G_2})$, and preserves the partitions, dashed lines, and labellings.
\end{dfn}

\begin{dfn}
Let $\mathbf{G}$ be an $(r,\h)$-graph. Let $e$ be an edge or a contracted boundary tail of some $\Gamma\in V(\mathbf{G})$. Since $T^I(\Gamma)\backslash H^{CB}(\Gamma)=T^I(d_e \Gamma)\backslash H^{CB}(d_e \Gamma)$ and $T^B(\Gamma)=T^B(d_e \Gamma)$, we define the \textit{smoothing} of $\mathbf{G}$ along $e$ to be the $(r,\h)$-graph $d_e \mathbf{G}$ obtained by replacing $\Gamma$ with $d_e \Gamma$.

We say $\mathbf{G}$ is smooth if all $\Gamma\in V(\mathbf{G})$ are smooth stable graded $r$-spin graphs.
We denote by $\GPI^{r,\h}_g$ the set of all genus-g $(r,\h)$-graphs.
\end{dfn}

\subsection{The moduli space of (non reduced) $(r,\h)$-surfaces}

For each $\mathbf{G}\in \GPI^{r,\h}_g$, let $\operatorname{Aut} \mathbf{G}$ be the group of automorphisms of $\mathbf{G}\in \GPI^{r,\h}_g$, then there is a natural action of $\operatorname{Aut} \mathbf{G}$ over the product $\prod_{\Gamma\in V(\mathbf{G})} \Mbar_\Gamma$. We define
$$
\Mbar_\mathbf{G}:=\left(\prod_{\Gamma\in V(\mathbf{G})} \Mbar_\Gamma\right)\bigg\slash \operatorname{Aut} \mathbf{G}
$$
and, when $\bm{G}$ is genus-zero or genus-one, 
$$
\oQMb_\mathbf{G}:=\left(\prod_{\Gamma\in V(\mathbf{G})} \oQMb_\Gamma\right)\bigg\slash \operatorname{Aut} \mathbf{G}.
$$

In the genus-zero or genus-one case,
denoting by $\mathcal W_\Gamma$ the Witten bundle over $\oQMb_\Gamma$, we define the Witten bundle $\mathcal W_\mathbf{G}$ over $\oQMb_\mathbf{G}$ to be 
$$
 \mathcal W_\mathbf{G}:= \left(\bboxplus_{\Gamma\in V(\mathbf{G})}\mathcal W_\Gamma\right)\bigg\slash \operatorname{Aut} \mathbf{G}.
$$
\begin{rmk}
Note that for a genus-zero or genus-one $(r,\h)$-graph $\mathbf{G}$, the automorphism group $\operatorname{Aut}\mathbf{G}$ is always trivial. 

\end{rmk}
Let $o_\Gamma$ be the canonical relative orientation of $\mathcal W_\Gamma$ over $\oQMb_\Gamma$, we define the canonical relative orientation $o_\mathbf{G}$ of $\mathcal W_\mathbf{G}$ over $\oQMb_\mathbf{G}$ by
\begin{equation}\label{eq orientation point insertion}
    o_\mathbf{G}:=(-1)^{\lvert E(\mathbf{G})\rvert}\bigwedge_{\Gamma\in V(\mathbf{G})}o_\Gamma;
\end{equation}
Observe that $o_\mathbf{G}$ is independent of the order of the wedge product since for each $\Gamma\in V(\mathbf{G})$ we have 
$$
\dim \Mbar_\Gamma \equiv \operatorname{rank} W_\Gamma \mod 2.
$$
\begin{dfn} Given an integer $\h\in \{0,1,\dots,\lfloor\frac{r-2}{2}\rfloor\}$, a finite set $I$ of internal twists lying in $\{0,1,\dots,r-1\}$ and a finite set $B$ of boundary twists lying in $\{r-2-2\h,r-2\h,\dots,r-4,r-2\}$, we define the (unglued) moduli space $\Mbar^{\frac{1}{r},\h}_{g,B,I}$ of genus-$g$  $(r,\h)$-surfaces labelled by $B,I$  to be
\begin{equation}\label{eq def moduli point insertion}
    \Mbar^{\frac{1}{r},\h}_{g,B,I}:=\bigsqcup_{\substack{\mathbf{G}\in \GPI^{r,\h}_g, \mathbf{G} \text{ smooth}\\I(\mathbf{G})=I,B(\mathbf{G})=B}}\Mbar_\mathbf{G}.
\end{equation}
In the case $g=0,1$, we also define
$$\oQMb^{\frac{1}{r},\h}_{g,B,I}:=\bigsqcup_{\substack{\mathbf{G}\in \GPI^{r,\h}_g, \mathbf{G} \text{ smooth}\\I(\mathbf{G})=I,B(\mathbf{G})=B}}\oQMb_\mathbf{G};$$
the Witten bundles with relative orientations over the connected components of $\oQMb^{\frac{1}{r},\h}_{g,B,I}$ induce the Witten bundle $\mathcal W^{\frac{1}{r},\h}_{g,B,I}$ over $\oQMb^{\frac{1}{r},\h}_{g,B,I}$ with relative orientation.
\end{dfn}

\subsubsection*{Relative cotangent line bundles}
For $i\in I$, we denote by $\mathbb L_i\to \prod_{\Gamma\in V(\mathbf{G})} \Mbar_\Gamma$ the line bundle pulled back from $\mathbb L_i \to \Mbar_{\Gamma_i}$ via the projection, where $\Gamma_i$ is the unique graded $r$-spin graph in $V(\mathbf{G})$ such that $i\in H^I(\Gamma_i)$. Since the action of $\operatorname{Aut}\mathbf{G}$ on $\prod_{\Gamma\in V(\mathbf{G})} \Mbar_\Gamma$ can be naturally lifted to $\mathbb L_i$, we have a relative cotangent line bundle 
$
\mathbb L_i \to \Mbar_\mathbf{G}
$
on the quotient space $\Mbar_\mathbf{G}=\left(\prod_{\Gamma\in V(\mathbf{G})} \Mbar_\Gamma\right)\big\slash \operatorname{Aut} \mathbf{G}$, and therefore a relative cotangent line bundle 
$\mathbb L_i \to \Mbar^{\frac{1}{r},\h}_{g,B,I}.$

In the case $\h=0$, there is also a modified relative cotangent line bundle 
$\check{\mathbb L}_i \to \Mbar^{\frac{1}{r},\h=0}_{g,B,I}.$
Let $\Gamma'_i$ be the graded $r$-spin graph obtained from $\Gamma_i$ after forgetting all the half-edges in $I^{PI}$ (which have twist zero), and let $\operatorname{For}_{I^{PI}}\colon \Mbar_{\Gamma_i}\to \Mbar_{\Gamma'_i}$ be the forgetful morphism. Then we denote by 
$\check{\mathbb L}_i:=\operatorname{For}^*_{I^{PI}}\mathbb L_i\to \Mbar_{\Gamma_i}$ the line bundle pulled back from $\mathbb L_i\to  \Mbar_{\Gamma'_i}$ via this forgetful morphism, and by
$\check{\mathbb L}_i\to \prod_{\Gamma\in V(\mathbf{G})} \Mbar_\Gamma$ the line bundle pulled back via projection. The $\operatorname{Aut} \mathbf{G}$-action also lift naturally to $\check{\mathbb L}_i$, thus we have $\check{\mathbb L}_i \to \Mbar_\mathbf{G}$ and $\check{\mathbb L}_i \to \Mbar^{\frac{1}{r},\h=0}_{g,B,I}.$

\subsection{Boundary Strata}
For  an $(r,\h)$-graph $\mathbf{G}$, we write
$$
E(\mathbf{G}):=\bigsqcup_{\Gamma\in V(\mathbf{G})} E(\Gamma)
$$
and 
$$
H^{CB}(\mathbf{G}):= \bigsqcup_{\Gamma\in V(\mathbf{G})}  H^{CB}(\Gamma).
$$
For a set 
$S\subseteq E(\mathbf{G})\sqcup H^{CB}(\mathbf{G})$, 
one can perform a sequence of smoothings, and the graph obtained is independent of the order in which those smoothings are performed;  denote the result by $d_S \mathbf{G}$.  Let
\begin{align*}
&\partial^!\mathbf{G} = \{\mathbf{H} \; | \; \mathbf{G} = d_S\mathbf{H} \text{ for some } S\},\\
&\partial \mathbf{G} = {\partial}^!\mathbf{G} \setminus \{\mathbf{G}\},\\
&\partial^B \mathbf{G} = \{\mathbf{H} \in {\partial} \mathbf{G} \; | \; E^B(\mathbf{H}) \cup G^{CB}(\mathbf{H}) \neq \emptyset\}.
\end{align*}
For a stable graded $r$-spin graph $\Gamma$, we can also define $\partial^!{\Gamma}$, $\partial \Gamma$ and $\partial^B{\Gamma}$ in the same way.

For a smooth $(r,\h)$-graph $\mathbf{G}$, a boundary stratum of $\Mbar_{\bm{G}}$ corresponds to a graph in $\partial^!\mathbf{G}$, or more precisely, a choice of $\Delta_i\in \partial^!\Gamma_i$ for each $\Gamma_i\in V(\mathbf{G})$. In particular,  a codimension-$1$ boundary of $\Mbar_\mathbf{G}$ is determined by a choice of $\Gamma\in V(\mathbf{G})$ and a graph $\Delta \in \partial^B \Gamma$, where $\Delta$ has either one contracted boundary tail and no edges, or exactly one edge which is a boundary edge. There are five different types of codimension-$1$ boundaries of $\Mbar_\mathbf{G}$ depending on the type of the (half-)edge of $\Delta$:

\begin{enumerate}
    \item[{CB}] contracted boundary tails;
    \item[{R}] Ramond boundary edges;
    \item[{NS+}] NS boundary edges whose twist on the illegal side is greater than $2\h$;
    \item[AI] NS boundary edges whose twist on the illegal side is less than or equal to $2\h$, and the vertex containing the legal half-node only contains this half-edge and an internal tail $a\in I(\Gamma)\cap I^{PI}(\mathbf{G})$;
    \item[BI] the remaining NS boundary edges whose twist on the illegal side is less than or equal to $2\h$.
\end{enumerate}

Therefore, the codimension-$1$ boundary of $\Mbar^{\frac{1}{r},\h}_{g,B,I}$ is a union of five different types of boundaries. We claim that there is a one-to-one correspondence between the AI boundaries and the BI boundaries.

 \begin{thm}\label{thm: PI boundaries paried}
 For fixed $I$ and $B$, there is a one-to-one correspondence $\PI$ between the BI boundaries and the AI boundaries of $\Mbar^{\frac{1}{r},\h}_{g,B,I}$. Two boundaries paired by the correspondence $\PI$ are canonically diffeomorphic. 
 
 In the $g=0$ case, this diffeomorphism can be extended to a diffeomorphism between the closures of these boundaries in $\Mbar^{\frac{1}{r},\h}_{g,B,I}$, which can be further lifted to the Witten bundles and the relative cotangent line bundles restricted to closures. Moreover, the relative orientations on the paired boundaries induced by the canonical relative orientations are opposite to each other.

    In the $g=1$ case, the diffeomorphism between a type-AI boundary $\text{bd}_{AI}$ and a type-BI boundary $\text{bd}_{BI}$ can be extend to a surjective morphism from the closure $\overline{\text{bd}}_{AI}$ to $\overline{\text{bd}}_{BI}$, which sends  $\overline{\text{bd}}_{AI}\cap \oQMb^{\frac{1}{r},\h}_{1,B,I}$ surjectively onto $\overline{\text{bd}}_{BI}\cap \oQMb^{\frac{1}{r},\h}_{1,B,I}$; and similar to the $g=0$ case, this restricted surjective morphism pulls back the Witten bundles and the relative cotangent line bundles restricted to them. Moreover, the relative orientations on the paired boundaries induced by the canonical relative orientations are opposite to each other. 
   
 \end{thm}

\begin{proof}
 Let $\Mbar_{\mathbf{G_C}}$ be a connected component of $\Mbar^{\frac{1}{r},\h}_{g,B,I}$ and $\text{bd}_{BI}$ be a codimension-$1$ BI boundary of $\Mbar_{\mathbf{G_C}}$ given by $\Gamma_{C_1}\in V(\mathbf{G_C})$ and $\Gamma_{C_1^{BI}}\in {\partial}^!\Gamma_{C_1}$ as on the left-hand side of Figure \ref{fig rh surface}, we construct an $(r,\h)$-graph $\mathbf{G_D}$ such that $\Mbar_{\mathbf{G_D}}$ has an BI boundary $\text{bd}_{AI}$ which is canonically diffeomorphic to $\text{bd}_{BI}$.
 
 We first consider the case where the BI boundary edge $e$ of $\Gamma_{C_1^{BI}}$ is separating.
Let $v_1,v_2$ and $e$ be the vertices of $\Gamma_{C_1^{BI}}$ connected by $e$, we denote by $h_1,h_2$ the corresponding half-edges of $e$ on $v_1,v_2$ and assume that $h_1$ is illegal and $h_1$ is legal. 

Since $e$ is a BI edge we have $\text{tw}(h_1)\le 2\h$ and hence $\text{tw}(h_2)\ge r-2-2\h$, therefore we can regard the vertex $v_2$ as a legal level-$\h$ stable graded $r$-spin graph, we denote it by $\Gamma_{D_2}$. We define a new legal level-$\h$ stable graded $r$-spin graph $\Gamma_{D_1}$ by removing $h_1$ in $v_1$ and adding a new internal tail $\hat h_1$ with $\tw(\hat h_1)=\tw(h_1)/2$. We define a new $(r,\h)$-graph $\mathbf{G_D}$ by replacing $\Gamma_{C_1}$ in $V(\mathbf{G_C})$ with $\Gamma_{D_1},\Gamma_{D_2}$ and adding a new dashed line $(\hat h_1,h_2)$. By construction, we have $\Mbar_{\mathbf{G_D}}\subseteq \Mbar^{\frac{1}{r},\h}_{g,B,I}$. 

Let $\Gamma_{D_1^{AI}}\in {\partial}^!\Gamma_{D_1}$ be the graded $r$-spin graph with two vertices $u_1$ and $u_2$ connected by a boundary $e'$, where $u_1$ is the same as $v_1$, and the only half-edges on $u_2$ are one internal tail $\hat h_1$ and one boundary half-edge $h'_2$. Since $$\tw(\hat h_1)=\frac{\tw(h_1)}{2}\le \h \le \frac{r-2}{2},$$
by \eqref{eq legal mod 2} we have
$\alt(h'_2)=1$ and $\tw(h'_2)=r-2-2\tw(\hat h_1)\ge r-2-2\h$, therefore $e'$ is an AI boundary edge and $\Gamma_{D_1^{BI}}$ determine an AI boundary $\text{bd}_{AI}$ of $\Mbar_{\mathbf{G_D}}$.

The above process is reversible. If $\Gamma_{D_1^{AI}}\in {\partial}^!\Gamma_{D_1}$ determines an AI boundary of $\mathbf{G_D}$ as above, then we take $\Gamma_{D_2}$ to be the vertex of $\mathbf{G_D}$ containing the boundary half-edge $\delta(\hat h_1)$ ($\delta$ is the map induced by the dashed lines as in Definition \ref{dfn rh graph}). We can obtain $\Gamma_{C_1^{BI}}$ by replacing $u_2$ with $\Gamma_{D_2}$ in $\Gamma_{D_1^{AI}}$, then we smooth it to get $\Gamma_{C_1}$. By replacing $\Gamma_{D_1},\Gamma_{D_1}\in V(\mathbf{G_D})$ with $\Gamma_{C_1}$, we get $\mathbf{G_C}$.

For a vertex $v$ in the stable graded $r$-spin dual graph, we denote by $\Mstar_v$ the moduli of $r$-spin surfaces (possibly with illegal boundary markings) corresponding to $v$. Since $e$ corresponds to a boundary NS node, by Remark \ref{rmk decompose NS boundary node}
we have a diffeomorphism
\begin{equation*}
\begin{split}
\text{bd}_{BI}&=\M_{\Gamma_{C_1^{BI}}}\times\prod_{\Gamma_C\in V(\mathbf{G_C})\backslash \{\Gamma_{C_1}\}}\M_{\Gamma_C}\\
&\cong \Mstar_{v_1}\times\M_{v_2}\times\prod_{\Gamma_C\in V(\mathbf{G_C})\backslash \{\Gamma_{C_1}\}}\M_{\Gamma_C}.
\end{split}
\end{equation*}
Similarly, we have 
\begin{equation*}
\begin{split}
\text{bd}_{AI}&=\M_{\Gamma_{D_1^{AI}}}\times\M_{\Gamma_{D_2}}\times\prod_{\Gamma_{D}\in V(\mathbf{G_D})\backslash \{\Gamma_{D_1},\Gamma_{D_2}\}}\M_{\Gamma_D}\\
&\cong \Mstar_{u_1}\times \M_{u_2}\times\M_{\Gamma_{D_2}}\times\prod_{\Gamma_D\in V(\mathbf{G_D})\backslash \{\Gamma_{D_1},\Gamma_{D_2}\}}\M_{\Gamma_D}.
\end{split}
\end{equation*}
Note that by construction we have $v_1=u_1$, $v_2=\Gamma_{D_2}$ and $V(\mathbf{G_C})\backslash \{\Gamma_{C_1}\}=V(\mathbf{G_D})\backslash \{\Gamma_{D_1},\Gamma_{D_2}\}$. Since $\M_{u_2}$ is a single point we have a natural diffeomorphism
\begin{equation}\label{eq iso bdry}
\varphi^{PI}\colon\text{bd}_{AI} \xrightarrow{\sim} \text{bd}_{BI}.
\end{equation}
In this case with separating node we can prove the diffeomorphism between the closures
\begin{equation}\label{eq iso bdry closure}
\bar{\varphi}^{PI}_{sp}\colon\overline{\text{bd}}_{AI} \xrightarrow{\sim} \overline{\text{bd}}_{BI}.
\end{equation}
in the same way as we have $\Mbar_{\Gamma_{D_1^{AI}}}\cong \Mbarstar_{v_1}\times\Mbar_{v_2}\cong \Mbar_{\Gamma_{D_1^{AI}}}\times\Mbar_{\Gamma_{D_2}}$.

In the case where the BI boundary edge $e$ of $\Gamma_{C_1^{BI}}$ is non-separating, we still denote by $v_1$ the vertex containing the illegal half-edge $h_1$ after detaching $e$. We can define $\Gamma_{D_1}$ by replacing $h_1$ with a new internal tail $\hat h_1$ in the same way as in the previous case, as well as $\Gamma_{D_1^{AI}}\in {\partial}^!\Gamma_{D_1}$. Note that in this case, the legal half-edge $h_2$ is also contained within $\Gamma_{D_1}$, instead of in another graded $r$-spin graph $\Gamma_{D_2}$. The $(r,\h)$-graph $\mathbf{G_D}$ in this case is obtained by replacing $\Gamma_{C_1}\in V(\mathbf{G_C})$ with $\Gamma_{D_1}$ and adding a new dashed line $(\hat h_1,h_2)$. According to Remark \ref{rmk decompose NS boundary node} we have $\M_{\Gamma_{C_1^{BI}}}\cong \M_{\Gamma_{D_1^{AI}}}$, which indicates we also have a diffeomorphism \eqref{eq iso bdry} in this case. However, note that \eqref{eq iso bdry closure} does not hold in this case with non-separating node since $\Mbar_{\Gamma_{D_1^{AI}}} \ne \Mbar_{\Gamma_{C_1^{BI}}}$ in general: we only have a surjective morphism sending $\Mbar_{\Gamma_{D_1^{AI}}}$ to $\Mbar_{\Gamma_{C_1^{BI}}}$, and thus a surjective morphism
\begin{equation}\label{eq surj between bds}
    \bar{\varphi}^{PI}_{nsp}\colon  \overline{\text{bd}}_{AI} \rightarrow \hspace{-8pt} \rightarrow\overline{\text{bd}}_{BI}.
\end{equation}

Note that in the $g=0$ case, we do not have non-separating edges so the diffeomorphisms extend to the closures for all paired boundaries. According to Remark \ref{rmk decompose NS boundary node}, the corresponding Witten bundles and the relative cotangent line bundles are also isomorphic.
By Corollary \ref{cor:orientation open open}, the  relative orientation (of the Witten bundle) on  $$\M_{\Gamma_{C_1^{BI}}}=\Mstar_{v_1}\times\M_{v_2}=\Mstar_{v_1}\times\M_{\Gamma_{D_2}}$$ induced by the canonical relative orientation $o_{\Gamma_{C_1}}$ on $\M_{\Gamma_{C_1}}$ is $o_{v_1}^{h_1}\boxtimes o_{\Gamma_{D_2}}$, and the relative orientation on $$\M_{\Gamma_{D_1^{AI}}}=\Mstar_{u_1}\times \M_{u_2}\times\M_{\Gamma_{D_2}}=\Mstar_{v_1}\times\M_{\Gamma_{D_2}}$$ induced by the  relative orientation $o_{\Gamma_{D_1}}\boxtimes o_{\Gamma_{D_2}}$  on $\M_{\Gamma_{D_1}}\times \M_{\Gamma_{D_2}}$ is also $o_{v_1}^{h_1}\boxtimes o_{\Gamma_{D_2}}$. Then the moreover part about relative orientation follows from \eqref{eq orientation point insertion} and the fact that $\vert E(\mathbf{G_D})\vert=\vert E(\mathbf{G_C})\vert +1$.

The proof in the $g=1$ case is identical to the $g=0$ case after showing that morphism \eqref{eq iso bdry closure} or \eqref{eq surj between bds} sends $\overline{\text{bd}}_{AI}\cap \oQMb^{\frac{1}{r},\h}_{1,B,I}$ to $\overline{\text{bd}}_{BI}\cap \oQMb^{\frac{1}{r},\h}_{1,B,I}$ surjectively. Actually, since $g=1$, all except at most one graphs in $\Gamma_C\in V(\mathbf{G_C})$ are genus-zero. There are three possibilities.
\begin{itemize}
    \item 
    If all of graphs in $V(\mathbf{G_C})$ are genus-zero, then both $\overline{\text{bd}}_{BI}$ and $\overline{\text{bd}}_{AI}$ are entirely contained in $ \oQMb^{\frac{1}{r},\h}_{1,B,I}$, hence $\overline{\text{bd}}_{BI}\cap \oQMb^{\frac{1}{r},\h}_{1,B,I}=\overline{\text{bd}}_{BI}$ and $\overline{\text{bd}}_{AI}\cap \oQMb^{\frac{1}{r},\h}_{1,B,I}=\overline{\text{bd}}_{AI}$. 
    
    \item 
    If one of $\Gamma_C^{g=1}\in V(\mathbf{G_C})\backslash \{\Gamma_{C_1}\}$ is genus-one, then we have 
\begin{equation*}
\overline{\text{bd}}_{BI} \cap \oQMb^{\frac{1}{r},\h}_{1,B,I}\cong \Mbarstar_{v_1}\times\Mbar_{v_2}\times \oQMb_{\Gamma_C^{g=1}}\times\prod_{\Gamma_C\in V(\mathbf{G_C})\backslash \{\Gamma_{C_1},\Gamma_C^{g=1}\}}\Mbar_{\Gamma_C},
\end{equation*}
which is diffeomorphism (via \eqref{eq iso bdry closure}) to \begin{equation*}
\overline{\text{bd}}_{AI}\cap\oQMb^{\frac{1}{r},\h}_{1,B,I}\cong \Mbarstar_{u_1}\times \Mbar_{u_2}\times\Mbar_{\Gamma_{D_2}}\times \oQMb_{\Gamma_D^{g=1}}\times\prod_{\Gamma_D\in V(\mathbf{G_D})\backslash \{\Gamma_{D_1},\Gamma_{D_2},\Gamma_D^{g=1}\}}\Mbar_{\Gamma_D},
\end{equation*}
where $\Gamma_D^{g=1}\in V(\mathbf{G_D})$ is the graph corresponding to $\Gamma_C^{g=1}\in V(\mathbf{G_C})$.
\item If the graph $\Gamma_{C_1}$ is genus-one, there are two cases.
\begin{itemize}
    \item The edge $e\in E(C^{BI}_1)$ is non-separating. In this case we have $\oJMb_{\Gamma_{C_1}}\cap \Mbar_{C^{BI}_1}=\emptyset$ since $e$ is type-AI (notice that all non-separating edges in a  graph intersecting the dimension-jump locus are Ramond), therefore $\textbf{bd}_{BI}$ and $\text{bd}_{AI}$ are entirely contained in $ \oQMb^{\frac{1}{r},\h}_{1,B,I}$, hence $\overline{\text{bd}}_{BI}\cap \oQMb^{\frac{1}{r},\h}_{1,B,I}=\overline{\text{bd}}_{BI}$ and $\overline{\text{bd}}_{AI}\cap \oQMb^{\frac{1}{r},\h}_{1,B,I}=\overline{\text{bd}}_{AI}$.
    \item The edge $e\in E(C^{BI}_1)$ is separating. We assume $v_1$ is genus-one (the case $v_2$ is genus-one is similar). We have 
    \begin{equation*}
    \overline{\text{bd}}_{BI}\cap\oQMb^{\frac{1}{r},\h}_{1,B,I}\cong \oQMb_{v_1}\times\Mbar_{v_2}\times\prod_{\Gamma_C\in V(\mathbf{G_C})\backslash \{\Gamma_{C_1}\}}\Mbar_{\Gamma_C},
    \end{equation*}
which is diffeomorphism to 
\begin{equation*}
\overline{\text{bd}}_{AI}\cap\oQMb^{\frac{1}{r},\h}_{1,B,I}\cong \oQMb_{u_1}\times \Mbar_{u_2}\times\Mbar_{\Gamma_{D_2}}\times\prod_{\Gamma_D\in V(\mathbf{G_D})\backslash \{\Gamma_{D_1},\Gamma_{D_2}\}}\Mbar_{\Gamma_D}.
\end{equation*}
\end{itemize}

\end{itemize}

The claim about the relative orientation follows from Corollary \ref{cor:orientation open open}, Corollary \ref{cor:orientation open open g=1} and Corollary \ref{cor:orientation relative nonseperating} in the same way as the $g=0$ case.

\end{proof}

\subsection{The moduli space of reduced $(r,\h)$-surfaces}
    Let $\sim_{PI}$ be the equivalent relation induced by the correspondence $\PI$ on the boundaries of $\Mbar^{\frac{1}{r},\h}_{g,B,I}$, where $p_{AI}\in \overline{\text{bd}}_{AI}$ is equivalent to $p_{BI}\in \overline{\text{bd}}_{AI}$ if $p_{BI}$ is the image of $p_{AI}$ under the morphism \eqref{eq iso bdry closure} or \eqref{eq surj between bds}. Theorem \ref{thm: PI boundaries paried} shows that we can glue $\Mbar^{\frac{1}{r},\h}_{g,B,I}$ along the paired boundaries and obtain a piecewise smooth glued moduli space
    $$
    \widetilde{\mathcal M}^{\frac{1}{r},\h}_{g,B,I}:=\Mbar^{\frac{1}{r},\h}_{g,B,I}\big/\sim_{PI}
    $$
    parametrizing the reduced genus-$g$ $(r,\h)$-surfaces  (see Definition \ref{def reduced rh surface}) whose unpaired boundary and internal tails are marked by $B$ and $I$. 
    Similarly, we can glue $\oQMb^{\frac{1}{r},\h}_{1,B,I}\subseteq\Mbar^{\frac{1}{r},\h}_{1,B,I}$ to $\widetilde{\mathcal QM}^{\frac{1}{r},\h}_{1,B,I}\subseteq \widetilde{\mathcal M}^{\frac{1}{r},\h}_{1,B,I}$.
    Note that $\widetilde{\mathcal M}^{\frac{1}{r},\h}_{g,B,I}$ or $\widetilde{\mathcal QM}^{\frac{1}{r},\h}_{1,B,I}$ has only boundaries of type CB, R and NS+.

    The Witten bundles and the relative cotangent line bundles over the different connected components of $\oQMb^{\frac{1}{r},\h}_{0,B,I}:=\Mbar^{\frac{1}{r},\h}_{0,B,I}$ or $\oQMb^{\frac{1}{r},\h}_{1,B,I}$  can also be glued along the same boundaries, by Theorem \ref{thm: PI boundaries paried}. By the same theorem $\widetilde{\mathcal W}\to \widetilde{\mathcal QM}^{\frac{1}{r},\h}_{1,B,I}$ or $\widetilde{\mathcal QM}^{\frac{1}{r},\h}_{0,B,I}:=\widetilde{\mathcal M}^{\frac{1}{r},\h}_{0,B,I}$, the glued Witten bundle, is canonically relatively oriented.
    \begin{rmk}
In the case $r=2,\h=0$ and only NS insertions the Witten bundle is a trivial zero rank bundle. In this case the idea of gluing different moduli spaces to obtain an orbifold without boundary is due to Jake Solomon and the first named author \cite{ST_unpublished}. \cite{ST_unpublished} worked with a different definition of the glued cotangent lines, and related that construction to the construction of \cite{PST14}.
\end{rmk}

    The direct sum $E\oplus E$ of any two copies of a vector bundle $E\to\widetilde{\mathcal QM}^{\frac{1}{r},\h}_{g,B,I}$ carries a canonical orientation. Indeed, any oriented basis $(v_1,\ldots,v_{\text{rk}(E)})$ for a fibre $E_p,$ induces an oriented basis
    \[
    (v^{(1)}_1,\ldots,v^{(1)}_{\text{rk}(E)},
    v^{(2)}_1,\ldots,v^{(2)}_{\text{rk}(E)})
    \] for $(E\oplus E)_p$
    where $v_j^{(i)}$ is the copy of $v_j$ in the $i$th summand, $i=1,2,~j=1,\ldots,\text{rk}(E).$ We will orient that fibre $(E\oplus E)_p$ via this basis.
    If $v'_1,\ldots,v'_{\text{rk}(E)}$ is another basis for $E_p,$ and $A$ is the transition matrix between the two bases, then the transition matrix between the induced bases of $(E\oplus E)_p$ is the block matrix
    \[\begin{pmatrix}
A &0\\0 &A\end{pmatrix}.\]Its determinant is $\det(A)^2,$ which is positive. Hence this choice of orientation for $E\oplus E$ is independent of the choices of basis, and extends globally. 
As a consequence, the fibres of $(\widetilde{{\mathcal{W}}})^{2d}$ are canonically oriented for every natural $d.$
    
    The glued relative cotangent line bundles over $\widetilde{\mathcal M}^{\frac{1}{r},\h}_{g,B,I}$ still carry the canonical complex orientations.

    Combining all the above, we obtain the following theorem:
    \begin{thm}\label{thm:main_or}
      For $g=0,1$, all bundles of the form
    \[(\widetilde{{\mathcal{W}}})^{2d+1}\oplus\bigoplus_{i=1}^l\mathbb{L}_i^{\oplus d_i}\to \widetilde{\mathcal QM}^{\frac{1}{r},\h}_{g,B,I}\]are canonically relatively oriented.  
    \end{thm}

\bibliographystyle{abbrv}
\bibliography{OpenBiblio}

\end{document}